\documentclass[11pt,english,a4paper]{smfart}

\usepackage[francais,english]{babel}
\usepackage[utf8, latin1]{inputenc}

\usepackage{graphicx}
\usepackage{ae,amsfonts,euscript,enumerate}
\usepackage[cm]{aeguill}

\renewcommand{\theequation}{\thesection.\arabic{equation}}

\newtheorem{thm}{Theorem}[section]
\newtheorem{exam}[thm]{Example}
\newtheorem{defn}[thm]{Definition}

\newtheorem{prop}[thm]{Proposition}
\newtheorem{lem}[thm]{Lemma}
\newtheorem{cor}[thm]{Corollary}

\newtheorem{rem}[thm]{Remark}

\numberwithin{equation}{section}
\usepackage{fullpage}
\usepackage[T1]{fontenc}
\usepackage{graphicx}
\usepackage{amsmath}
\usepackage{amsfonts}
\usepackage{amssymb}
\usepackage{a4wide}

\author{B. Adamczewski}
\address{
CNRS, Universit\'e de Lyon, Universit\'e Lyon 1\\
Institut Camille Jordan  \\
43 boulevard du 11 novembre 1918 \\
69622 Villeurbanne Cedex, France}
\email{Boris.Adamczewski@math.cnrs.fr}

\author{Jason P.~Bell}
\address{
Department of Pure Mathematics\\
University of Waterloo\\
Waterloo, ON, Canada\\
 N2L 3G1}
\email{jpbell@uwaterloo.ca}

\author{E. Delaygue}
\address{
Universit\'e de Lyon, Universit\'e Lyon 1\\
Institut Camille Jordan  \\
43 boulevard du 11 novembre 1918 \\
69622 Villeurbanne Cedex, France}
\email{delaygue@math.univ-lyon1.fr}
\title{Algebraic independence of $G$-functions and congruences ``\`a la Lucas" }
\date{}

\thanks{This project has received funding from the European Research Council (ERC) under the European Union's Horizon 2020 research and innovation programme 
under the Grant Agreement No 648132. }

\begin{document}


\begin{abstract}  We develop a new method for proving algebraic independence of $G$-functions.  
Our approach rests on the following observation:  $G$-functions do not always come with a single linear differential equation,  
but also sometimes with an infinite family of linear difference equations associated with the Frobenius that are obtained by reduction modulo prime ideals. 
When these linear difference equations have order one, the coefficients of the $G$-function satisfy  congruences reminiscent of a classical theorem of Lucas 
on binomial coefficients. We use this to  
derive a Kolchin-like algebraic independence criterion.  We show the relevance of this criterion by proving, using $p$-adic tools, that many classical families of $G$-functions 
turn out to satisfy congruences ``\`a la Lucas". 
\end{abstract}
\maketitle
\tableofcontents

\section{Introduction}\label{sec: intro}

This paper is the fourth of a series started by the first two authors \cite{AB12,AB13,AB15} concerning several number theoretical problems 
involving linear difference equations, called Mahler's equations, as well as underlying structures associated with automata theory. 
We investigate here a class of analytic functions  introduced by Siegel \cite{Si} in his landmark 1929 paper under the name of $G$-functions.  
Let us recall that  $f(z) := \sum_{n=0}^{\infty} a(n) z^n$ is a $G$-function if it satisfies   
the following conditions. Its coefficients $a(n)$ are algebraic numbers and there exists a positive real number $C$ such that for every non-negative integer $n$:   
\begin{itemize}

\medskip

\item[\rm{(i)}] The moduli of all Galois conjugates of $a(n)$ are at most $C^n$.

\medskip

\item[\rm{(ii)}] There exists a sequence of positive integers $d_n<C^n$ such that $d_na_m$ is an algebraic integer for all $m$, $0\leq m\leq n$.  

\medskip

\item[\rm{(iii)}] The function $f$ satisfies a linear differential equation with coefficients in $\overline{\mathbb Q}(z)$. 
\end{itemize}
\medskip

\noindent
Their study leads to a remarkable interplay between number theory, algebraic geometry, combinatorics, and the study of linear differential equations 
(see \cite{An89,DGS,KZ,Miw06}).  

In this paper, we focus on the algebraic relations over $\overline{\mathbb C}(z)$ that may or may not exist 
between $G$-functions. In this respect, our main aim is to develop a new method for proving algebraic independence of such functions.   
Our first motivation is related to transcendence theory. A large part of the theory is actually devoted to the study of algebraic 
relations over $\overline{\mathbb Q}$ between periods\footnote{A period is a complex number whose real and imaginary parts are
 values of absolutely convergent integrals of rational fractions over domains of $\mathbb R^n$ defined by polynomial inequalities with rational coefficients. 
 Most complex numbers of interest to arithmeticians turn out to be periods.}.  
Unfortunately, this essentially remains {\it terra incognita}.  
At least conjecturally, $G$-functions may be thought of as  their functional counterpart  
  (smooth algebraic deformations of periods).  Understanding algebraic relations among $G$-functions 
  thus appears  to be a first step in this direction and, first of all,  a  
  much more tractable problem. For instance, a conjecture of 
  Kontsevich \cite{Ko99} (see also \cite{KZ}) claims that any algebraic relation between periods can be derived from the three fundamental operations associated with integration:   
  additivity, change of variables, and Stokes' formula.  It is considered completely out of reach by specialists, but recently Ayoub \cite{Ay} proved a 
  functional version of the conjecture (see also \cite{An15}). 
Despite the depth of this result, it does not help that much in deciding whether given $G$-functions are or are not algebraically independent. 
  
  A second motivation finds its source in enumerative combinatorics. Indeed, most generating series that have been studied so far by combinatorists 
  turn out to be $G$-functions.  To some extent, the nature of a generating series reflects the underlying structure of the objects it counts  
  (see \cite{BM06}).  By nature, we mean whether the generating series is  rational, algebraic, or $D$-finite. 
 In the same line,  algebraic independence of generating series can be considered as a reasonable way to  
   measure how distinct families of combinatorial objects may be (un)related.    
  Though combinatorists have a very long tradition of proving transcendence of  
  generating functions, it seems that algebraic independence has never been studied so far in this setting.

\medskip

Our approach rests on the following observation: a $G$-function often comes with not just a single linear differential equation,  
but also sometimes with an infinite family of linear difference equations obtained by reduction modulo prime ideals. Let us formalize this claim somewhat. 
 Let $K$ be a number field, $f(z):=\sum_{n=0}^{\infty}a(n)z^{n}$ be a $G$-function in $K\{z\}$, and let us denote by $\mathcal O_K$  the ring of integers of $K$.  
For prime ideals $\mathfrak p$ of $\mathcal O_K$  such that all coefficients of $f$ belong to the localization of $\mathcal O_K$ at $\mathfrak p$, it makes sense to 
consider the reduction of $f$ modulo $\mathfrak p$: 
$$
f_{|\mathfrak p}(z):=\sum_{n=0}^{\infty}\big(a(n)\bmod \mathfrak p\big)z^{n}\in 
(\mathcal O_K/\mathfrak p)[[z]] \, .
$$
When $\mathfrak p$ is above the prime $p$, the residue field $\mathcal O_K/\mathfrak p$ is a finite field of characteristic $p$, and the linear difference equation 
mentioned above is of the form: 
\begin{equation}\label{eq: lde}
a_0(z)f_{\vert \mathfrak p}(z) +a_1(z)f_{\vert \mathfrak p}(z^p) + \cdots + a_d(z)f_{\vert \mathfrak p}(z^{p^d}) = 0 \, ,
\end{equation}
where $a_i(z)$ belong to $(\mathcal O_K/\mathfrak p)(z)$. That is, a linear difference equation associated with the Frobenius endomorphism 
$\sigma_p: z\mapsto z^p$.  Note that $f_{|\mathfrak p}$ satisfies an equation of the form \eqref{eq: lde} if, and only if, it is algebraic 
over $(\mathcal O_K/\mathfrak p)(z)$. 
A theorem of Furstenberg \cite{Fu67} and Deligne \cite{De} shows that this holds true for all diagonals of multivariate algebraic power series 
and almost every prime ideals\footnote{Diagonals of algebraic power series form a distinguished class of $G$-functions (see for instance \cite{AB13,Christol,Ch86}).}.    
Furthermore, classical conjectures of Bombieri and Dwork 
would imply that this should also be the case for any globally bounded $G$-functions (see \cite{Ch88}). Note that even when a $G$-function is not globally bounded, 
but can still be reduced modulo $\mathfrak p$ 
for infinitely many prime ideals $\mathfrak p$, a similar situation may be expected. For instance, the hypergeometric function  
$$
{}_{2}F_{1}\left[\begin{array}{cc}(1/2),(1/2)\\ (2/3)\end{array}; z\right] 
$$
is not globally bounded but satisfies 
a relation of the form \eqref{eq: lde} for all prime numbers congruent to $1$ modulo $6$ (see Section \ref{sec: hyp}). 

A case of specific interest is when $f_{|\mathfrak p}$ satisfies a linear difference equation of order one with respect to a power of the Frobenius. 
Then one obtains a simpler equation of the form:
\begin{equation}\label{eq: lde2}
f_{\vert \mathfrak p}(z) = a(z)f_{\vert \mathfrak p}(z^{p^k})  \, ,
\end{equation}
for some positive integer $k$ and some rational fraction $a(z)$ in $(\mathcal O_K/\mathfrak p)(z)$.  
As explained in Section \ref{sec: ldks}, these equations lead to congruences for the coefficients of $f$ that are reminiscent to a classical theorem of Lucas \cite{Lu} on 
binomial coefficients and the so-called $p$-Lucas congruences. At first glance, it may seem somewhat miraculous that a $G$-function could satisfy such 
congruences for infinitely many prime ideals. Surprisingly enough, we will show that this situation occurs remarkably often. For instance, motivated by the search of differential 
operators associated with particular families of Calabi-Yau varieties, Almkvist {\it et al.} \cite{AESZ} gave a list of more than 400 differential operators selected as potential candidates. 
They all have a unique solution 
analytic at the origin and it turns out that more than fifty percent of these analytic solutions  do satisfy Lucas-type congruences (see the discussion in Section~\ref{section: tables}).  
Another interesting example is due to Samol and van Straten \cite{SS}. Consider a Laurent polynomial 
$$
\Lambda(\mathbf{x})=\sum_{i=1}^k\alpha_i\mathbf{x}^{\mathbf{a}_i}\in \mathbb Z[x_1^\pm,\dots,x_d^\pm],
$$
where $\mathbf{a}_i\in\mathbb{Z}^d$ and $\alpha_i\neq 0$ for $i$ in $\{1,\dots,k\}$. Then the Newton polyhedron of $\Lambda$ is the convex hull 
of $\{\mathbf{a}_1,\dots,\mathbf{a}_k\}$ in $\mathbb{R}^d$. In \cite{SS}, it is proved that 
if $\Lambda(\mathbf{x})$ is a Laurent polynomial in $\mathbb{Z}[x_1^\pm,\dots,x_d^\pm]$ such that the origin is the only interior integral point of the Newton polyhedron of $\Lambda$, 
then the sequence of the constant terms of its powers $([\Lambda(\mathbf{x})^n]_{\mathbf{0}})_{n\geq 0}$ has the $p$-Lucas property for all primes $p$.  
More generally, there is a long tradition (and a corresponding extensive literature) in proving that 
some sequences of natural numbers satisfy the $p$-Lucas property or some related congruences.  
Most classical sequences which are known to enjoy the $p$-Lucas property turn out to be multisums of products 
of binomial coefficients such as, for example, 
the Ap\'ery numbers
$$
\sum_{k=0}^n\binom{n}{k}^2\binom{n+k}{k}\quad\textup{and}\quad\sum_{k=0}^n\binom{n}{k}^2\binom{n+k}{k}^2.
$$ 
Other classical examples are ratios of factorials such as 
$$
\frac{(3n)!}{n!^3}\quad\textup{and}\quad\frac{(10n)!}{(5n)!(3n)!n!^2} \, \cdot
$$
However, the known proofs are quite different and strongly depend on the particular forms of the binomial coefficients 
and of the number of sums involved in those sequences. We note that some attempts to obtain more systematic results can be found in \cite{McIntosh2} and 
more recently in \cite{MS}. We also refer the reader to  \cite{Mestrovic} for a recent survey, including  many references, about $p$-Lucas congruences. 
In Section \ref{sec: mgh}, we provide a way to unify many proofs, as well as to obtain a lot of new examples. 
We introduce a family of multivariate hypergeometric series and study their $p$-adic properties. 
Using specializations of these series, we are able to prove in Section \ref{sec: appli}  
that a large variety of classical families of $G$-functions actually satisfy such congruences.  
This includes families of hypergeometric series and generating series associated with multisums of products of binomial coefficients. 
This also includes more exotic examples such as 
$$
\sum_{n=0}^{\infty}\left(\underset{k\equiv n\mod 2}{\sum_{k=0}^{\lfloor n/3\rfloor}}2^k3^{\frac{n-3k}{2}}\binom{n}{k}\binom{n-k}{\frac{n-k}{2}}\binom{\frac{n-k}{2}}{k}\right)z^n \, .
$$

Let $f_1(z),\ldots,f_s(z)$ be power series such that \eqref{eq: lde2} holds for all elements in an infinite 
set $S$ of prime ideals and for rational fractions $a(z)$ whose height  is in $O(p^k)$\footnote{See Section \ref{sec: not} for a precise definition of the height 
of a rational fraction and Section \ref{Sec: ldks} for precise conditions imposed upon $f_1(z),\ldots,f_s(z)$.}.  
In Section \ref{sec: kolchin} we prove an algebraic independence criterion, Theorem \ref{thm: ind}, saying 
that  $f_1(z),\ldots,f_n(z)$ are algebraically dependent over $\mathbb C(z)$ if, and only if, there exist integers $a_1,\ldots,a_s$, not all zero, such that: 
$$
f_1(z)^{a_1}\cdots f_s(z)^{a_s} \in \overline{\mathbb Q}(z) \, .
$$
Thus if $f_1(z),\ldots,f_n(z)$ are algebraically dependent, they  
should satisfy a very special kind of relation: a Laurent monomial is equal to a rational fraction. 
This kind of result is usually attached to the name of Kolchin.  
At this point, 
it is often possible to apply asymptotic techniques and analysis of singularities, as described in Section \ref{sec: sing}, to easily deduce a contradiction 
and finally prove that the functions $f_1(z),\ldots,f_s(z)$ are algebraically independent over $\mathbb C(z)$.  

\begin{rem}\emph{
Since $G$-functions do satisfy linear differential equations,  
differential Galois theory provides a natural framework to look at these 
questions. For instance, it leads to strong results concerning hypergeometric functions \cite{BH}. 
However, the major drawback of this approach is that 
 things become increasingly tricky when working with differential equations of higher orders.  
 Given some $G$-functions $f_1(z),\ldots,f_s(z)$, it may be non-trivial to determine the  differential Galois group associated with a 
 differential operator annihilating these functions.  
The method developed in this paper follows a totally different road. 
An important feature is that, contrary to what would happen using differential Galois theory, 
we do not have to care about the derivatives of the functions $f_1(z),\ldots,f_s(z)$.  
It is also worth mentioning that in order to apply Theorem \ref{thm: ind}, we do not even need that the functions $f_1(z),\ldots,f_s(z)$ satisfy linear differential equations. 
We only need properties about their reduction modulo prime ideals.  As described in Sections \ref{sec: ldks} and \ref{sec: kolchin}, 
all of this makes perfect sense in the general framework of a Dedekind domain. However, we only focus in this paper on applications of our method to $G$-functions.}
\end{rem}

The present work was initiated with the following concrete example.   
Given a positive integer $r$,  the function 
$$
\mathfrak f_r(z):= \sum_{n= 0}^{\infty} {2n \choose n}^r z^n
$$ 
is a $G$-function annihilated by the differential operator $\mathcal L_r:= \theta^r - 4^rz(\theta+1/2)^r$, where $\theta = z\frac{d}{dz}$. 
In 1980, Stanley \cite{St80} conjectured that the $\mathfrak f_r$'s  
are transcendental over $\mathbb C(z)$ unless for $r=1$, in which case we have 
$\mathfrak f_1(z)=(\sqrt{ 1 - 4z})^{-1}$.
 He also proved the transcendence in the case where 
 $r$ is even. The conjecture was solved independently by Flajolet \cite{Fl} and by Sharif and Woodcock \cite{SW2} with totally different methods. 
 Incidentally, this result is also a consequence of work of Beukers and Heckman \cite{BH} concerning generalized hypergeometric series. 
 Let us briefly describe these different proofs. We assume in the sequel that $r>1$. 
 
 \begin{itemize}
 \medskip
 
 \item[(i)] The proof of Flajolet is based on asymptotics. Indeed, it is known that for an algebraic function 
$f(z)=\sum_{n=0}^\infty a(n)z^n\in\mathbb{Q}\{z\}$, one has: 
$$
a(n)=\frac{\alpha^nn^s}{\Gamma(s+1)}\sum_{i=0}^mC_i\omega_i^n+\underset{n\rightarrow\infty}{O}(\alpha^nn^t),
$$
where $s\in\mathbb{Q}\setminus\mathbb{Z}_{<0}$, $t<s$, $\alpha$ is an algebraic number and the $C_i$'s and $\omega_i$'s are 
algebraic with $|\omega_i|=1$.  On the other hand, Stirling formula leads to the following asymptotics
$$
\binom{2n}{n}^r \underset{n\rightarrow\infty}\sim \frac{2^{(2n+1/2)r}}{(2\pi n)^{r/2}} \, \cdot
$$
A simple comparison between these two asymptotics shows that  
$\mathfrak f_r$ cannot be algebraic when $r$ is even, as already observed by Stanley in \cite{St80}.  
Flajolet \cite{Fl} shows that it also leads to the same conclusion for odd $r$, but then it requires the transcendence of $\pi$. 
 
 \medskip 
 
 \item[(ii)]  The proof of 
Sharif and Woodcock is based on the Lucas theorem previously mentioned. Indeed, Lucas' theorem on binomial coefficients 
implies that 
 $$
 {2(np+m) \choose (np+m)}^r \equiv {2n\choose n}^r{2m\choose m}^r \bmod p 
 $$ 
 for all prime numbers $p$, all non-negative integers $n$ and all $m$, $0\leq m\leq p-1$. 
 This leads to the algebraic equation:
 $$
\mathfrak f_{r\vert  p}(z) = A_p(z) \mathfrak f_{r\vert  p}(z)^p
 $$ 
 where $A_p(z):= \sum_{n=0}^{p-1}\big( {2n\choose n}\bmod p\big)z^n$.  
In \cite{SW2}, Sharif and Woodcock prove that the degree of algebraicity of $\mathfrak f_{r\vert  p}$ cannot remains bounded 
when $p$ runs along the primes, which ensures the transcendence of $\mathfrak f_r$. 

\medskip

 \item[(iii)]  The proof based on the work of Beukers and Heckman used the fact that 
 $$\mathfrak f_r(z)=
{}_{r+1}F_{r}\left[\begin{array}{cc}(1/2),\ldots,(1/2)\\ 1,\dots,1\end{array}; 2^{2r}z\right]
$$ 
is a hypergeometric function. Then it is easy to see that $\mathfrak f_r$ fails the beautiful interlacing criterion proved in \cite{BH}. 
In consequence, the differential Galois groups associated with the $\mathfrak f_r$'s are all infinite and 
these functions are thus transcendental. 
 
 \medskip
 \end{itemize}

Though there are three different ways to obtain the transcendence of $\mathfrak f_r$,  
 not much was apparently known about their algebraic independence.  
Roughly, our approach can be summed up by saying that (ii) + (i) leads to algebraic independence in a rather straightforward manner, while (iii) would 
be the more usual 
method.  In this line, we will complete in Section~\ref{sec: sing} the result of \cite{AB13}, proving that the 
functions $\mathfrak f_r$ are all algebraically independent.

To give a flavor of the kind of results we can obtain, we just add the following two examples. They correspond respectively to Theorems \ref{thm: appli1} and 
\ref{thm: prop3} proved in the sequel.  
The first one concerns several families of generating series associated 
with Ap\'ery numbers, Franel numbers, and some of their generalizations. The second one involves a mix of hypergeometric series and 
generating series associated with factorial ratios and  Apery numbers.  

\begin{prop}\label{prop: appli1}
Let $\mathcal F$ be the set formed by the union of the three following sets: 
$$
\left\{\sum_{n=0}^\infty\left(\sum_{k=0}^n\binom{n}{k}^r\right)z^n\,: \,r\geq 3\right\},\quad\left\{\sum_{n=0}^\infty
\left(\sum_{k=0}^n\binom{n}{k}^r\binom{n+k}{k}^r\right)z^n\,: \,r\geq 2\right\}
$$
and
$$
\left\{\sum_{n=0}^\infty\left(\sum_{k=0}^n\binom{n}{k}^{2r}\binom{n+k}{k}^r\right)z^n\,: \,r\geq 1\right\} \, .
$$
Then all elements of $\mathcal F$ are algebraically independent over $\mathbb C(z)$.
\end{prop}

Observe that the restriction made on the parameter $r$ in each case is optimal since the functions 
$$
\sum_{n=0}^\infty\left(\sum_{k=0}^n\binom{n}{k}\right) z^n= \frac{1}{1-2z}, \quad\sum_{n=0}^\infty\left(\sum_{k=0}^n\binom{n}{k}^2\right)z^n= \frac{1}{\sqrt{1-4z}} 
$$
and 
$$
\quad\sum_{n=0}^\infty\left(\sum_{k=0}^n\binom{n}{k}\binom{n+k}{k}\right)z^n=\frac{1}{\sqrt{1-6z+z^2}}
$$ 
are all algebraic over $\mathbb Q(z)$.

\begin{prop}
The functions 
$$
f(z) :=\sum_{n=0}^\infty\frac{(4n)!}{(2n)!n!^2}z^n, \; g(z):=\sum_{n=0}^\infty\left(\sum_{k=0}^n\binom{n}{k}^2\binom{n+k}{k}^2\right)z^n, \; 
h(z) := \sum_{n=0}^\infty\frac{(1/6)_n(1/2)_n}{(2/3)_nn!}z^n
$$
and
$$
i(z):=\sum_{n=0}^\infty\frac{(1/5)_n^3}{(2/7)_nn!^2}z^n 
$$
are algebraically independent over $\mathbb{C}(z)$.  
\end{prop}

\section{Notation}\label{sec: not} 

Let us introduce some notation that will be used throughout this paper.  
Let $d$ be a positive integer. Given $d$-tuples of real numbers $\mathbf{m}=(m_1,\dots,m_d)$ and $\mathbf{n}=(n_1,\dots,n_d)$, 
 we set $\mathbf{m}+\mathbf{n}:=(m_1+n_1,\dots,m_d+n_d)$ and  
$\mathbf{m}\cdot\mathbf{n}:=m_1n_1+\cdots+m_dn_d$. If moreover  $\lambda$ is a real number, then we set 
$\lambda\mathbf{m}:=(\lambda m_1,\dots, \lambda m_d)$. 
We write $\mathbf{m}\geq\mathbf{n}$ if   we have 
$m_k\geq n_k$ for all $k$ in $\{1,\dots,d\}$. We also set $\mathbf{0}:=(0,\dots,0)$ and $\mathbf{1}:=(1,\dots,1)$. 
We let $\mathcal P$ denote the set of all prime numbers. 

\medskip

\subsubsection{Polynomials.} Given a   $d$-tuple of natural numbers $\mathbf{n}=(n_1,\dots,n_d)$ and a vector of 
indeterminates  $\mathbf{x}=(x_1,\dots,x_d)$, 
we will denote by ${\bf x}^{\bf n}$  the monomial $x_1^{n_1}\cdots x_d^{n_d}$.  The (total) degree 
of such a monomial is the non-negative integer $n_1+\cdots + n_d$.  
Given a ring $R$ and a 
polynomial $P$ in $R[{\bf x}]$, we denote by $\deg P$ the (total) degree of $P$, that is 
the maximum of the total degrees of the monomials appearing in $P$ with non-zero coefficient.   
The partial degree of $P$ with respect to the indeterminate $x_i$ 
is denoted by $\deg_{x_i}(P)$. 
Given a polynomial $P(Y)$ in $R[{\bf x}][Y]$, we define the height of $P$ as 
the maximum of the total degrees (in ${\bf x}$) of its coefficients.  

\subsubsection{Algebraic functions.} Let $K$ be a field.  We denote by $K[[{\bf x}]]$ the ring of formal power series with coefficients 
in $K$ and associated with the vector of indeterminates ${\bf x}$. We denote by $K[[{\bf x}]]^{\times}$ the group of units of $K[[{\bf x}]]$, 
that is the subset of $K[[{\bf x}]]$ formed by all power series with non-zero constant coefficients. 
We say that a power series 
$$
f({\bf x})=\sum_{{\bf n} \in \mathbb N^d} a({\bf n}) {\bf x}^{\bf n}\in K[[{\bf x}]]
$$
is algebraic if it is algebraic over the field of rational functions $K({\bf x})$, 
that is,  if there exist polynomials $A_0,\ldots, A_m$ in ${K}[{\bf x}]$, not all zero, 
such that
\begin{displaymath}
\sum_{i=0}^{m} A_i({\bf x})f({\bf x})^i \ = \ 0\, .
\end{displaymath}
Otherwise, $f$ is said to be transcendental. 
The degree of an algebraic power series $f$, denoted by $\deg f$, is defined as the degree of the minimal polynomial of 
$f$, or equivalently, as the minimum of the natural numbers $m$ for which such a relation holds. 
The (naive) height of $f$, denoted by $H(f)$,   
 is then defined as the height of the minimal polynomial of $f$, or equivalently, 
 as the minimum of the heights of the non-zero polynomials $P(Y)$ in $K[{\bf x}][Y]$ 
 that vanish at $f$. For a rational function $f$, written as $A({\bf x})/B({\bf x})$ with $A$ and $B$ two coprime polynomials, 
 then one has  $H(f)= \max (\deg A, \deg B)$.  Note that we just introduced two different notions: the degree of a polynomial and the degree 
of an algebraic function. Since polynomials are also algebraic functions we have to be careful. 
For instance, the polynomial $x^2y^3$ in $K[x,y]$ has degree $5$ but viewed as an element of $K[[x,y]]$ it is an 
algebraic power series of degree $1$.    
In the sequel, this should not be a source of confusion.

\subsubsection{Algebraic independence.} Let $f_1,\ldots,f_n$ be in $K[[{\bf x}]]$. We say that $f_1,\ldots,f_n$ are algebraically dependent 
if they are algebraically dependent over the field $K({\bf x})$, that is, if there exists a non-zero polynomial 
$P(Y_1,\ldots,Y_n)$ in $K[{\bf x}][Y_1,\ldots,Y_n]$ 
such that $P(f_1,\ldots,f_n)=0$. This is also equivalent to declaring that the field extension $K({\bf x})(f_1,\ldots,f_n)$ of $K({\bf x})$ 
has transcendence degree less than $n$.  When the degree of such a polynomial $P$ (here, the total degree with respect to $Y_1,\ldots, Y_n$) 
is at most $d$, then we say that $f_1,\ldots,f_n$ satisfy a polynomial (or an algebraic) relation of degree at most $d$. 
When there is no algebraic relation between them, the power series $f_1,\ldots,f_n$ are said to be algebraically independent 
(over $K({\bf x})$). A set or family $S$ of power series is said to 
be algebraically independent if all finite subsets of $S$ consist of algebraically independent elements.

\subsubsection{Dedekind domains.} We recall here some basic facts about Dedekind domains (see for instance \cite{Se}). 
Let $R$ be a Dedekind domain; that is, $R$ is Noetherian, integrally closed,  
and every non-zero prime ideal of $R$ is a maximal ideal. Let  $K$ denote the field of fractions of $R$. 
The localization of $R$ at a maximal ideal $\mathfrak p$ is denoted by $R_{\mathfrak p}$. 
Recall here that $R_{\mathfrak p}$ can be seen as the following subset of $K$: 
$$
R_{\mathfrak p}= \left\{a/b \,:\, a\in R, b\in R\setminus  \mathfrak p\right\}.
$$ 
Then $R_{\mathfrak p}$ is a discrete valuation ring and the residue field $R_{\mathfrak p}/\mathfrak p$ is equal to $R/\mathfrak p$. Furthermore, 
any non-zero element of $R$ belongs to at most a finite number of maximal ideals of $R$. 
In other words, given an infinite set $\mathcal S$ of maximal ideals of $R$, then one always has 
$\bigcap_{\mathfrak p\in \mathcal S} \mathfrak p =\{0\}$. This property 
implies that any non-zero element of $K$ belongs to $R_{\mathfrak p}$ for all but 
finitely many maximal ideal $\mathfrak p$ of $R$. Furthermore, we also have $\bigcap_{\mathfrak{p}\in\mathcal{S}}\mathfrak{p}R_{\mathfrak{p}}=\{0\}$. For every power series $f(\mathbf{x})=\sum_{\mathbf{n}\in\mathbb{N}^d} a(\mathbf{n})\mathbf{x}^{\mathbf{n}}$ with coefficients in $R_{\mathfrak p}$, we set
$$f_{|\mathfrak p}(\mathbf{x}):=\sum_{\mathbf{n}\in\mathbb{N}^d}\big(a(\mathbf{n})\bmod \mathfrak p\big)\mathbf{x}^{\mathbf{n}}\in 
(R/\mathfrak p)[[\mathbf{x}]] \, .$$ 
The power series $f_{|\mathfrak p}$ is called the reduction of $f$ modulo $\mathfrak p$.

\section{Lucas-type congruences and two special sets of power series}\label{sec: ldks}

\subsection{The set $\mathcal L_d(R,\mathcal{S})$}\label{Sec: ldks}

We define a special set of power series that will  play a key role in this paper.  

\medskip

\begin{defn}\label{def: ldrs} {\em Let $R$ be a Dedekind domain and $K$ be its field of fractions. Let $\mathcal{S}$ be a set of prime ideals of $R$.   
Let $d$ be a positive integer and ${\bf x} = (x_1,\ldots,x_d)$ be a vector of indeterminates.  
We let $\mathcal{L}_d(R,{\mathcal S})$ denote the set of all power series 
$f({\bf x})$ in $K[[{\bf x}]]$ with constant term equal to $1$ and such that for every $\mathfrak p$ in $\mathcal{S}$: 
\begin{itemize}

\medskip

\item[{\rm(i)}]  $f(\mathbf{x})\in R_{\mathfrak p}[[\mathbf{x}]]$;

\medskip

\item[{\rm(ii)}]  The residue field $R/\mathfrak p$ is finite and has characteristic $p$\,;

\medskip

\item[{\rm(iii)}]  There exist a positive integer $k$ 
and a rational fraction $A$ in $K({\bf x})\cap R_{\mathfrak p}[[{\bf x}]]$  satisfying
\begin{equation*}\label{eq cond}
f(\mathbf{x})\equiv A(\mathbf{x})f\big(\mathbf{x}^{p^{k}}\big)\mod \mathfrak pR_{\mathfrak p}[[\mathbf{x}]];
\end{equation*}

\medskip

\item[{\rm(iv)}]   $H(A) = O(p^{k})$, where the constant involved in the Landau notation does not depend on $\mathfrak p$. 
\end{itemize}
}
\end{defn}

\begin{rem}\label{rem: a(0)} {\rm 
If $f({\bf x})$ is a formal power series that belongs to 
$\mathcal{L}_d(R,\mathcal{S})$,  then the constant coefficient $A({\bf 0})$ of the rational fraction $A$ involved in ${\rm(iii)}$ must be equal to $1$.  
In particular, $A({\bf x})$ belongs to the group of units 
$R_{\mathfrak{p}}[[{\bf x}]]^{\times}$. 
}
\end{rem}

\begin{rem}\label{rem: itere} {\rm 
Let $f(\mathbf{x})$ be a power series in $\mathcal{L}_d(R,\mathcal{S})$. Let 
$\mathfrak p$ be a prime in $\mathcal{S}$ such that $f(\mathbf{x})\equiv A(\mathbf{x})f\big(\mathbf{x}^{p^k}\big)
\mod \mathfrak p R_{\mathfrak p}[[\mathbf{x}]]$ 
with $H(A(\mathbf{x}))\leq Cp^k$. Iterating Congruence (iii), we observe that for all natural numbers $m$, we also have 
$$
f(\mathbf{x})\equiv A(\mathbf{x})A\big(\mathbf{x}^{p^k}\big)\cdots A\big(\mathbf{x}^{p^{mk}}\big)
f\big(\mathbf{x}^{p^{(m+1)k}}\big)\mod \mathfrak p R_{\mathfrak p}[[\mathbf{x}]],
$$
with 
\begin{align*}
H\big(A(\mathbf{x})A\big(\mathbf{x}^{p^k}\big)\cdots A\big(\mathbf{x}^{p^{mk}}\big)\big)&\leq Cp^k(1+p^{k}+\cdots+p^{km})\\
&\leq Cp^k\frac{p^{(m+1)k}-1}{p^k-1}\\
&\leq 2Cp^{(m+1)k}.
\end{align*}}
\end{rem}

\begin{rem}\label{rem: Q}{\rm 
In  our applications, we will focus on the fundamental case where $K$ is a number field. In that case, $K$ is the fraction field of its ring of integers $R=\mathcal{O}_K$ which is a Dedekind domain. Furthermore, for every prime ideal $\mathfrak{p}$ in $\mathcal{O}_K$ above a rational prime $p$, the residue field $\mathcal{O}_K/\mathfrak{p}$ is finite of characteristic $p$. In particular, we will consider the case $K=\mathbb{Q}$. In that case, we have $R=\mathcal{O}_K=\mathbb{Z}$ and, for every prime number $p$, the localization $\mathbb{Z}_{(p)}$ is the set of rational numbers whose denominator is not 
divisible by $p$.  
If there is no risk of confusion,  
we will simply write $\mathcal L_d(\mathcal S)$ 
instead of $\mathcal L_d(\mathbb Z, \mathcal S)$ and $\mathcal L(\mathcal S)$ instead of $\mathcal L_1(\mathbb Z, \mathcal S)$.  
}
\end{rem}

As we will see in the sequel, a good way to prove that some power series $f$ belongs to $\mathcal L_1(R, \mathcal S)$ 
is to show that $f$ arises as some specialization of a multivariate power series known to belong to  
$\mathcal L_d(R, \mathcal S)$. In this direction, we give the following useful result.

\begin{prop}\label{prop: special}
Let $R$ be a Dedekind domain with field of fractions $K$. Let $S$ and $S'$ be two sets of prime ideals of $R$ of finite index. Let $f$ be in $\mathcal{L}_d(R, \mathcal{S})$ and $g$ in $\mathcal{L}_e(R, \mathcal{S'})$. Then the following hold. 

\begin{itemize}
\item [\rm{(i)}] Let $a_1,\dots,a_d$ be non-zero elements of $K$ and $n_1,\dots,n_d$ be positive integers. Then $f(a_1x_1^{n_1},\dots,a_dx_d^{n_d})$ belongs to $\mathcal{L}_d(R, \mathcal T)$, where $\mathcal T$ is the set of primes $\mathfrak{p}$ in $\mathcal{S}$ such that $a_1,\dots,a_d$ belong to $R_{\mathfrak p}$.

\item[\rm{(ii)}] If $x$ is an indeterminate, then $f(x,x,x_3,\dots,x_d)$ belongs to $\mathcal{L}_{d-1}(R, \mathcal{S})$. 

\item[\rm{(iii)}]  If ${\bf x}$ and ${\bf y}$ are two vectors of indeterminates, then 
$f({\bf x}) \cdot h({\bf y})$ is in $\mathcal L_{d+e}(R,\mathcal{S}\cap\mathcal{S}')$. 
\end{itemize}
\end{prop}

The proof of Proposition \ref{prop: special} is a straightforward consequence of Definition \ref{def: ldrs} 
and of Remark \ref{rem: itere}. 

\subsection{The set $\mathfrak L_d(R,\mathcal S)$ and $p^k$-Lucas congruences} 

As we will see in the sequel, it often happens that elements of $\mathcal L_d(R,\mathcal S)$ satisfy a stronger form of Condition $\rm{(iii)}$. 
Typically, the rational fraction $A({\bf x})$ can just be a polynomial with even further restriction on its degree. 
This gives rise to stronger congruences that are of interest 
in combinatorics, and leads us  to define the following distinguished subset of $\mathcal L_d(R,\mathcal S)$.

\begin{defn}
 {\em
Let us define $\mathfrak{L}_d(R,\mathcal{S})$ as the subset of $\mathcal{L}_d(R, \mathcal{S})$ formed by 
the series $f(\mathbf{x})$ for which Condition $\rm{(iii)}$ is satisfied for a fixed $k$ (\textit{i.\,e.} independent of $\mathfrak p$) 
and a polynomial $A({\bf x})\in R_{\mathfrak p}[{\bf x}]$ with 
$\deg_{x_i}(A(\mathbf{x}))\leq p^k-1$ for all $i$ in $\{1,\dots,d\}$.  
}
\end{defn}

Again,  if there is no risk of confusion,  
we will simply write $\mathfrak L_d(\mathcal S)$ 
instead of $\mathfrak L_d(\mathbb Z, \mathcal S)$ and $\mathfrak L(\mathcal S)$ instead of $\mathfrak L_1(\mathbb Z,\mathcal S)$.   

\begin{rem}\label{rem: itere2} {\rm 
Let $f(\mathbf{x})\in\mathfrak{L}_d(R,\mathcal{S})$. Let 
$\mathfrak p$ be a prime ideal in $\mathcal{S}$ such that  $f(\mathbf{x})\equiv A(\mathbf{x})f\big(\mathbf{x}^{p^k}\big)
\mod \mathfrak p \mathcal R_{\mathfrak p}[[\mathbf{x}]]$ where $A(\mathbf x)$ belongs to $R_{\mathfrak p}[\mathbf x]$ with $\deg_{x_i}(A(\mathbf{x}))\leq p^k-1$ for all $i$ in $\{1,\dots,d\}$. Iterating Condition (iii), we observe that for all natural numbers $m$, we also have 
$$
f(\mathbf{x})\equiv A(\mathbf{x})A\big(\mathbf{x}^{p^k}\big)\cdots A\big(\mathbf{x}^{p^{mk}}\big)
f\big(\mathbf{x}^{p^{(m+1)k}}\big)\mod \mathfrak p \mathcal R_{\mathfrak p}[[\mathbf{x}]],
$$
with 
\begin{align*}
\deg_{x_i}\big(A(\mathbf{x})A\big(\mathbf{x}^{p^k}\big)\cdots A\big(\mathbf{x}^{p^{mk}}\big)\big)&\leq (p^k-1)(1+p^{k}+\cdots+p^{km})\\
&\leq  p^{(m+1)k}-1.
\end{align*}}
\end{rem}

We have the following practical characterization of power series in $\mathfrak{L}_d(R,\mathcal{S})$ in terms of congruences 
satisfied by their  coefficients. 

\begin{defn}\label{def: pkl}
 {\em
We  say that the family $(a({\bf n}))_{\bf n\in\mathbb N^d}$ with values in $K^d$ 
satisfies the $p^k$-\emph{Lucas property} with respect to $\mathcal S$ if for all maximal ideal $\mathfrak p$ in $\mathcal S$, 
$(a({\bf n}))_{\bf n\in\mathbb N^d}$ takes values in $R_{\mathfrak p}$ and 
$$
a(\mathbf{v}+\mathbf{m}p^k)\equiv a(\mathbf{v})a(\mathbf{m})\mod \mathfrak pR_{\mathfrak p},
$$
for all $\mathbf{v}$ in $\{0,\dots,p^k-1\}^d$ and $\mathbf{m}$ in $\mathbb{N}^d$.   
When  $\mathcal S$ is the set of all maximal ideals of $R$, then we say that $(a({\bf n}))_{\bf n\in\mathbb N^d}$, 
which takes thus values in $R$,   
satisfies the $p^k$-Lucas property.  When $k=1$, we simply say that  $(a({\bf n}))_{\bf n\in\mathbb N^d}$ 
satisfies the $p$-\emph{Lucas property} (or the $p$-Lucas property with respect to $\mathcal S$).  
}
\end{defn}

\begin{prop}\label{prop: plucas}
A power series $f({\bf x}) :=\sum_{\bf n \in \mathbb N^d} a({\bf n}) {\bf x}^{\bf n}$ belongs to $\mathfrak L_d(R,\mathcal S)$  
if and only if there exists a positive integer $k$ such that the family $(a({\bf n}))_{\bf n\in\mathbb N^d}$ satisfies $a(\bf 0)=1$ and has the $p^k$-Lucas property with respect to 
$\mathcal S$.
\end{prop}

We will also say  that a power series $f({\bf x}) :=\sum_{\bf n \in \mathbb N^d} a({\bf n}) {\bf x}^{\bf n}$ satisfies the $p^k$-Lucas property 
with respect to $\mathcal S$ when the family $(a({\bf n}))_{\bf n\in\mathbb N^d}$ satisfies the $p^k$-Lucas property with respect to~$\mathcal S$.

\begin{proof}[Proof of Proposition \ref{prop: plucas}] 
Let $f({\bf x}) :=\sum_{\bf n \in \mathbb N^d} a({\bf n}) {\bf x}^{\bf n}$ belong to $\mathfrak L_d(R,\mathcal S)$. 
By definition,  there exists a positive integer $k$ such that, for every $\mathfrak p$ in $\mathcal S$, one has 
\begin{equation}\label{eq: fcong} 
f(\mathbf{x})\equiv A(\mathbf{x})f\big(\mathbf{x}^{p^k}\big)\mod \mathfrak p \mathcal R_{\mathfrak p}[[\mathbf{x}]],
\end{equation}
where $A(\mathbf x)$ belongs to $R_{\mathfrak p}[\mathbf x]$ and $\deg_{x_i}(A(\mathbf x))\leq p^k-1$ for all $1\leq i\leq d$. 
Then we can write
$$
A(\mathbf{x})=\sum_{\mathbf{0}\leq\mathbf{v}\leq(p^k-1)\mathbf{1}}b(\mathbf{v})\mathbf{x}^{\mathbf{v}}
$$
and thus 
$$
A(\mathbf{x})f\big(\mathbf{x}^{p^k}\big) = \sum_{\bf m \in \mathbb N^d}\sum_{\mathbf{0}\leq\mathbf{v}\leq(p^k-1)\mathbf{1}} 
b({\bf v})a({\bf m}) {\bf x}^{\bf v + \bf mp^k}.
$$
The congruence satisfied by $f$ now implies that  
\begin{equation}\label{eq: ab} 
a(\mathbf{v}+\mathbf{m}p^k)\equiv b(\mathbf{v})a(\mathbf{m})\mod \mathfrak pR_{\mathfrak p},
\end{equation} 
for all ${\bf m}$ in $\mathbb N^d$ and all ${\bf 0} \leq {\bf v}\leq (p^k-1)\bf 1$. 
Choosing ${\bf m}=\bf 0$, we obtain that $a({\bf v}) \equiv b({\bf v})\mod \mathfrak pR_{\mathfrak{p}}$ 
for all ${\bf 0} \leq {\bf v} \leq (p^k-1)\bf 1$ because $a(\bf 0)=1$. 
This shows that the family $(a({\bf n}))_{\bf n\in\mathbb{N}^d}$ satisfies the $p^k$-Lucas property with respect to $\mathcal S$. 
\medskip

Reciprocally, assume that $(a({\bf n}))_{\bf n\in\mathbb N^d}$ is a family with $a(\bf 0)=1$ that 
satisfies the $p^k$-Lucas property with respect to $\mathcal S$. Then setting 
$$
A({\bf x}) := \sum_{\mathbf{0}\leq\mathbf{v}\leq(p^k-1)\mathbf{1}}a(\mathbf{v})\mathbf{x}^{\mathbf{v}} \in R_{\mathfrak p}[{\bf x}]\, 
$$
and $f({\bf x}) :=\sum_{\bf n \in \mathbb N^d} a({\bf n}) {\bf x}^{\bf n}$, 
we immediately obtain that 
$$
f(\mathbf{x})\equiv A(\mathbf{x})f\big(\mathbf{x}^{p^k}\big)\mod \mathfrak p R_{\mathfrak p}[[\mathbf{x}]],
$$
which shows that $f$ belongs to $\mathfrak L_d(R,\mathcal S)$. 
\end{proof}

\medskip

Contrary to elements of $\mathcal L_d(R,\mathcal{S})$, those of $\mathfrak L_d(R,\mathcal{S})$ satisfy the following two additional useful properties. 
Recall that given two power series $f(\mathbf{x}):=\sum_{\mathbf{n}\in\mathbb{N}^d}a(\mathbf{n})\mathbf{x}^{\mathbf{n}}$ and 
$g(\mathbf{x})=\sum_{\mathbf{n}\in\mathbb{N}^d}b(\mathbf{n})\mathbf{x}^{\mathbf{n}}$ with coefficients in an arbitrary ring, 
one can define the Hadamard product of $f$ and $g$ by 
$$
f \odot g := \sum_{\mathbf{n}\in\mathbb{N}^d}a(\mathbf{n})b({\bf n}) \mathbf{x}^{\mathbf{n}} \, 
$$
and the diagonal of $f$ by
$$
\Delta(f) := \sum_{n=0}^\infty a(n,\dots,n)x^n \, . 
$$

\begin{prop}\label{prop: special2}
Let $f(\mathbf{x})$ and $g({\bf x})$ belong to $\mathfrak{L}_d(R,\mathcal{S})$. 
Then the following hold.  
\begin{enumerate}

\item[\rm{(i)}]  $f\odot g \in {\mathfrak{L}}_d(R,\mathcal{S})$.

\item[\rm{(ii)}]  $\Delta(f) \in {\mathfrak{L}}_1(R,\mathcal{S})$.

\end{enumerate}
\end{prop}

The proof of Proposition \ref{prop: special2} is straightforward using that the coefficients of $f$ and $g$ satisfy the 
$p^k$-Lucas property with respect to $\mathcal S$.


\section{A criterion for algebraic independence }\label{sec: kolchin}

In this section, we show that any algebraically dependent power series $f_1,\ldots,f_n$ 
that belong to $\mathcal L_d(R,\mathcal S)$, for an infinite set of primes $\mathcal S$, should in fact 
satisfy a very special kind of relation: a Laurent monomial is equal to a rational fraction. 
This provides a simple and useful algebraic independence criterion for elements of $\mathcal L_d(R,\mathcal S)$. 

\begin{thm}\label{thm: ind}
Let $R$ be a Dedekind domain and $f_1(\mathbf{x}),\dots,f_n(\mathbf{x})$ be power series in $\mathcal{L}_d(R,\mathcal{S})$ where $\mathcal S$ 
is an infinite set of primes ideals of $R$ of finite index. Let $K$ be the fraction field of $R$. Then the power series 
$f_1(\mathbf{x}),\dots,f_n(\mathbf{x})$ are algebraically 
dependent over $K(\mathbf{x})$ if and only if there exist $a_1,\dots,a_s\in\mathbb{Z}$ not all zero, such that
$$
f_1(\mathbf{x})^{a_1}\cdots f_n(\mathbf{x})^{a_s}\in K(\mathbf{x}).
$$
\end{thm}

\begin{rem}\label{rem: quant}{\rm
We actually prove a slightly more precise version of Theorem \ref{thm: ind}: if $f_1(\mathbf{x}),\dots,f_n(\mathbf{x})$ 
satisfy a polynomial relation of 
degree at most $d$ over $K({\bf x})$, then 
$$
f_1(\mathbf{x})^{a_1}\cdots f_n(\mathbf{x})^{a_s} = A({\bf x})\, ,
$$
where $|a_1+\cdots+a_n|\leq d$, $|a_i|\leq d$ for $1\leq i\leq n$, and $A(\bf x)$ is a rational fraction of height at most $2Cdn$. 
Here $C$ denotes the 
constant involved in the Landau notation of Condition (iv) in Definition \ref{def: ldrs}.
}
\end{rem}

\subsection{A Kolchin-like proposition}

Statements of the type of Proposition \ref{prop: Kolchin} below often appear in the study of systems of homogeneous linear 
differential/difference equations of order one. They are usually associated with the name of Kolchin who first proved one 
version in the differential case. 
We give here a rather general quantitative version in the case of difference equations associated 
with an injective endomorphism. We provide the simple proof below for the sake of completeness. 

\begin{prop}\label{prop: Kolchin}
Let $A$ be an integral domain, let $\sigma$ be an injective endomorphism of $A$, and let $B$ be a subring of $A$ such that $\sigma(B)\subset B$.  
Let $L$ denote the field of fractions of $B$. We also write $\sigma$ for the canonical extension 
of $\sigma$ to $L$. Let $f_1,\dots,f_n$ be non-zero elements of $A$ satisfying a non-trivial polynomial relation of degree $d$ with coefficients in $L$. 
If there exist $a_1,\dots,a_n$ in $B$ such that $f_i=a_i\sigma(f_i)$ for all $i$ in $\{1,\dots,n\}$, then there exist 
$m_1,\dots,m_n\in\mathbb{Z}$, not all zero, and $r\in L^*$ such that
$$
a_1^{m_1}\cdots a_n^{m_n}=\frac{\sigma(r)}{r}.
$$ 
Furthermore,  $|m_1+\cdots+m_n|\leq d$ and $|m_i|\leq d$ for $1\leq i\leq n$.
\end{prop}


\begin{proof}
Let $P$ be a polynomial with a minimal number of monomials among the non-zero polynomials 
in $L[X_1,\dots,X_n]$ of degree at most $d$ satisfying $P(f_1,\dots,f_n)=0$. 
We write
$$
P(X_1,\dots,X_n)=\sum_{(i_1,\dots,i_n)\in I}r_{i_1,\dots,i_n}X_1^{i_1}\cdots X_n^{i_n},
$$
with $r_{i_1,\dots,i_n}$ in $L\setminus\{0\}$. 
By assumption, we have
\begin{align}
0=\sigma\big(P(f_1,\dots,f_n)\big)&=\sum_{(i_1,\dots,i_n)\in I}\sigma(r_{i_1,\dots,i_n})\sigma(f_1)^{i_1}\cdots \sigma(f_n)^{i_n}\notag\\
&=\sum_{(i_1,\dots,i_n)\in I}\sigma(r_{i_1,\dots,i_n})\frac{f_1^{i_1}\cdots f_n^{i_n}}{a_1^{i_1}\cdots a_n^{i_n}}\label{E1}.
\end{align}
Let fix ${\bf i_0}=(s_1,\dots,s_n)$ in $I$. We also have
\begin{equation}\label{E2}
\sigma(r_{\bf{i_0}})P(f_1,\dots,f_n)=0.
\end{equation}
By multiplying \eqref{E1} by $r_{\bf{i_0}}a_1^{s_1}\cdots a_n^{s_n}$ and subtracting \eqref{E2}, 
we obtain a new polynomial in $L[X_1,\dots,X_n]$ of degree less than or equal to $d$, vanishing at 
$(f_1,\dots,f_n)$, but with a smaller number of monomials, so this polynomial has to be zero. 
Since all the $f_i$'s are non-zero, the cardinality of $I$ is at least equal to $2$.  
It follows that there exists 
${\bf i_1}=(t_1,\dots,t_n)$ in $I$, $\bf{i_1}\neq\bf{i_0}$, such that
$$
r_{\bf{i_0}}\sigma(r_{\bf{i_1}})a_1^{s_1-t_1}\dots a_n^{s_n-t_n}=\sigma(r_{\bf{i_0}})r_{\bf{i_1}},
$$
which leads to
$$
a_1^{s_1-t_1}\dots a_n^{s_n-t_n}=\frac{\sigma(r_{\bf{i_0}})r_{\bf{i_1}}}{r_{\bf{i_0}}\sigma(r_{\bf{i_1}})}.
$$
Hence it suffices to take $m_i=s_i-t_i$ and $r=r_{\bf{i_0}}/r_{\bf{i_1}}$. Furthermore, since $P$ has total degree at most $d$, we have $|m_1+\cdots+m_n|\leq d$ and $|m_i|\leq d$, for $1\leq i\leq n$. 
\end{proof}

\subsection{Reduction modulo prime ideals}

We can now proceed with the proof of Theorem~\ref{thm: ind}. We first recall the following simple lemma.

\begin{lem}
\label{lem: 0} Let $R$ be a Dedekind domain, $K$ its field of fractions and $f_1,\ldots, f_n$ power series in $K[[{\bf x}]]$. Let 
${\mathcal S}$ denote an infinite set of prime ideals of $R$ such that $f_1,\ldots, f_n$ belong to $R_{\mathfrak p}[[{\bf x}]]$ 
for every $\mathfrak p$ in $\mathcal S$. If 
$f_{1 \vert \mathfrak p},\ldots, f_{n \vert \mathfrak p}$ are linearly dependent over 
$R/\mathfrak p$ for all ideals $\mathfrak p$ in $\mathcal S$,  
then $f_1,\ldots, f_n$ are linearly dependent over $K$. 
\end{lem}

\begin{proof}
Let $a_i(n)$ denote the $n$-th coefficient of the power series $f_i$. Let us consider 
$$
 \begin{pmatrix} a_1(0)&a_1(1)&a_1(2)&\cdots \\
a_2(0)&a_2(1)&a_2(2)& \cdots\\
\vdots &\vdots &\vdots &\ldots\\
a_n(0)&a_n(1)&a_n(2)&\cdots
\end{pmatrix},
$$
 the $n\times \infty$ matrix whose coefficient in position $(i,j)$ is $a_i(j-1)$. By assumption, 
 $f_{1\vert \mathfrak p},\ldots, f_{n\vert \mathfrak p}$ are linearly dependent over $R/\mathfrak p$ for all $\mathfrak p$ in $\mathcal S$.   
This implies that, for such a prime ideal, every $n\times n$ minor has determinant that vanishes modulo $\mathfrak p$. In other words, every  
$n\times n$ minor has determinant  that belongs to $\mathfrak pR_{\mathfrak p}$.  
Since $R$ is a Dedekind domain, a non-zero element in $K$ belongs to only finitely many ideals $\mathfrak pR_{\mathfrak p}$. 
Since $\mathcal S$ is infinite, we obtain that all $n\times n$ minors are actually equal to zero in $K$.  
This means that the set of all column vectors of our matrix generate a vector space $E$ of dimension less than $n$. Hence there is a non-zero linear form on $K^n$ which vanishes on $E$ and we obtain a non-zero vector $(b_1,\dots,b_n)$ in $K^n$ such that $\sum_{i=1}^n b_if_i = 0$. Thus $f_1,\ldots, f_n$ are linearly dependent over $K$.  
\end{proof}

We can now complete the proof of Theorem \ref{thm: ind}. 

\begin{proof}[Proof of Theorem \ref{thm: ind}] 
Let $f_1(\mathbf{x}),\dots,f_n(\mathbf{x})$ in $\mathcal{L}_d(R,\mathcal{S})$ be algebraically dependent over $K({\bf x})$. Let $Q(\mathbf{x},y_1,\dots,y_n)$ be a non-zero polynomial in $R[\mathbf{x}][y_1,\dots,y_n]$ of total degree at most 
$d$ in $y_1,\dots,y_n$ such that 
$$
Q\big(\mathbf{x},f_1(\mathbf{x}),\dots,f_n(\mathbf{x})\big)=0. 
$$
With all $\mathfrak p$ in $\mathcal S$, we associate a prime number $p$ such that the residue field 
$R/\mathfrak p$ is a finite field of characteristic $p$. Let $d_{\mathfrak p}$ be the degree of the field extension 
$R/\mathfrak p$ over $\mathbb F_p$.  
By Definition \ref{def: ldrs}, for all $i$ in $\{1,\dots,n\}$, there exists a positive real number $C_i$ such that, 
for all $\mathfrak p$ in $\mathcal S$, there are positive integers $k_{i}$ and $A_{i}(\mathbf{x})$ in $K({\bf x})\cap R_{\mathfrak p}[[\bf x]]$ satisfying
$$
f_i(\mathbf{x})\equiv A_{i}(\mathbf{x})f_i\big(\mathbf{x}^{p^{k_{i}}}\big)\mod \mathfrak pR_{\mathfrak p}[[{\bf x}]],
$$
with $H(A_{i})\leq C_i p^{k_{i}}$. 
We set $C:=2\max(C_1,\dots,C_n)$ and $k:=\mathrm{lcm}(d_{\mathfrak p},k_{1},\dots,k_{n})$. 
Hence by Remark \ref{rem: itere}, for all $i$ in $\{1,\dots,n\}$ and all $\mathfrak p$ in $\mathcal S$, there exists 
$B_{i}(\mathbf{x})$ in $K({\bf x})\cap R_{\mathfrak p}[[\mathbf{x}]]$ satisfying
$$
f_i(\mathbf{x})\equiv B_{i}(\mathbf{x})f_i\big(\mathbf{x}^{p^{k}}\big)\mod \mathfrak pR_{\mathfrak p}[[\mathbf{x}]],
$$
with $H(B_{i})\leq C p^{k}$. 

Since $Q$  is non-zero and $R$ is a Dedekind domain, the coefficients of $Q$ belong 
to at most finitely many prime ideals  $\mathfrak p$ of $\mathcal S$. There thus exists an infinite subset $\mathcal S'$ of $\mathcal S$ such 
that, for every $\mathfrak p$ in $\mathcal S'$, $Q_{|\mathfrak p}$ is a non-zero polynomial in $(R/\mathfrak p)[\mathbf{x}][y_1,\dots,y_n]$ of total degree at 
most $d$ in $y_1,\ldots,y_n$ that vanishes at $(f_{1| \mathfrak p}(\mathbf{x}),\dots,f_{n|\mathfrak p}(\mathbf{x}))$. 
We can thus apply Proposition~\ref{prop: Kolchin} to $f_{1|\mathfrak p},\dots,f_{s|\mathfrak p}$ with $A:=(R/\mathfrak p)[[\mathbf{x}]]$, $B:= (R/\mathfrak p)(\mathbf{x})\cap A$, $L=(R/\mathfrak p)(\bf x)$ and $\sigma$ the injective endomorphism of $A$ defined by 
$$
\sigma\big(g(\mathbf{x})\big)=g\big(\mathbf{x}^{p^k}\big)=g(\mathbf{x})^{p^k},\;(g(\mathbf{x})\in A),
$$
where the last equality holds because $d_{\mathfrak p}$ divides $k$. Then Proposition~\ref{prop: Kolchin} implies that  there exist integers $m_{1},\dots,m_{s}$, not all zero, 
and a non-zero rational fraction $r(\mathbf{x})$ in $(R/\mathfrak p)(\mathbf{x})$ such that
\begin{equation}\label{eq: Bi}
B_{1}(\mathbf{x})^{m_{1}}\cdots B_{n}(\mathbf{x})^{m_{n}}=
\frac{r\big(\mathbf{x}^{p^k}\big)}{r(\mathbf{x})}=r(\mathbf{x})^{p^k-1}.
\end{equation}
By Remark \ref{rem: a(0)}, the constant coefficient in the left-hand side of \eqref{eq: Bi} is equal to $1$. It thus follows that the constant coefficient of the power series 
$r$ is non-zero. We can thus assume without any loss of generality that the constant coefficient of $r$ is equal to $1$. 
Furthermore, we have $|m_1+\cdots+m_n|\leq d$ and $|m_i|\leq d$ for $1\leq i\leq n$.
Note that the rational fractions $B_i$, $r$ and the integers $m_i$ all depend on $\mathfrak p$. 
However, since all the $m_i$'s belong to a finite set, the pigeonhole principle implies the existence 
of  an infinite subset $\mathcal{S}''$ of $\mathcal{S}'$ and of integers $t_1,\dots,t_n$ independent of $\mathfrak p$ 
such that, for all $\mathfrak p$ in $\mathcal{S}''$, we have $m_i=t_i$ for $1\leq i \leq n$.  
Assume now that $\mathfrak p$ is a prime ideal in $\mathcal{S}''$  
and write $r(\mathbf{x})=s(\mathbf{x})/t(\mathbf{x})$ with $s(\mathbf{x})$ and $t(\mathbf{x})$ in $(R/\mathfrak p)[{\bf x}]$ and coprime.  
Since $H(B_{i})\leq C p^k$, the degrees of $s({\bf x})$ and $t({\bf x})$ are bounded by 
$$
\frac{p^k}{p^k-1}C\big(|t_1|+\cdots+|t_n|\big)\leq 2Cdn.
$$
Set 
$$
h(\mathbf{x}):=f_1(\mathbf{x})^{-t_1}\cdots f_n(\mathbf{x})^{-t_n}\in  K[[\mathbf{x}]].
$$
Then, for every $\mathfrak p$ in $\mathcal S''$, we obtain that 
\begin{align*}
h_{|\mathfrak p}\big(\mathbf{x}^{p^k}\big)&\equiv f_{1|\mathfrak p}\big(\mathbf{x}^{p^k}\big)^{-t_1}\cdots f_{n|\mathfrak p}\big(\mathbf{x}^{p^k}\big)^{-t_n} 
\mod \mathfrak pR_{\mathfrak p}[[\mathbf{x}]]\\
&\equiv f_{1|\mathfrak p}(\mathbf{x})^{-t_1}\cdots f_{n|\mathfrak p}(\mathbf{x})^{-t_n}B_{1}(\mathbf{x})^{t_1}\cdots B_{n}(\mathbf{x})^{t_n}
\mod \mathfrak pR_{\mathfrak p}[[\mathbf{x}]]\\
&\equiv h_{|\mathfrak p}(\mathbf{x})r(\mathbf{x})^{p^k-1} \mod \mathfrak pR_{\mathfrak p}[[\mathbf{x}]] .
\end{align*}
Since $h_{|\mathfrak p}$ is not zero, we obtain that 
$h_{|\mathfrak p}(\mathbf{x})^{p^k-1} \equiv r(\mathbf{x})^{p^k-1}$ and there is $a$ in a suitable algebraic extension of 
$R/\mathfrak p$  
such that $h_{|\mathfrak p}(\mathbf{x})=ar(\mathbf{x})$.  But,  the constant coefficients of both 
$h_{|\mathfrak p}$ and $r$ are equal to $1$, and hence   
$h_{|\mathfrak p}(\mathbf{x})= r(\mathbf{x})$.   
Thus, for infinitely many prime ideals $\mathfrak p$, the reductions modulo $\mathfrak p$ of the power series  
$x_i^mh(\mathbf{x})$ and $x_i^m$, $1\leq i\leq d$, $0\leq m\leq 2Cdn$, are linearly independent 
over $R/\mathfrak p$.  Since $R$ is a Dedekind domain, Lemma~\ref{lem: 0}  
implies that these power series are linearly dependent over 
$K$, which means that $h(\mathbf{x})$ belong to $K({\bf x})$. This ends the proof. 
\end{proof}

\section{Algebraic functions in $\mathcal{L}_d(R,\mathcal{S})$ and $\mathfrak{L}_d(R,\mathcal{S})$}\label{sec: Allouche}

 The aim of this section is to describe which power series among $\mathcal{L}_d(R,\mathcal{S})$ and $\mathfrak{L}_d(R,\mathcal{S})$ 
 are algebraic over $K({\bf x})$.  Here, we keep the notation of Section \ref{sec: ldks}, we fix a Dedekind domain $R$ and $K$ still denotes the fraction field of $R$. For every prime ideal $\mathfrak{p}$ of $R$ of finite index, we write $d_{\mathfrak{p}}$ for the degree of the field extension $R/\mathfrak{p}$ over $\mathbb{F}_p$. As a consequence of Theorem~\ref{thm: ind}, we deduce the following generalization of the main result of \cite{AGS}. In their  Theorem 1,  Allouche, Gouyou-Beauchamps and Skordev  \cite{AGS} characterize the algebraic power series of one variable with rational coefficients that have the $p$-Lucas property with respect to primes in an arithmetic progression of the form $1+ s\mathbb N$.

\begin{prop}\label{prop: alg1}
Let $f(\mathbf{x})$ be in $\mathcal{L}_d(R,\mathcal{S})$ for an infinite set $\mathcal{S}$. Assume that $f(\mathbf{x})$ is algebraic over $K(\mathbf{x})$ of degree less than or equal to $\kappa$. Then there exists a rational fraction $r(\mathbf{x})$ in $K(\mathbf{x})$, with $r({\bf 0})=1$, and a positive integer $a\leq \kappa$ such that $f(\mathbf{x})=r(\mathbf{x})^{1/a}$.

Reciprocally, if $f(\mathbf{x})=r(\mathbf{x})^{1/a}$ where $r(\mathbf{x})$ is in $K(\mathbf{x})$, with $r({\bf 0})=1$, and $a$ is a positive integer, then $f(\mathbf{x})$ belongs to 
$\mathcal{L}_d(R,\mathcal{S})$ where $\mathcal{S}$ is the set of all prime ideals $\mathfrak{p}$ of $R$ such that $r(\mathbf{x})\in R_{\mathfrak{p}}(\mathbf{x})$ and $R/\mathfrak{p}$ is a finite field of characteristic in $1+a\mathbb{N}$.
\end{prop}

\begin{proof}[Proof of Proposition \ref{prop: alg1}]
Let us first assume that there is an infinite set $\mathcal S$ such that $f$ belongs to $\mathcal{L}_d(R,\mathcal{S})$ and is algebraic. We can apply Theorem 
\ref{thm: ind} in the case of a single function. We obtain that there exists
a positive integer $a\leq \kappa$ and a rational fraction $r(\mathbf{x})$ in $K(\mathbf{x})$  
such that $f(\mathbf{x})^a=r(\mathbf{x})$, and $r({\bf 0})=1$ as expected.

\medskip

Conversely, assume that there is a positive integer $a\leq\kappa$ such that $f(\mathbf{x})=r(\mathbf{x})^{1/a}$ with $r(\bf x)$ in $K(\bf x)$ and $r(\mathbf{0})=1$. Of course, $f$ is algebraic over $K({\bf x})$ with degree at most $\kappa$. Let $\mathfrak{p}$ be a prime ideal of $R$ such that $r(\mathbf{x})\in R_{\mathfrak{p}}(\mathbf{x})$ and $R/\mathfrak{p}$ is a finite field of characteristic $p$ in $1+a\mathbb{N}$. Note that there exists a natural number $b$ such that $p=1+ab$, and thus  
we have $f(\mathbf{x})^{p-1}=f(\mathbf{x})^{ab}=r(\mathbf{x})^{b}$. This gives   
$f(\mathbf{x})=r(\mathbf{x})^{-b} f(\mathbf{x})^p$ and thus 
$$
f(\mathbf{x})=r(\mathbf{x})^{ -b(1+p+\cdots + p^{d_{ \mathfrak p }-1} ) }f(\mathbf{x})^{ p^{ d_{\mathfrak p } }} .
$$   
It follows that
$$
f(\mathbf{x})\equiv A(\mathbf{x}) f\big(\mathbf{x}^{p^{d_{\mathfrak{p}}}}\big)\mod \mathfrak pR_{\mathfrak p}[[\mathbf{x}]],
$$
with $A(\mathbf{x})$ in $K(\mathbf{x})\cap R_{\mathfrak p}[[\mathbf{x}]]$ and
$H(A )\leq 2bp^{d_{\mathfrak{p}}-1} H(r)\leq 2H(r) p^{d_{\mathfrak{p}}}$. 
This shows that $f$ and $\mathfrak{p}$ satisfy Conditions $(\rm{i})$--$(\rm{iv})$ in Definition \ref{def: ldrs}, as expected.
\end{proof}

We also have the following similar characterization of algebraic formal power series in $\mathfrak{L}_d(R,\mathcal{S})$.

\begin{prop}\label{prop: alg2}
Let $f(\mathbf{x})$ be in $\mathfrak{L}_d(R,\mathcal{S})$ for an infinite set $\mathcal{S}$. Assume that $f(\mathbf{x})$ is algebraic over $K(\mathbf{x})$ of degree less than or equal to $\kappa$. Then there exists a polynomial $P(\mathbf{x})$ in $K[\mathbf{x}]$, with $P(\mathbf{0})=1$, and a positive integer $a\leq\kappa$ such that $f(\mathbf{x})=P(\mathbf{x})^{-1/a}$ with $\deg_{x_i}(P(\mathbf{x}))\leq a$ for all $i$ in $\{1,\dots,d\}$. 

Reciprocally, if $f(\mathbf{x})=P(\mathbf{x})^{-1/a}$ where $P(\mathbf{x})$ is in $K[\mathbf{x}]$, with $P(\mathbf{0})=1$, and $a$ is a positive integer such that $\deg_{x_i}(P(\mathbf{x}))\leq a$ for all $i$ in $\{1,\dots,d\}$, then for every prime ideal $\mathfrak{p}$ in $R$ such that $P(\mathbf{x})$ is in $R_{\mathfrak{p}}[\mathbf{x}]$ and $R/\mathfrak{p}$ is a finite field of characteristic $p$ in $1+a\mathbb{N}$, $f(\mathbf{x})$ satisfies the $p^{d_{\mathfrak{p}}}$-Lucas property.
\end{prop}

\begin{proof} 
Let us first assume that there is an infinite set $\mathcal{S}$ such that $f$ belongs to $\mathfrak{L}_d(R,\mathcal{S})$ and is algebraic. By Proposition \ref{prop: alg1}, there are a positive integer $a\leq \kappa$ and a rational fraction
$r(\mathbf{x})$ in $K(\mathbf{x})$, with $r(\mathbf{0})=1$, such that $f(\mathbf{x})=r(\mathbf{x})^{1/a}$. 
We write $r(\mathbf{x})=s(\mathbf{x})/t(\mathbf{x})$ with $s(\mathbf{x})$ coprime to $t(\mathbf{x})$ and $s(\mathbf{0})=t(\mathbf{0})=1$. 
Since the resultant of $s(\mathbf{x})$ and $t(\mathbf{x})$ is a non-zero element of $K$, there exists an infinite subset 
$\mathcal{S}'$ of $\mathcal{S}$ such that $s_{|\mathfrak p}(\mathbf{x})$ and $t_{|\mathfrak p}(\mathbf{x})$ are coprime and non-zero 
for all prime ideals $\mathfrak p$ in $\mathcal{S}'$. 
We fix $\mathfrak p$ in $\mathcal{S}'$ and we let $p$ be the characteristic of $R/\mathfrak p$. By assumption, there exist a positive integer $k$ and $A(\mathbf{x})$ in $R_{\mathfrak p}[\mathbf{x}]$, with $\deg_{x_i}(A)\leq p^k-1$ 
 for all $i$ in $\{1,\ldots,d\}$, such that
$$
f(\mathbf{x})\equiv A(\mathbf{x})f(\mathbf{x}^{p^k}) \mod \mathfrak pR_{\mathfrak p}[[ {\bf x} ]].
$$
By Remark \ref{rem: itere2}, we can assume that $d_{\mathfrak p}$ divides $k$. 
This yields
$$
f(\mathbf{x})^{p^k-1}\equiv  A(\mathbf{x})^{-1}\mod \mathfrak pR_{\mathfrak p}[[\bf x]] 
$$
and
\begin{equation}\label{eq: degA}
t(\mathbf{x})^{p^k-1} \equiv s(\mathbf{x})^{p^k-1}A(\mathbf{x})^{a} \mod \mathfrak p R_{\mathfrak p}[[\bf x]] \, . 
\end{equation}
Since $t_{|\mathfrak p}(\mathbf{x})$ is coprime to $s_{|\mathfrak p}(\mathbf{x})$ and $s(\mathbf{0})=1$, we deduce that $s_{|\mathfrak p}(\mathbf{x})=1$. Since $\mathcal{S}'$ is infinite, we obtain $s(\mathbf{x})=1$, as expected. Finally, we have $\deg_{x_i}(t) \leq a$ for all $i$ in $\{1,\dots,d\}$. Indeed, Congruence \eqref{eq: degA} implies that $\deg_{x_i}(t_{|\mathfrak p}) \leq a$ for all $i$ in $\{1,\dots,d\}$ and all $\mathfrak p$ in $\mathcal S'$. 
\medskip

Conversely, assume that $f(\mathbf{x})=P(\mathbf{x})^{-1/a}$ where $P(\mathbf{x})$ is in $K[\mathbf{x}]$, with $P(\mathbf{0})=1$, and $a$ is a positive integer such that $\deg_{x_i}(P(\mathbf{x}))\leq a$ for all $i$ in $\{1,\dots,d\}$. Let $\mathfrak{p}$ be a prime ideal in $R$ such that $P(\mathbf{x})$ is in $R_{\mathfrak{p}}[\mathbf{x}]$ and $R/\mathfrak{p}$ is a finite field of characteristic $p$ in $1+a\mathbb{N}$. Following the proof of Proposition \ref{prop: alg1}, we obtain that
$$
f(\mathbf{x})=P(\mathbf{x})^{b(1+p+\cdots+p^{d_{\mathfrak{p}}-1})}f(\mathbf{x})^{p^{d_{\mathfrak{p}}}},
$$
where $b$ satisfies $p=1+ab$. It follows that 
$$
f(\mathbf{x})\equiv A(\mathbf{x})f\big(\mathbf{x}^{p^{d_{\mathfrak{p}}}}\big)\mod \mathfrak{p}R_{\mathfrak{p}}[[\mathbf{x}]],
$$
with $A(\mathbf{x})$ in $R_{\mathfrak{p}}[\mathbf{x}]$ and $\deg_{x_i}A(\mathbf{x})\leq p^{d_{\mathfrak{p}}}-1$ as expected.
\end{proof}

\begin{rem}{\rm  With these first principles in hand, one can already deduce non-trivial results.  
As a direct consequence of  Proposition \ref{prop: alg2} with $R=\mathbb{Z}$, we get that the bivariate power series $\frac{1}{1-x(1+y)}$ satisfies the 
$p$-Lucas property for all prime numbers $p$, which is equivalent to Lucas' theorem since we have 
$$\frac{1}{1-x(1+y)} = \sum_{n,k\geq 0} {n\choose k} x^ny^k \, .$$ 
Using Proposition \ref{prop: special2} and Proposition \ref{prop: alg2}, we also obtain  
that given any polynomial $P(\mathbf{x})$ in $\mathbb{Q}[\mathbf{x}]$ with $P({\bf 0})=1$, then  we have 
$$
\Delta\left(\frac{1}{P(\mathbf{x})^{1/a}}\right)\in \mathfrak{L}(\mathcal{S}),
$$
for every positive integer $a\geq\max\{\deg_{x_i}(P(\mathbf{x}))\,:\,1\leq i\leq d\}$ and 
$$
\mathcal{S}=\big\{p\in\mathfrak{P}\,:\, p\equiv 1\mod a \textup{ and } P({\bf x})\in \mathbb Z_{(p)}[{\bf x}]\big\}.
$$
This has interesting consequences. Choosing for instance 
$P(x_1,\ldots,x_d)= 1- (x_1+\cdots+x_d)$ and $a=1$, we deduce from Proposition  \ref{prop: special2} 
that for every positive integer $t$ the power series 
$\sum_{n=0}^{\infty} {dn \choose n,n,\ldots,n}^t x^n$ satisfies the $p$-Lucas property for all prime numbers $p$. Choosing 
$P(x_1,x_2,x_3,x_4) = (1-x_1-x_2)(1-x_3-x_4)-x_1x_2x_3x_4$ and $a=1$, we recover a classical result of Gessel \cite{Ge82}: the sequence of Apr\'ery numbers 
$$\left(\sum_{k=0}^n {n\choose k}^2{n+k\choose k}^2  \right)_{n\geq 0}$$ 
satisfies the $p$-Lucas property for all prime numbers $p$\footnote{The fact that the diagonal of $1/P$ is the generating series of the Ap\'ery numbers can be found in  \cite{Str}.}.  
}
\end{rem}

\section{From asymptotics and singularity analysis to algebraic independence}\label{sec: sing}

In this section,  we emphasize the relevance of Theorem \ref{thm: ind} for proving algebraic independence by using complex analysis. We fix a Dedekind domain $R\subset\mathbb{C}$ and an infinite set $\mathcal{S}$ of prime ideals of $R$ of finite index. We still write $K$ for the fraction field of $R$. Let $f_1(\mathbf{x}),\dots,f_n(\mathbf{x})$ be power series in $\mathcal{L}_d(R,\mathcal{S})$ algebraically dependent over $\mathbb{C}(\mathbf{x})$. By Proposition \ref{prop: descente} below, those power series are algebraically dependent over $K(\mathbf{x})$. Then Theorem \ref{thm: ind} yields integers $a_1,\dots,a_n$, not all zero, and a rational fraction $r(\mathbf{x})$ in $K(\mathbf{x})$ such that
\begin{equation}\label{eq contr}
f_1(\mathbf{x})^{a_1}\cdots f_n(\mathbf{x})^{a_n}=r(\mathbf{x}) \, . 
\end{equation}
We will describe in this section some basic principles that allow to reach a contradiction with \eqref{eq contr} and that thus lead to the algebraic independence of the 
$f_i$'s. The key feature when dealing with one-variable complex functions is that one can 
derive a lot of information from the study of their singularities and asymptotics for their coefficients. 
It is well-known that asymptotics of coefficients of analytic functions and analysis of their singularities are intimately related and there are strong transference theorems 
that allow to go from one to the other viewpoint. This connection is for instance described in great detail in the book of Flajolet and Sedgewick \cite{FS}.

\begin{rem}{\em If the multivariate 
functions $f_1(\mathbf{x}),\dots,f_n(\mathbf{x})$ in $\mathcal L_d(R,\mathcal S)$ are algebraically dependent over $K(\mathbf{x})$, then 
 the univariate power series $f_i(\lambda_1x^{n_1},\dots,\lambda_dx^{n_d})$, $1\leq i\leq n$, where  
 $\lambda_i\in \mathbb{C}^*$ and $n_i\geq 1$,  are algebraically dependent over $K(x)$.  We thus stress that  the principles described in this section 
 could also be used to prove the algebraic independence of multivariate functions.  }
\end{rem}

As announced above, we have the following result.

\begin{prop}\label{prop: descente}
Let $K$ be a subfield of $\mathbb{C}$ and $f_1(\mathbf{x}),\dots,f_n(\mathbf{x})$ be power series in $K(\mathbf{x})$ algebraically dependent over $\mathbb{C}(\mathbf{x})$. Then those power series are algebraically dependent over $K(\mathbf{x})$.
\end{prop}

\begin{proof}
Let $P(Y_1,\dots,Y_n)$ be a non-zero polynomial with coefficients in $\mathbb{C}[\mathbf{x}]$ such that $P(f_1(\mathbf{x}),\dots,f_n(\mathbf{x}))=0$. Let $w_1(\mathbf{x}),\dots,w_r(\mathbf{x})$ be the coefficients of $P$. Since the $w_i$'s are polynomials, their coefficients span a finitely generated field extension of $K$, say $L=K(t_1,\dots,t_s,\varphi)$ where $t_1,\dots,t_s$ are complex numbers algebraically independent over $K$ and $\varphi$ is a complex number algebraic over $K(t_1,\dots,t_s)$ of degree $h$. Let $m$ be the maximum of the partial degrees of the $w_i$'s. By multiplying $P$ by a suitable polynomial in $K[t_1,\dots,t_s]$, we may assume that
$$
w_i(\mathbf{x})=\sum_{\mathbf{0}\leq\mathbf{k}\leq m\mathbf{1}}\left(\sum_{\ell=0}^{h-1}\sum_{\mathbf{j}\in J_{\mathbf{k},\ell}}\lambda_i(\mathbf{j},\mathbf{k},\ell)\mathbf{t}^{\mathbf{j}}\varphi^\ell\right)\mathbf{x}^{\mathbf{k}},
$$
where the sets $J_{\mathbf{k},\ell}$ are finite and the coefficients $\lambda_i(\mathbf{j},\mathbf{k},\ell)$ belong to $K$ and are not all zero. Hence there are pairwise distinct $\mathbf{a}_1,\dots,\mathbf{a}_r$ in $\mathbb{N}^n$ such that
\begin{equation}\label{eq: phi}
\sum_{\ell=0}^{h-1}\left(\sum_{i=1}^r\sum_{\mathbf{0}\leq\mathbf{k}\leq m\mathbf{1}}\sum_{\mathbf{j}\in J_{\mathbf{k},\ell}}\lambda_i(\mathbf{j},\mathbf{k},\ell)\mathbf{t}^{\mathbf{j}}\mathbf{x}^{\mathbf{k}}\mathbf{f}(\mathbf{x})^{\mathbf{a}_i}\right)\varphi^\ell=0,
\end{equation}
where $\mathbf{f}(\mathbf{x})$ denotes $(f_1(\mathbf{x}),\dots,f_n(\mathbf{x}))$. By expanding the left-hand side of \eqref{eq: phi}, we obtain a power series in $\mathbf{x}$ whose coefficients are polynomial of degree at most $h-1$ in $\varphi$ with coefficients in $K(t_1,\dots,t_s)$. Hence the latter coefficients are zero and, for every $\ell$, we have
\begin{equation}\label{eq: t_i}
\sum_{\mathbf{0}\leq\mathbf{k}\leq m\mathbf{1}}\sum_{\mathbf{j}\in J_{\mathbf{k},\ell}}\left(\sum_{i=1}^r\lambda_i(\mathbf{j},\mathbf{k},\ell)\mathbf{x}^{\mathbf{k}}\mathbf{f}(\mathbf{x})^{\mathbf{a}_i}\right)\mathbf{t}^{\mathbf{j}}=0.
\end{equation}
Since the variables $x_1,\dots,x_d$ are algebraically independent over $\mathbb{C}$, the complex numbers $t_1,\dots,t_s$ are still algebraically independent over $K(\mathbf{x})$ and Equation \eqref{eq: t_i} yields a non-trivial polynomial relation over $K(\mathbf{x})$ for the $f_i$'s.
\end{proof}

 \subsection{General principle}

We write $\mathcal{L}(R,\mathcal{S})$ for $\mathcal{L}_1(R,\mathcal{S})$ and we let $\mathbb{C}\{z\}$ denote the set of complex functions that are analytic at the origin. 
For such a function $f(z)$, we denote by $\rho_f$  its radius of convergence. We recall that when $\rho_f$ is finite, 
 $f$ must have a singularity on the circle $|z|=\rho_f$. 

\begin{defn}{\rm 
Let $\mathfrak W$ denote the set of all analytic functions $f(z)$ in $\mathbb{C}\{z\}$ whose radius of convergence 
is finite and for which there exists $z\in\mathbb{C}$, $|z|=\rho_f$, such that no positive integer power of $f$ admits a meromorphic 
continuation to a neighborhood of $z$. }
\end{defn}

\begin{prop}\label{prop: sing1}
 Let $f_1(z),\dots,f_n(z)$ be functions that belong to $\mathcal{L}(R,\mathcal{S})\cap\mathfrak{W}$ for an infinite set $\mathcal{S}$, and 
 such that $\rho_{f_1},\dots,\rho_{f_n}$ are pairwise distinct. Then $f_1(z),\ldots,f_n(z)$ are algebraically independent over $\mathbb{C}(z)$.
\end{prop}

\begin{proof}
Let us assume by contradiction that  $f_1,\ldots,f_n$ are algebraically dependent over $\mathbb{C}(z)$. By Proposition \ref{prop: descente}, 
they are algebraically dependent over $K(z)$. 
Since the $f_i$'s belong to $\mathcal{L}(R,\mathcal{S})$, we can first apply Theorem \ref{thm: ind}.  
We obtain that there exist integers $a_1,\dots,a_n$, not all zero, and a rational fraction $r(z)$ in $K(x)$ such that
\begin{equation}\label{eq: imposs}
f_1(z)^{a_1}\cdots f_n(z)^{a_n}=r(z) \, .
\end{equation}
Thereby, it suffices to prove that \eqref{eq: imposs} leads to a contradiction. 
We can assume without loss of generality that $\rho_{f_1}<\cdots <\rho_{f_n}$. Let $j$ be the smallest index for which 
 $a_j \not =0$.   
We obtain that
\begin{equation}\label{members}
f_j(z)^{a_j}=r(z) f_{j+1}(z)^{-a_{j+1}}\cdots f_n(z)^{-a_n} \, .
\end{equation}
We can assume that $a_j$ is positive since otherwise we could write 
$$
f_j(z)^{-a_j}= r(z)^{-1}f_{j+1}(z)^{a_{j+1}}\cdots f_n(z)^{a_n}.
$$ 
By assumption, the function $f_j(z)$ belongs to $\mathfrak W$ and has thus a singularity at a point $z_0\in\mathbb{C}$ with 
$|z_0|=\rho_{f_j}$ and such that $f_j^{a_j}$ has no 
meromorphic continuation to a neighborhood of $z_0$. But the right-hand side is clearly meromorphic in a neighborhood of $z_0$. 
Hence we have a contradiction. This ends the proof. 
\end{proof}

\begin{rem}\label{rem: W}{\rm The set $\mathfrak{W}$ contains all functions 
$f(z)$ in $\mathbb{C}\{z\}$ with a finite radius of convergence and whose coefficients $a(n)$ satisfy:    
$$
a(n)\in\mathbb{R}_{\geq 0} \quad\mbox{and} \quad a(n)=O\left(\frac{\rho_f^{-n}}{n}\right) \, \cdot
$$
Indeed,  for such a function there exist positive constants $C_1$ and $C_2$ such that, for all $z$ in $\mathbb{C}$ satisfying $\frac{\rho_f}{2}<|z|<\rho_f$, we have 
\begin{eqnarray*}
|f(z)| & \leq& |a(0)|+C_1\sum_{n=1}^\infty\frac{(|z|/\rho_f)^n}{n}\\ 
& \leq &-C_2\log\left(1-\frac{|z|}{\rho_f}\right) \, .
\end{eqnarray*}
By Pringsheim's theorem, $f$ has a singularity at $\rho_f$. If $f$ is not in $\mathfrak W$, then there exists a positive integer $r$ such that $f^r$ is meromorphic at $\rho_f$. 
On the other hand, the inequality above shows that $f^r$ cannot have a pole at $\rho_f$. Thus $f^r$ is analytic at $\rho_f$ and $f^r(z)$ has a limit as $z$ tends to $\rho_f$. But if this limit is non-zero, then $f$ would be also analytic at $\rho_f$, a contradiction.  
It follows that 
$$
\lim_{z\to \rho_f} f(z)  =0\, ,
$$
but this is impossible because the coefficients of $f$ are non-negative (and not all zero since $\rho_f$ is finite). Hence $f(z)$ belongs to $\mathfrak W$. 
}
\end{rem}

Here is a first example illustrating the relevance of Proposition \ref{prop: sing1}.  

\begin{thm}\label{thm: 2nn}
Set
$$
\mathcal F := \left\{f_r(z):= \sum_{n=0}^{\infty}\binom{2n}{n}^r z^n\,:\, r\geq 2 \right\}.
$$ 
Then all elements of $\mathcal F$ are algebraically independent over $\mathbb C(z)$.
\end{thm}

 In \cite{AB13}, the first two authors proved that all functions 
$f_r(z)$, $r\geq 4$, are algebraically independent. Note that Theorem \ref{thm: 2nn} is optimal since $f_1(z) = 1/\sqrt{1-4z}$ is an algebraic function.

\begin{proof}
Let $r\geq 2$ be an integer. 
We first observe, as a direct consequence of Lucas' theorem, that the functions $f_r$ all satisfy the $p$-Lucas property for all primes. 
Hence $f_r$ belongs to $\mathfrak L(\mathcal P)$. 
Using Stirling formula, 
we also get that 
$$
\binom{2n}{n}^r \underset{n\rightarrow\infty}\sim \frac{2^{(2n+1/2)r}}{(2\pi n)^{r/2}} \, \cdot
$$
This shows that $\rho_{f_r} = 2^{-2r}$. 
By Remark \ref{rem: W}, it follows that $f_r(z)\in \mathfrak W$ since $r\geq 2$ implies that  
$$
\frac{2^{(2n+1/2)r}}{(2\pi n)^{r/2}}= O\left(\frac{\rho_{f_r}^{-n}}{n}\right) \, .
$$
The result then follows from a direct application of Proposition \ref{prop: sing1}.
\end{proof}

In Proposition \ref{prop: sing1} and the application above, we use the fact that the radius of convergence of the involved functions 
are all distinct. We observe below that this condition is not necessary to apply our method.

\begin{prop}\label{prop: sing2}
 Let $f_1(z)$ and $f_2(z)$ be two transcendental functions in $\mathcal L(R,\mathcal S)$ with same finite positive radius of 
 convergence $\rho$. Assume that $f_1$ and $f_2$ have a singularity at a point $z_0\in\mathbb{C}$, with $|z_0|=\rho$, such that 
 the following hold. 
\begin{enumerate}
\item[\rm (i)] There is no $(C,\alpha)$ in $\mathbb{C^*}\times\mathbb{Q}$ such that $f_1(tz_0)\underset{\substack{t \to 1 \\ t\in(0,1)}}{\sim}C(t-1)^{\alpha}$. 
\item[\rm (ii)] $\lim\limits_{\substack{t \to 1 \\ t\in(0,1)}} f_2(tz_0) = l \in \mathbb{C^*}$.
\end{enumerate}
Then $f_1(z)$ and $f_2(z)$ are algebraically independent over $\mathbb{C}(z)$.
\end{prop}

\begin{proof} 
Let us assume that $f_1$ and $f_2$ are algebraically dependent over $\mathbb{C}(z)$ and hence over $K(z)$ by Proposition \ref{prop: descente}. 
Since $f_1(z)$ and $f_2(z)$ belong to $\mathcal L_1(R,{\mathcal S})$, we can apply Theorem~\ref{thm: ind}. 
We obtain that there exist $a_1,a_2\in\mathbb{Z}$, not both equal to $0$, and a rational function $r(z)$ such that
\begin{equation}\label{imposs}
f_1(z)^{a_1}f_2(z)^{a_2}=r(z) \,.
\end{equation}
Thereby, it suffices to prove that \eqref{imposs} leads to a contradiction. Note that since $f_1$ and $f_2$ are transcendental, 
we have $a_1a_2\not=0$. Without loss of generality  we can assume that $a_1\geq 1$. 
Hence, we have
$$
f_1(tz_0)\underset{\substack{t \to 1 \\ t\in(0,1)}}{\sim}r(tz_0)^{1/a_1}\ell^{-a_2/a_1} \, ,
$$
which contradicts Assertion $(i)$. This ends the proof. 
\end{proof}

Let us give a first example of application of Proposition \ref{prop: sing2}. 

\begin{thm}
The functions 
$$
\sum_{n=0}^\infty\frac{(4n)!}{(2n)!n!^2}z^n\quad\textup{and}\quad\sum_{n=0}^\infty\left(\sum_{k=0}^n\binom{n}{k}^6\right)z^n
$$
are algebraically independent over $\mathbb{Q}(z)$.
\end{thm}

\begin{proof}
Using Stirling formula, 
we obtain that 
\begin{equation}\label{eq: 1/n}
\frac{(4n)!}{(2n)!n!^2} \underset{n\rightarrow\infty}{\sim}\frac{2^{6n}}{\pi n} \, ,
\end{equation}
while a result of McIntosh \cite{McIntosh} (stated in the proof of Theorem \ref{thm: appli1}) gives that  
\begin{equation}\label{eq: 5/2}
\sum_{k=0}^n\binom{n}{k}^6\underset{n\rightarrow\infty}{\sim}\frac{2^{6n}}{\sqrt{6(\pi n/2)^{5}}}\,.
\end{equation}
By Flajolet's asymptotic for algebraic functions (see \cite{Flajolet}), we know that if a power series $\sum_{n=0}^\infty a(n)z^n$ in $\mathbb{Q}[[z]]$ is algebraic over $\mathbb{Q}(z)$, then we have
$$
a(n)\underset{n\rightarrow\infty}{\sim}\frac{\alpha^nn^s}{\Gamma(s+1)}\sum_{i=0}^mC_i\omega_i^n,
$$
where $s\in\mathbb{Q}\setminus\mathbb{Z}_{<0}$ and $\alpha$, the $C_i$'s and the $\omega_i$'s are algebraic numbers. Since $\Gamma(-3/2)$ is a rational multiple of $\sqrt{\pi}$ and 
$\pi$ is a transcendental number, we obtain that  
$$
f_1(z):= \sum_{n=0}^\infty\frac{(4n)!}{(2n)!n!^2}z^n\quad\textup{and}\quad f_2(z) := \sum_{n=0}^\infty\left(\sum_{k=0}^n\binom{n}{k}^6\right)z^n
$$ 
are both transcendental over $\mathbb{Q}(z)$.
It follows that $f_1$ and $f_2$ have the same radius of convergence 
$\rho = (1/2)^6$. Furthermore, \eqref{eq: 1/n} shows that $f_1$ does satisfy Assumption (i) of Proposition~\ref{prop: sing2}. Indeed, 
\eqref{eq: 1/n} shows that $f_1$ has a logarithmic singularity at $\rho$ which is not compatible with an asymptotic of the form $C(z-\rho)^{\alpha}$.   
On the other hand,  \eqref{eq: 5/2} shows that $f_2$ satisfies Assumption (ii) of Proposition~\ref{prop: sing2}.  
As we will prove in Section \ref{sec: appli}, $f_1$ and $f_2$ both belong to $\mathfrak L(\mathcal P)$, 
we can apply Proposition~\ref{prop: sing2} to conclude. 
\end{proof}

\subsection{Singularities of $G$-functions and asymptotics of their coefficients}

We are mainly interested in $G$-functions so we will focus on elements in sets of the form $\mathcal L_1(\mathcal{O}_K,\mathcal S)$ (also denoted by  $\mathcal L(\mathcal{O}_K,\mathcal S)$), where $\mathcal{O}_K$ is the ring of integers of a number field $K$ assumed to be embedded in $\mathbb C$. In this case, it is well-known that $K$ is the fraction field of $\mathcal{O}_K$ which is a Dedekind domain. In this section, we briefly recall some background about the kind of singularities a $G$-function may have. 
As we will see, those are subject to severe restrictions. In particular, this explains why the same kind of 
asymptotics always come up when studying the coefficients of 
$G$-functions. 
\medskip

Let $f$ be a $G$-function and $\mathcal{L}$ be a non-zero 
differential operator in $\overline{\mathbb{Q}}[z,d/dz]$ of minimal order such that $\mathcal{L}\cdot f(z)=0$. Then it is known that $\mathcal{L}$ is a Fuchsian operator, 
that is all its singularities are regular. Furthermore,  
its exponents at each singularity are rational numbers. This follows from results of Chudnovsky \cite{Chudnovsky}, 
Katz \cite{Katz2}, and Andr\'e \cite{An00} (see \cite{An89} for a discussion).  In particular, these results have the following consequence. 
Let $\rho$ be a singularity of $\mathcal{L}$ at finite distance and consider a closed half-line 
$\Delta$ starting at $\rho$ and ending at infinity. Then there is a simply connected open set 
$U\supset\{0,\rho\}$ such that $f$ admits an analytic continuation to $V:=U\setminus\Delta$ (again 
denoted $f$)   which is annihilated by $\mathcal{L}$. In a neighborhood $W$ of $\rho$ in $V$, 
there exist rational numbers $\lambda_1,\dots,\lambda_s$, natural integers $k_1,\dots,k_s$ and functions $f_{i,k}(z)$ in $\mathbb{C}\{z-\rho\}$ such that
\begin{equation}\label{eq: decg}
f(z)=\sum_{i=1}^s\sum_{k=0}^{k_i}(z-\rho)^{\lambda_i}\log(z-\rho)^kf_{i,k}(z),
\end{equation}
and where $\lambda_i-\lambda_j\in\mathbb{Z}$ implies that $\lambda_i=\lambda_j$. 
By grouping terms with $\lambda_i=\lambda_j$, we may assume that if $\lambda_i-\lambda_j$ is an integer, then $i=j$. 

The following result shows that  a $G$-function that does not belong to the set $\mathfrak W$ should have  
a decomposition of a very restricted form on its circle of convergence.  
Roughly speaking, this means that transcendental $G$-functions usually tend to belong to $\mathfrak W$. 

\begin{prop}\label{propo singularity}
Let $f$ be a $G$-function and let $\rho$ be a singularity of $f$.   
Then  there is a positive integer $m$ such that  $f^m$ has an analytic continuation to a neighborhood of $\rho$ which is meromorphic at $\rho$
if, and only if, in any decomposition of the form \eqref{eq: decg}, we have $s=1$ and $k_1=0$, that is $f(z)=(z-\rho)^\lambda g(z)$ for $z\in W$,
 where $\lambda\in\mathbb{Q}$ and $g(z)\in\mathbb{C}\{z-\rho\}$.
\end{prop}

\begin{proof}
Let $f$ be a $G$-function, $\rho$ be a singularity of $f$, and let us fix a decomposition for $f$ of the form \eqref{eq: decg}. 
Assume that there is a positive integer $m$ such that $f^m$ has a meromorphic continuation $h$ to a neighborhood $V_\rho$ of $\rho$. Let $(i,k)$ be a pair such that
\begin{equation}\label{mini}
\lambda_i+\mathrm{ord}_{\rho}f_{i,k}=\min\{\lambda_j+\mathrm{ord}_{\rho}f_{j,\ell}\,\mid\,1\leq j\leq s,\,0\leq\ell\leq k_j\} \, .
\end{equation}
Observe that $i$ is unique because, if $\lambda_i+\mathrm{ord}_{\rho}f_{i,k}=\lambda_j+\mathrm{ord}_{\rho}f_{j,\ell}$, 
then $\lambda_i-\lambda_j$ is an integer so that by assumption $i=j$. 
Let $\kappa$ be the greatest $\ell\in\{0,\dots,k_i\}$ such that
$$
\lambda_i+\mathrm{ord}_{\rho}f_{i,\ell}=\lambda_i+\mathrm{ord}_{\rho}f_{i,k}\, .
$$
For $z$ in $W$ we set
$$
w(z):=f(z)^m-(z-\rho)^{m\lambda_i}\log(z-\rho)^{m\kappa}f_{i,\kappa}(z)^m\, .
$$
Hence, as $z$ tends to $\rho$ in $W$, we have
$$
\frac{w(z)}{(z-\rho)^{m\lambda_i}\log(z-\rho)^{m\kappa}f_{i,\kappa}(z)^{m}}\longrightarrow 0\, ,
$$
so that
$$
\frac{f(z)^m}{(z-\rho)^{m\lambda_i}\log(z-\rho)^{m\kappa}f_{i,\kappa}(z)^{m}}\longrightarrow 1\, ,
$$
which shows that $\kappa=0$ and $\lambda_i\in\frac{1}{m}\mathbb{Z}$. 
In particular, this shows that the pair $(i,k)$ is uniquely defined by \eqref{mini} and that $w(z)$ admits a meromorphic continuation 
to a neighborhood of $\rho$.

\medskip

We finish the proof by contradiction while assuming that $s\geq 2$ or $k_1\geq 1$. Hence, there is a pair $(i',k')$ which minimizes $\lambda_{i'}+\mathrm{ord}_\rho f_{i',k'}$ while satisfying 
\begin{equation}\label{mini2}
\lambda_{i}+\mathrm{ord}_\rho f_{i,0}<\lambda_{i'}+\mathrm{ord}_\rho f_{i',k'}.
\end{equation}
Again, one can show that $i'$ is unique. Let $\kappa'$ be the greatest $\ell\in\{0,\dots,k_{i'}\}$ such that
$$
\lambda_{i'}+\mathrm{ord}_{\rho}f_{i',\ell}=\lambda_{i'}+\mathrm{ord}_{\rho}f_{i',k'}.
$$
For $z$ in $W$ we set
$$
v(z):=w(z)-(z-\rho)^{(m-1)\lambda_i+\lambda_{i'}}\log(z-\rho)^{\kappa'}f_{i,0}(z)^{m-1}f_{i',\kappa'}(z).
$$
Hence, as $z$ tends to $\rho$ in $W$, we have
$$
\frac{v(z)}{(z-\rho)^{(m-1)\lambda_i+\lambda_{i'}}\log(z-\rho)^{\kappa'}f_{i,0}(z)^{m-1}f_{i',\kappa'}(z)}\longrightarrow 0\, ,
$$
so that
$$
\frac{w(z)}{(z-\rho)^{(m-1)\lambda_i+\lambda_{i'}}\log(z-\rho)^{\kappa'}f_{i,0}(z)^{m-1}f_{i',\kappa'}(z)}\longrightarrow 1\, .
$$
Since $w(z)$ has a meromorphic continuation to a neighborhood of $\rho$, this shows that $\kappa'=0$ and $(m-1)\lambda_i+\lambda_{i'}\in\mathbb{Z}$. 
We have $m\lambda_i\in\mathbb{Z}$ so $\lambda_i-\lambda_{i'}$ is an integer and $i=i'$. Thus $(i',\kappa')=(i,0)$ which contradicts \eqref{mini2}. This ends the proof. 
\end{proof}

\subsection{$G$-functions with integer coefficients}

In this section we focus on $G$-functions with integer coefficients. We first introduce the following set of analytic functions. 

 \begin{defn}{\rm 
Let $\mathfrak G$ denote the set of all analytic functions $f(z)$ in $\mathbb{C}\{z\}$ satisfying the following conditions. 
\begin{itemize}

\medskip

\item[{\rm (i)}] $f(z)$ satisfies a non-trivial linear differential equation with coefficients in $\mathbb Q[z]$.   

\medskip

\item[{\rm (ii)}] $f(z)$ belongs to $\mathbb Z[[z]]$. 

\end{itemize}}
\end{defn}

We observe that elements of $\mathfrak G$ are $G$-functions. The transcendental elements of $\mathfrak G$ have specific singularities.

\begin{prop}\label{prop: trans}
Every transcendental $f$ in $\mathfrak G$ has a singularity $\beta\in\mathbb{C}$ with $|\beta|<1$ such that no non-zero power of $f$ admits a meromorphic continuation at $\beta$.
\end{prop}

\begin{proof}
Let argue by contradiction. Let $f$ be an element of $\mathfrak G$ such that, for every complex numbers $\beta$ with $|\beta|<1$, there is a positive integer $n=n(\beta)$ such that $f^n$ admits a meromorphic continuation at $\beta$. Since $f$ is a $G$-function, it has only finitely many singularities and they all are at algebraic points. It implies that there exists a polynomial $P(z)$ in $\mathbb{Z}[z]$ and a positive integer $N$ such that $g(z):=P(z)f(z)^N$ is holomorphic in the open unit disk. Hence $g$ is a power series with integer coefficients such that $\rho_g\geq 1$. By the P\`olya-Carlson theorem, $g$ is either a rational fraction or admits the unit circle as a natural boundary. Since $g$ has only finitely many singularities, we obtain that $g$ is a rational fraction and $f$ is algebraic, which is a contradiction.
\end{proof}

We have the following generalization of Theorem \ref{thm: 2nn} concerning algebraic independence of Hadamard powers 
of elements of $\mathfrak G$. 

\begin{thm}\label{thm: had}
 Let $f(z):=\sum_{n=0}^{\infty}a(n)z^n$ be a transcendental function in  $\mathcal L(\mathcal S)\cap\mathfrak G$ and such that 
 $$
 a(n) \geq 0 \quad\mbox{and} \quad a(n)=O\left(\frac{\rho_f^{-n}}{n}\right) \, \cdot
 $$  
 Then the functions $f_1:=f,f_2:=f\odot f,f_3:=f\odot f \odot f,\ldots$ are algebraically independent over $\mathbb C(z)$. 
\end{thm}

\begin{proof}
By Proposition \ref{prop: trans}, the radius of convergence of $f(z)$ satisfies $0<\rho_f<1$. It follows that all the $f_r$'s have distinct radius of convergence since  $\rho_{f_r} = \rho_f^r$. On the other hand, Remark \ref{rem: W} 
implies that each $f_r$ belongs to $\mathfrak W$. We can thus apply Proposition \ref{prop: sing1} to conclude the proof. 
\end{proof}

 In all previous applications of Theorem \ref{thm: ind}, we used some knowledge about asymptotics of coefficients and/or 
 about the singularities  of the functions involved. We give here a general result that does not require any {\it a priori} knowledge of this kind. It applies to any transcendental element of $\mathfrak G$ which satisfies some Lucas-type congruences. 

\begin{prop}\label{prop: sing3} Let $f(z)$ be a transcendental function in  $\mathcal L(\mathcal S)\cap\mathfrak G$.   
Then the following hold. 
\begin{itemize}

\item[{\rm (i)}] Let $\lambda_1,\ldots,\lambda_n$ be non-zero rational numbers with distinct absolute values. Then the series $f(\lambda_1z),\ldots,f(\lambda_nz)$ are algebraically independent over $\mathbb Q(z)$. 

\item[\rm (ii)] The family $\left\{f(z),f(z^2),f(z^3),\ldots\right\}$ is algebraically independent over $\mathbb Q(z)$.
\end{itemize}
\end{prop}

\begin{proof} Let us first prove Assertion (i). We first note that, for all but finitely many primes $p$ in $\mathcal S$, the rational numbers 
$\lambda_1,\ldots,\lambda_n$ belong to $\mathbb Z_{(p)}$. We can thus replace $\mathcal S$ by an infinite subset $\mathcal S'$ 
such that this holds for all primes in $\mathcal S'$. Set $g_i(z):=f(\lambda_iz)$ for every $i$ in $\{1,\ldots,n\}$.  Taking $p$ in 
$\mathcal S'$ and using that $f(z)$ belongs to $\mathcal L(\mathcal S)$, we obtain that there exist a rational fraction $A(z)$ and a positive integer $k$ 
such that $f(z) \equiv A(z) f(z^{p^k}) \bmod p\mathbb Z_{(p)}[[z]]$. This gives: 
\begin{align*}
g_i(z)\equiv f(\lambda_iz)&\equiv A(\lambda_iz) f\big((\lambda_iz)^{p^k}\big) \bmod p\mathbb Z_{(p)}[[z]] \\
& \equiv  A(\lambda_iz) f\big(\lambda_iz^{p^k}\big) \bmod p\mathbb Z_{(p)}[[z]]  \\
&\equiv A(\lambda_iz)g_i\big(z^{p^k}\big) \bmod p\mathbb Z_{(p)}[[z]] . 
\end{align*}
We thus have that $g_1(z),\ldots,g_n(z)$ all belong to $\mathcal L(\mathcal S')$. Let us assume by contradiction that they 
are algebraically dependent over $\mathbb Q(z)$. Then  Theorem \ref{thm: ind} ensures the existence of integers 
$a_1,\ldots,a_n$, not all zero, and of a rational fraction $r(z)$ in $\mathbb{Q}(z)$ such that 
\begin{equation}\label{eq: gi}
g_1(z)^{a_1}\cdots g_n(z)^{a_n} = r(x).
\end{equation}
Without any loss of generality, we can assume that $\vert \lambda_1\vert < \cdots < \vert \lambda_n\vert$. 
Let $j$ be the largest index for which $a_j\neq 0$.  

\medskip

Let $\alpha$ denote the infimum of all $|\beta|$, where 
$\beta$ ranges over all complex numbers such that, for all $n\ge 1$, $f(z)^n$ has no meromorphic continuation at  $\beta$.    
By Proposition \ref{prop: trans}, we have that $0<\alpha<1$.  We pick now a complex number $\beta$ such that,  
for every positive integer $n$, $f(z)^n$ is not meromorphic at $\beta$ and such that $\vert (\lambda_i/\lambda_j)\beta\vert< \alpha$ for all $i < j$.  
Then Equation~\eqref{eq: gi} 
can be rewritten as 
$$
g_j(z)^{a_j} = r(z)\prod_{i=1}^{j-1} g_i(z)^{-a_i}.
$$  
We assume that $a_j>0$, otherwise we would write $g_j(z)^{-a_j} = r(z)^{-1}\prod_{i=1}^{j-1} g_i(z)^{a_i}$. 
Our choice of $\beta$ ensures, for every $i=1,\ldots,j-1$, the existence of a positive integer $n_i$ such that $g_i(z)^{n_i}$ is meromorphic at 
$z=\beta/\lambda_j$. Taking $n:= \gcd (n_1,\ldots,n_{j-1})$, we obtain that $g_j(z)^{na_j}$ is meromorphic at $\beta/\lambda_j$. 
This provides a contradiction since no power of $f(z)$ is meromorphic at $\beta$. This proves Assertion $(i)$. 
\medskip

A similar argument handles Assertion (ii). In that case, we have to choose $j$ to be the smallest index for which $a_j\neq 0$ 
and  $\beta$ to be such that for every positive integer $n$, $f(z)^n$ is not meromorphic at $\beta$ and such that$\vert \beta^{i/j}\vert< \alpha$ 
for all $i>j$. The rest of the proof remains unchanged. 
\end{proof}


\section{A family of multivariate generalized hypergeometric series and their $p$-adic properties}\label{sec: mgh}

In this section, we introduce a family of multivariate generalized hypergeometric series denoted by $F_{\mathbf{u},\mathbf{v}}({\bf x})$ and already    
mentioned in Section \ref{sec: intro}. Our aim is to deduce from 
$p$-adic properties of their coefficients an efficient condition on the parameters ${\bf u}$ and ${\bf v}$ that forces 
$F_{\mathbf{u},\mathbf{v}}({\bf x})$ to belong to $\mathfrak{L}_d(\mathcal{S})$ for an infinite set of primes $\mathcal{S}$.   
A key idea is then that classical examples of sequences with the $p$-Lucas property 
arise from specializations of particular multivariate generalized hypergeometric series. 
As we will see in Section \ref{sec: appli}, 
various specializations of the parameters or of the variables of functions of type $F_{\mathbf{u},\mathbf{v}}$ 
lead us to prove that interesting families of $G$-functions belong to $\mathfrak{L}_1(\mathcal{S})$. 
This includes generating series associated with factorial ratios, with some sums and products of binomials, and 
generalized hypergeometric series. 
In this direction, we stress that Propositions \ref{prop: fuv}, \ref{theo Lucas Hypergeom}, and~\ref{prop: spe}  allow to recover most examples in the literature 
of sequences known  to satisfy  $p$-Lucas congruences.  They also  provide a lot of new examples.

\medskip

Let us give now the definition of our family of multivariate generalized hypergeometric series. 
For every tuples $\mathbf{u}=((\alpha_1,\mathbf{e}_1),\dots,(\alpha_r,\mathbf{e}_r))$ and 
$\mathbf{v}=((\beta_1,\mathbf{f}_1),\dots,(\beta_s,\mathbf{f}_s))$ of elements in $\mathbb{Q} \times\mathbb{N}^d$, and every vector ${\bf n}$ in $\mathbb N^d$, 
we set  
$$
\mathcal{Q}_{\mathbf{u},\mathbf{v}}(\mathbf{n}) :=
\frac{(\alpha_1)_{\mathbf{e}_1\cdot\mathbf{n}}\cdots(\alpha_r)_{\mathbf{e}_r\cdot\mathbf{n}}}
{(\beta_1)_{\mathbf{f}_1\cdot\mathbf{n}}\cdots(\beta_s)_{\mathbf{f}_s\cdot\mathbf{n}}} \, ,
$$
where $(x)_n:=x(x+1)\cdots(x+n-1)$ if $n\geq 1$, and $(x)_0=1$ denote the Pochhammer symbol. 
Then the multivariate power series $F_{\mathbf{u},\mathbf{v}}({\bf x})$ is defined by:
$$
F_{\mathbf{u},\mathbf{v}}({\bf x}) := \sum_{{\bf n}\in \mathbb N^d} \mathcal{Q}_{\mathbf{u},\mathbf{v}}(\mathbf{n}) {\bf x}^{\bf n} \in \mathbb Q[[{\bf x}]]\, .
$$
We also set 
$$
\mathcal H_d := \Bigg\{F_{\mathbf{u},\mathbf{v}}({\bf x}) \,:\, \alpha_i,\beta_i\in\mathbb{Q}\cap(0,1] \mbox{ and }   
\sum_{i=1}^r\mathbf{e}_i=\sum_{j=1}^s\mathbf{f}_j \Bigg\} \, .
$$
The class of power series $\mathcal H:= \cup_{d\geq 1} \mathcal H_d$ contains of course the generalized hypergeometric series, but  also the multivariate 
Apell and Lauricella hypergeometric series.

\subsection{Lucas-type congruences for elements of $\mathcal H_d$} 

Our aim is to provide an efficient condition for $F_{\mathbf{u},\mathbf{v}}({\bf x})$ to satisfy the $p^k$-Lucas property for infinitely many primes $p$. 
In this direction, our main result is Proposition \ref{theo Lucas Hypergeom}. However, the statement of this result is a bit technical and 
requires some notation. We thus choose to first state it in a particular case, as Proposition \ref{prop: fuv}, that avoids most of these technical aspects.

Our first condition is based on the analytical properties of the following simple step function defined from $\mathbb{R}^d$ to $\mathbb{Z}$:
$$
\xi_1(\mathbf{x}):=\sum_{i=1}^r\lfloor\mathbf{e}_i\cdot\mathbf{x}-\alpha_i\rfloor-\sum_{j=1}^s\lfloor\mathbf{f}_j\cdot\mathbf{x}-\beta_j\rfloor+r-s \,.
$$
Note that if $F_{\mathbf{u},\mathbf{v}}(\mathbf{x})$ belongs to $\mathcal{H}_d$, then the behavior of $\xi_1$ on $\mathbb{R}^d$ is determined by its values on $[0,1)^d$ because we have
$$
\xi_1(\mathbf{x})=\xi_1\big(\{\mathbf{x}\}\big)+\Bigg(\sum_{i=1}^r\mathbf{e}_i-\sum_{j=1}^s\mathbf{f}_j\Bigg)\cdot\lfloor\mathbf{x}\rfloor=\xi_1\big(\{\mathbf{x}\}\big) \,.
$$
Let $\mathcal{D}_{\mathbf{u},\mathbf{v}}^1$ be the semi-algebraic set defined by:
$$
\mathcal{D}_{\mathbf{u},\mathbf{v}}^1:=\big\{\mathbf{x}\in[0,1)^d\,: \,\mathbf{e}\cdot\mathbf{x}\geq\alpha
\textup{ for some coordinate $(\alpha,\mathbf{e})$ of either $\mathbf{u}$ or $\mathbf{v}$}\big\} \,.
$$
Note that outside $\mathcal{D}_{\mathbf{u},\mathbf{v}}^1$ the step function $\xi_1$ trivially vanishes when the $\alpha_i$'s and $\beta_i$'s belong to $(0,1]$. 
 We need a last notation before stating our first result. We let $d_{\boldsymbol{\alpha},\boldsymbol{\beta}}$ stand for the least common 
 multiple of the  denominators of the rational numbers $\alpha_1,\ldots,\alpha_r,\beta_1,\ldots,\beta_s$, written in lowest form. 
 Then our first proposition reads as follows.

\medskip

\begin{prop}\label{prop: fuv} 
Let $F_{\mathbf{u},\mathbf{v}}({\bf x})$ be in $\mathcal H_d$ and 
let us assume that $\xi_1(\mathbf{x})\geq 1$ for all $\mathbf{x}$ in $\mathcal{D}_{\mathbf{u},\mathbf{v}}^1$. 
Then $F_{\mathbf{u},\mathbf{v}}({\bf x})$ belongs to $\mathfrak L_d(\mathcal{S})$, where 
$$
\mathcal{S}:=\{p \in \mathcal P \,:\, p\equiv 1\mod d_{\boldsymbol{\alpha},\boldsymbol{\beta}} \} \, . 
$$
More precisely,  $F_{\mathbf{u},\mathbf{v}}({\bf x})$ satisfies the $p$-Lucas property for all primes $p$ in $\mathcal{S}$. 
\end{prop}

Let us illustrate Proposition \ref{prop: fuv} with the following example. Set ${\bf u}:=((1,(2,1),(1,(1,1)))$ and 
${\bf v}:=((1,(1,0)),(1,(1,0)),(1,(1,0)),(1,(0,1)) ,(1,(0,1)))$. 
Then  
\begin{align*}
F_{{\bf u},{\bf v}}(x,y) &= \sum_{(n,m)\in\mathbb N^2} \frac{(1)_{2n+m}(1)_{n+m}}{(1)_n^3(1)_m^2} x^ny^m  \\ 
&=  \sum_{(n,m)\in\mathbb N^2} \frac{(2n+m)!(n+m)!}{n!^3m!^2} x^ny^m \, . 
\end{align*}
For all $x_1$ and $x_2$ in $[0,1)$, we have
\begin{align*}
\xi_1(x_1,x_2)&=\lfloor 2x_1+x_2-1\rfloor+\lfloor x_1+x_2-1\rfloor-3\lfloor x_1-1\rfloor-2\lfloor x_2-1\rfloor-3\\
&=\lfloor 2x_1+x_2\rfloor+\lfloor x_1+x_2\rfloor-3\lfloor x_1\rfloor-2\lfloor x_2\rfloor\\
&=\lfloor 2x_1+x_2\rfloor+\lfloor x_1+x_2\rfloor \, .
\end{align*}
Clearly, we have that $\sum_{i=1}^r\mathbf{e}_i=\sum_{j=1}^s\mathbf{f}_j$, $\boldsymbol{\alpha}$ and 
$\boldsymbol{\beta}$ are tuples of elements in $(0,1]$ 
and $d_{\boldsymbol{\alpha},\boldsymbol{\beta}}=1$. Furthermore, we have
$$
\mathcal{D}_{\mathbf{u},\mathbf{v}}^1=\Big\{(x_1,x_2)\in[0,1)^2\,: \,2x_1+x_2\geq 1\textup{ or }x_1+x_2\geq 1\Big\} \,,
$$
so that  $\xi_1(x_1,x_2)\geq 1$, for all $(x_1,x_2)$ in $\mathcal{D}_{\mathbf{u},\mathbf{v}}^1$. 
Hence we infer from Proposition \ref{prop: fuv}  that $F_{\mathbf{u},\mathbf{v}}$ satisfies the $p$-Lucas property for all prime numbers $p$.
\medskip

We are now going to generalize Proposition \ref{prop: fuv} in order to possibly enlarge the set of primes 
$p$ for which an element of $\mathcal H_d$ does satisfy the $p^k$-Lucas property.  
Towards that goal, we consider the following generalizations of $\xi_1$ and $\mathcal{D}_{\mathbf{u},\mathbf{v}}^1$. 
For all $x\in\mathbb{R}$, we denote by $\langle x\rangle$ the unique element in $(0,1]$ such that $x-\langle x\rangle$ is an integer. 
In other words, we have $\langle x\rangle=1-\{1-x\}$, or equivalently:
$$
\langle x\rangle=\begin{cases}
\{x\}&\textup{ if $x\notin\mathbb{Z}$,}\\
1&\textup{ otherwise}.
\end{cases}.
$$
Let $a$ denote an element of $\{1,\dots,d_{\boldsymbol{\alpha},\boldsymbol{\beta}}\}$ which is coprime to 
$d_{\boldsymbol{\alpha},\boldsymbol{\beta}}$. Then we consider the following step function from $\mathbb{R}^d$ to $\mathbb{Z}$: 
$$
\xi_a(\mathbf{x}) :=\sum_{i=1}^r\big\lfloor\mathbf{e}_i\cdot\mathbf{x}-\langle a\alpha_i\rangle \big\rfloor-\sum_{j=1}^s
\big\lfloor\mathbf{f}_j\cdot\mathbf{x} -\langle a\beta_j\rangle\big\rfloor+r-s \,.
$$
Furthermore, we consider the semi-algebraic set
$$
\mathcal{D}_{\mathbf{u},\mathbf{v}}^a:=\Big\{\mathbf{x}\in[0,1)^d\,:\,\mathbf{e}\cdot\mathbf{x}\geq\langle a\alpha\rangle\textup
{ for some coordinate  $(\alpha,\mathbf{e})$ of either $\mathbf{u}$ or $\mathbf{v}$}\Big\} \,.
$$
If $\alpha$ belongs to $\mathbb{Q}\cap(0,1]$, then we have $\langle \alpha\rangle=\alpha$, so that we recover our initial definition of $\xi_1$ when $a=1$. 
Our generalization of Proposition \ref{prop: fuv} reads as follows.

\begin{prop}\label{theo Lucas Hypergeom} Let $F_{\mathbf{u},\mathbf{v}}({\bf x})$ be in $\mathcal H_d$. 
Let us assume that there exists a subset $A$ of $\{1,\dots,d_{\boldsymbol{\alpha},\boldsymbol{\beta}}\}$ such that the following hold. 

\medskip

\begin{itemize}
\item[{\rm (i)}] $\{a\mod d_{\boldsymbol{\alpha},\boldsymbol{\beta}}\,:\,  a\in A\}$ is a subgroup of 
$(\mathbb{Z}/d_{\boldsymbol{\alpha},\boldsymbol{\beta}}\mathbb{Z})^\times$.

\medskip

\item[{\rm (ii)}] For all $a$ in $A$ and all $\mathbf{x}$ in $\mathcal{D}_{\mathbf{u},\mathbf{v}}^a$, we have $\xi_a(\mathbf{x})\geq 1$.
\end{itemize}
\medskip

Set $k:=\mbox{\rm Card } A$. Then $F_{\mathbf{u},\mathbf{v}}({\bf x})$ belongs to $\mathfrak L_d(\mathcal{S})$, where  
$$
\mathcal{S}:= \left \{ p \in \mathcal P \,:\, p\equiv a\mod d_{\boldsymbol{\alpha},\boldsymbol{\beta}} 
\mbox{ for some } a\in A, \mbox{ and } p> d_{\boldsymbol{\alpha},\boldsymbol{\beta}} \right\} \, . 
$$
More precisely,  $F_{\mathbf{u},\mathbf{v}}({\bf x})$ satisfies the $p^k$-Lucas property for all primes $p$ in $\mathcal{S}$. 
\end{prop}

\begin{rem}{\rm Proposition \ref{prop: fuv} is simply obtained by  taking $A=\{1\}$.  
It is worth mentioning that in order to apply Proposition \ref{theo Lucas Hypergeom}, 
one always has to check that $\xi_1(\mathbf{x})\geq 1$ for $\mathbf{x}$ in $\mathcal{D}_{\mathbf{u},\mathbf{v}}^1$.  
Hence Proposition \ref{prop: fuv} applies too. This explains why Proposition \ref{prop: fuv} 
is sufficient for all our applications concerning algebraic independence. Indeed, for using our algebraic independence criterion, 
one only needs to work with infinitely many primes, so that we do not really care about 
a precise description of the set $\mathcal{S}$.}
\end{rem}

\medskip

Let us illustrate Proposition \ref{theo Lucas Hypergeom} with the following one-variable hypergeometric series. We set
$$
F(x):=\sum_{n=0}^\infty\frac{(1/5)_n^2}{(2/7)_n(1)_n}x^n \, .
$$
For all $x$ in $[0,1)$, we have 
$$
\xi_1(x)=2\lfloor x-1/5\rfloor-\lfloor x-2/7\rfloor+1\quad\textup{and}\quad\xi_6(x)=2\lfloor x-1/5\rfloor-\lfloor x-5/7\rfloor+1.
$$
Furthermore, we have
$$
\mathcal{D}^1=\mathcal{D}^6=\big\{x\in[0,1)\,:\,x\geq 1/5\big\}\, .
$$
We thus deduce  that both $\xi_1(x)$ and $\xi_6(x)$ are greater than or equal to $1$ for all $x$ in $\mathcal{D}^1$. 
Since 
$d_{\boldsymbol{\alpha},\boldsymbol{\beta}}=35$ and $6^2\equiv 1\mod 35$, 
we can apply Proposition \ref{theo Lucas Hypergeom} with $A=\{1\}$ and $A=\{1,6\}$. 
We thus deduce that for almost all primes $p\equiv 1\mod 35$, $F(z)$ satisfies the $p$-Lucas property, while, 
for almost all primes $p\equiv 6\mod 35$, $F(z)$ satisfies the $p^2$-Lucas property.

\medskip

In order to transfer the $p^k$-Lucas property from multivariate series of type $F_{{\bf u},{\bf v}}$ to 
one-variable formal power series, we will use the following useful complement to Proposition \ref{theo Lucas Hypergeom}. 

\begin{prop}\label{prop: spe} Let $F_{\mathbf{u},\mathbf{v}}({\bf x})$ be in $\mathcal H_d$.  
We keep all assumptions and notation of Proposition \ref{theo Lucas Hypergeom}.   
Set 
$$
\mathcal N := \big\{ \mathbf{n}\in\mathbb{N}^d \,:\, \forall a\in A, \forall \mathbf{x}\in [0,1)^d \mbox{ with } \mathbf{n}\cdot\mathbf{x}\geq 1,  \mbox{ one has } 
\xi_a(\mathbf{x})\geq 1\big\}\, .
$$
Let     $\mathbf{n}=(n_1,\dots,n_d)$ be in $\mathcal N$ and $(b_1,\dots,b_d)$ be a vector of non-zero rational numbers.  
Then $F_{\mathbf{u},\mathbf{v}}(b_1x^{n_1},\dots,b_dx^{n_d})$ belongs to $\mathfrak L_1(\mathcal{S}')$, where 
$$
\mathcal{S}' := \Big\{ p\in \mathcal{S} \,:\, (b_1,\dots,b_d)\in\mathbb{Z}_{(p)}^d\Big\} \, .
$$
\end{prop}

Let us illustrate this result with the example given just after Proposition \ref{prop: fuv}, that is with the function 
$$
F_{{\bf u},{\bf v}}(x,y)=\sum_{(n,m)\in\mathbb{N}^2}\frac{(2n+m)!(n+m)!}{n!^3m!^2}x^ny^m \, .
$$
We consider the specialization given by $\mathbf{n}=(1,1)$ and $b_1=b_2=1$. Then Proposition \ref{prop: spe} applies with 
$A=\{1\}$ and $d_{\boldsymbol{\alpha},\boldsymbol{\beta}}=1$, because we already observed that $\xi_1(x_1,x_2)\geq 1$ for all $(x_1,x_2)$ in $[0,1)^2$ 
satisfying $x_1+x_2\geq 1$. 
A small computation shows that we obtain yet another proof of Gessel's result stating that the Ap\'ery sequence
$$
\left(\sum_{k=0}^n\binom{n}{k}^2\binom{n+k}{k}\right)_{n\geq 0}
$$
satisfies the $p$-Lucas property for all primes $p$. 

Applying Proposition~\ref{prop: spe} to the same function but with the specialization given 
by $\mathbf{n}=(2,1)$, $b_1=-1$ and $b_2=2$, we get that 
$$
F_{{\bf u},{\bf v}}(-x^2,2x) =\sum_{n=0}^\infty\left(\sum_{k=0}^{\lfloor n/2\rfloor}(-1)^k2^{n-2k}\binom{n}{k}\binom{n-k}{n-2k}^2\right)x^n
$$
also satisfies  the $p$-Lucas property for all primes $p$.

\subsection{Proofs of Propositions \ref{theo Lucas Hypergeom} and \ref{prop: spe}}

We first introduce some $p$-adic tools that will be useful for proving these results.

\subsubsection{$p$-adic tools for the study of $\mathcal{Q}_{\mathbf{u},\mathbf{v}}$}\label{Section $p$-adic tools}

By Definition \ref{def: pkl} and Proposition \ref{prop: plucas}, proving that $F_{{\bf u},{\bf v}}({\bf x})$ belongs to $\mathfrak L_d(\mathcal{S})$ is the same as proving the $p^k$-Lucas 
congruence for the multivariate sequence $\mathcal{Q}_{\mathbf{u},\mathbf{v}}$, that is:
$$
\mathcal{Q}_{\mathbf{u},\mathbf{v}} ({\bf n}+{\bf m}p^k) \equiv \mathcal{Q}_{\mathbf{u},\mathbf{v}} ({\bf n})\mathcal{Q}_{\mathbf{u},\mathbf{v}} ({\bf m}) 
\bmod p\,,
$$
for all ${\bf n}$ and ${\bf m}$ in $\mathbb N^d$ and all primes $p$ in $S$. To that purpose, we shall use some classical techniques associated with 
Pochhammer symbols to compute the $p$-adic valuation 
of $\mathcal{Q}_{\mathbf{u},\mathbf{v}}(\mathbf{n})$. These techniques are reminiscent of works of Dwork \cite{Dwork}, Katz \cite{Katz}, Christol \cite{Christol} and, 
more recently, Delaygue, Rivoal, and Roques \cite{Delaygue4}. The following discussion should highlight the role played by the functions $\xi_a$ 
in the proof of Proposition \ref{theo Lucas Hypergeom}.

\medskip

Christol \cite{Christol} gave a useful formula to compute the $p$-adic valuation of $(a)_n$ where $n$ is a natural integer, and later Delaygue, Rivoal, and Roques 
\cite{Delaygue4} reformulated this result in order to deduce a formula which may seem closer to the classical one of Legendre:
$$
v_p(n!)=\sum_{\ell=1}^\infty\left\lfloor\frac{n}{p^\ell}\right\rfloor.
$$
We recall now this reformulation. We refer the reader to \cite[Section~3]{Delaygue4} for more details on the following definitions. 
For all primes $p$ and all $\alpha$ in $\mathbb{Z}_{p}\cap \mathbb Q$,  there is a unique element in $\mathbb{Z}_{p}$, denoted by $\mathfrak{D}_p(\alpha)$, 
such that  $p\mathfrak{D}_p(\alpha)-\alpha$ belongs to $\{0,\dots,p-1\}$. Hence, if $\kappa_0$ denotes the first digit in the Hensel expansion of $-\alpha$, then we have
$$
\mathfrak{D}_p(\alpha)=\frac{\alpha+\kappa_0}{p}\, \cdot
$$
The map $\mathfrak{D}_p$ was introduced by Dwork in \cite{Dwork}, where it is simply denoted by $a\mapsto a'$. 
For all positive integers $\ell$, we denote by $\delta_{p,\ell}(\alpha,\cdot)$ the step function defined, for all real numbers $x$, by
$$
\delta_{p,\ell}(\alpha,x)=\left\lfloor x-\mathfrak{D}_p^\ell(\alpha)-\frac{\lfloor 1-\alpha\rfloor}{p^\ell}\right\rfloor+1\,.
$$
By Proposition 4 in \cite{Delaygue4}, if $\alpha$ does not belong to $\mathbb{Z}_{\leq 0}$, then we get that
$$
v_p\big((\alpha)_{\mathbf{e}\cdot\mathbf{n}}\big)=\sum_{\ell=1}^\infty\delta_{p,\ell}
\left(\alpha,\frac{\mathbf{e}\cdot\mathbf{n}}{p^\ell}\right)=\sum_{\ell=1}^\infty
\delta_{p,\ell}\left(\alpha,\mathbf{e}\cdot\left\{\frac{\mathbf{n}}{p^\ell}\right\}\right)+\mathbf{e}\cdot\big(v_p(n_1!),\dots,v_p(n_d!)\big) \,.
$$
For all primes $p$ that does not divide $d_{\boldsymbol{\alpha},\boldsymbol{\beta}}$, we set
$$
\Delta^{p,\ell}_{\mathbf{u},\mathbf{v}}(\mathbf{x}):=
\sum_{i=1}^r\delta_{p,\ell}(\alpha_i,\mathbf{e}_i\cdot\mathbf{x})-\sum_{j=1}^s\delta_{p,\ell}(\beta_j,\mathbf{f}_j\cdot\mathbf{x}) \, ,
$$
so that
\begin{align*}
v_p\big(\mathcal{Q}_{\mathbf{u},\mathbf{v}}(\mathbf{n})\big)&=\sum_{\ell=1}^\infty\Delta^{p,\ell}_{\mathbf{u},\mathbf{v}}\left(\frac{\mathbf{n}}{p^\ell}\right)\\
&=\sum_{\ell=1}^\infty\Delta^{p,\ell}_{\mathbf{u},\mathbf{v}}\left(\left\{\frac{\mathbf{n}}{p^\ell}\right\}\right)
+\Bigg(\sum_{i=1}^r\mathbf{e}_i-\sum_{j=1}^s\mathbf{f}_j\Bigg)\cdot\big(v_p(n_1!),\dots,v_p(n_d!)\big) \,.
\end{align*}

This formula may seem complicated because the step functions $\Delta_{\mathbf{u},\mathbf{v}}^{p,\ell}$ \textit{a priori} highly depend 
on $p$ and $\ell$, but they actually only depend on the congruence class of $p^\ell$ modulo $d_{\boldsymbol{\alpha},\boldsymbol{\beta}}$. 
To be more precise, let $\boldsymbol{\alpha}$ and $\boldsymbol{\beta}$ be tuples of elements in $\mathbb{Q}\cap(0,1]$,  
$p>d_{\boldsymbol{\alpha},\boldsymbol{\beta}}$ be a prime, and let $a$ be such that $ap^\ell\equiv 1\mod d_{\boldsymbol{\alpha},\boldsymbol{\beta}}$. 
Then we have $\Delta_{\mathbf{u},\mathbf{v}}^{p,\ell}=\xi_a$. Indeed, if $\alpha$ is a rational number in $(0,1]$, then
we have $\lfloor 1-\alpha\rfloor=0$ and  Lemma 3 in \cite{Delaygue4} implies that $\mathfrak{D}_p^\ell(\alpha)=\langle a\alpha\rangle$, as expected. 
Let us denote by  $\iota(p^\ell)$ the unique element of $\{0,\dots,d_{\boldsymbol{\alpha},\boldsymbol{\beta}}\}$ 
such that $\iota(p^\ell)p^\ell\equiv 1\mod d_{\boldsymbol{\alpha},\boldsymbol{\beta}}$. 
We thus get the following formula:  
$$
v_p\big(\mathcal{Q}_{\mathbf{u},\mathbf{v}}(\mathbf{n})\big)=\sum_{\ell=1}^\infty\xi_{\iota(p^\ell)}\left(\frac{\mathbf{n}}{p^\ell}\right) \, ,
$$
which  is somewhat reminiscent of Legendre formula. 

\subsubsection{Proof of Proposition \ref{theo Lucas Hypergeom}} 
We are now ready to prove Propositions \ref{theo Lucas Hypergeom} and \ref{prop: spe}. 
Given a tuple of elements in $\mathbb{Z}_{(p)}\times\mathbb{N}^d$, say  
$\mathbf{w}=((\gamma_1,\mathbf{g}_1),\dots,(\gamma_t,\mathbf{g}_t))$, we set 
$$
\mathfrak{D}_p(\mathbf{w}):=\big(\big(\mathfrak{D}_p(\gamma_1),\mathbf{g}_1\big),\ldots,\big(\mathfrak{D}_p(\gamma_t),\mathbf{g}_t\big)\big).
$$
In order to prove Proposition \ref{theo Lucas Hypergeom}, we will need the following lemma.

\begin{lem}\label{Lemma DecompProd}
Let $p>d_{\boldsymbol{\alpha},\boldsymbol{\beta}}$ be a prime and let $\mathbf{a}$ in $\{0,\dots,p-1\}^d$ be such 
that $\mathbf{a}/p$ does not belong to $\mathcal{D}_{\mathbf{u},\mathbf{v}}^{\iota(p)}$. If $\sum_{i=1}^r\mathbf{e}_i=\sum_{j=1}^s\mathbf{f}_j$, then  we have
$$
\mathcal{Q}_{\mathbf{u},\mathbf{v}}(\mathbf{a}+\mathbf{m}p)\in\mathcal{Q}_{\mathbf{u},\mathbf{v}}(\mathbf{a})\mathcal{Q}_{\mathfrak{D}_p(\mathbf{u}),\mathfrak{D}_p(\mathbf{v})}(\mathbf{m})(1+p\mathbb{Z}_{(p)}) \, ,
$$
for all $\mathbf{m}$ in $\mathbb{N}^d$. 
\end{lem}

\begin{proof}
Let $(\alpha,\mathbf{e})$ be a coordinate of either the vector $\mathbf{u}$ or $\mathbf{v}$, and let 
$p>d_{\boldsymbol{\alpha},\boldsymbol{\beta}}$ be a prime. 
Hence we have $\alpha\in\mathbb{Z}_{(p)}$ and $\mathfrak{D}_p(\alpha)$ is well defined. 
Let $\mathbf{m}$ in $\mathbb{N}^d$ and $\mathbf{a}$ in $\{0,\dots,p-1\}^d$ be such that $\mathbf{a}/p$ does not belong to 
$\mathcal{D}_{\mathbf{u},\mathbf{v}}^{\iota(p)}$. We have

\begin{equation}\label{Eq PochProd 1}
(\alpha)_{\mathbf{e}\cdot(\mathbf{a}+\mathbf{m}p)}=\left(\prod_{i=0}^{\mathbf{e}\cdot\mathbf{m}p-1}(\alpha+i)\right)
\left(\prod_{i=0}^{\mathbf{e}\cdot\mathbf{a}-1}(\alpha+i+\mathbf{e}\cdot\mathbf{m}p)\right)\,.
\end{equation}
Furthermore, we also have
\begin{align*}
\prod_{i=0}^{\mathbf{e}\cdot\mathbf{m}p-1}(\alpha+i)&=\prod_{k=0}^{p-1}\prod_{j=0}^{\mathbf{e}\cdot\mathbf{m}-1}(\alpha+k+jp)\\
&=p^{\mathbf{e}\cdot\mathbf{m}}\big(\mathfrak{D}_p(\alpha)\big)_{\mathbf{e}\cdot\mathbf{m}}
\underset{k\neq p\mathfrak{D}_p(\alpha)-\alpha}{\prod_{k=0}^{p-1}}(\alpha+k)^{\mathbf{e}\cdot\mathbf{m}}
\prod_{j=0}^{\mathbf{e}\cdot\mathbf{m}-1}\left(1+\frac{j}{\alpha+k}p\right).
\end{align*}
If $k\neq p\mathfrak{D}_p(\alpha)-\alpha$, then $\alpha+k\in\mathbb{Z}_{(p)}^\times$ and we obtain that
$$
1+\frac{j}{\alpha+k}p\in 1+p\mathbb{Z}_{(p)}\, ,
$$
for all non-negative integers $j$. 
By Wilson's theorem, it follows that 
$$
\underset{k\neq p\mathfrak{D}_p(\alpha)-\alpha}{\prod_{k=0}^{p-1}}(\alpha+k)\equiv (p-1)!\equiv -1\mod p\mathbb{Z}_{(p)} \,.
$$
This leads to
\begin{equation}\label{Eq PochProd 1.1}
\prod_{i=0}^{\mathbf{e}\cdot\mathbf{m}p-1}(\alpha+i)\in (-p)^{\mathbf{e}\cdot\mathbf{m}}
\big(\mathfrak{D}_p(\alpha)\big)_{\mathbf{e}\cdot\mathbf{m}}(1+p\mathbb{Z}_{(p)}) \,.
\end{equation}
In addition, since $\mathbf{a}/p$ is not in $\mathcal{D}_{\mathbf{u},\mathbf{v}}^{\iota(p)}$, we have $\mathbf{e}\cdot\mathbf{a}/p< \langle\iota(p)\alpha\rangle$,  
that is $\mathbf{e}\cdot\mathbf{a}<p\mathfrak{D}_p(\alpha)$. Since $\alpha\in(0,1]$ we obtain that $\mathbf{e}\cdot\mathbf{a}-1<p\mathfrak{D}_p(\alpha)-\alpha$. 
In particular, for all  $i$ satisfying $0\leq i\leq \mathbf{e}\cdot\mathbf{a}-1$, we have $\alpha+i\in\mathbb{Z}_{(p)}^\times$. This gives: 
\begin{equation}\label{Eq PochProd 1.2}
\prod_{i=0}^{\mathbf{e}\cdot\mathbf{a}-1}(\alpha+i+\mathbf{e}\cdot\mathbf{m}p)
=(\alpha)_{\mathbf{e}\cdot\mathbf{a}}\prod_{i=0}^{\mathbf{e}\cdot\mathbf{a}-1}
\left(1+\frac{\mathbf{e}\cdot\mathbf{m}}{\alpha+i}p\right)\in(\alpha)_{\mathbf{e}\cdot\mathbf{a}}(1+p\mathbb{Z}_{(p)}) \,. 
\end{equation}
We then infer from  \eqref{Eq PochProd 1}, \eqref{Eq PochProd 1.1} and \eqref{Eq PochProd 1.2} that
\begin{equation}\label{Eq PochProd mod}
(\alpha)_{\mathbf{e}\cdot(\mathbf{a}+\mathbf{m}p)}\in(\alpha)_{\mathbf{e}\cdot\mathbf{a}}(-p)^{\mathbf{e}\cdot\mathbf{m}}
\big(\mathfrak{D}_p(\alpha)\big)_{\mathbf{e}\cdot\mathbf{m}}(1+p\mathbb{Z}_{(p)}) \,.
\end{equation}
By applying \eqref{Eq PochProd mod} to all pairs in $\mathbf{u}$ and $\mathbf{v}$, and using that by assumption $\sum_{i=1}^r\mathbf{e}_i=\sum_{j=1}^s\mathbf{f}_j$, 
we finally deduce that 
$$
\mathcal{Q}_{\mathbf{u},\mathbf{v}}(\mathbf{a}+\mathbf{m}p)
\in\mathcal{Q}_{\mathbf{u},\mathbf{v}}(\mathbf{a})\mathcal{Q}_{\mathfrak{D}_p(\mathbf{u}),\mathfrak{D}_p(\mathbf{v})}(\mathbf{m})
(1+p\mathbb{Z}_{(p)}) \,,
$$
which ends the proof. 
\end{proof}

We are now ready to prove Proposition \ref{theo Lucas Hypergeom}. 

\begin{proof}[Proof of Proposition \ref{theo Lucas Hypergeom}] 
Let $A$ be a subset of $\{1,\dots,d_{\boldsymbol{\alpha},\boldsymbol{\beta}}\}$ which satisfies Assertions~(i) and (ii) in Proposition \ref{theo Lucas Hypergeom}. 
We recall that 
$$
\mathcal{S}:= \left \{ p \in \mathcal P \,:\, p\equiv a\mod d_{\boldsymbol{\alpha},\boldsymbol{\beta}} 
\mbox{ for some } a\in A, \mbox{ and } p> d_{\boldsymbol{\alpha},\boldsymbol{\beta}} \right\} \, . 
$$
Note that if $p$ belongs to $\mathcal{S}$ and $\ell$ is a positive integer, then there exists $a$ in $A$ such that 
$p^\ell\equiv a\mod d_{\boldsymbol{\alpha},\boldsymbol{\beta}}$. 

\medskip

We first prove that $F_{\mathbf{u},\mathbf{v}}(\mathbf{x})$ belongs to $\mathbb{Z}_{(p)}[[\mathbf{x}]]$ for all $p$ in $\mathcal{S}$.  
Let $p$ in $\mathcal{S}$ and $\mathbf{n}$ in $\mathbb{N}^d$ be fixed. 
By assumption, we have $\sum_{i=1}^r\mathbf{e}_i=\sum_{j=1}^s\mathbf{f}_j$ and the elements of $\boldsymbol{\alpha}$ and $\boldsymbol{\beta}$ 
belong to $(0,1]\cap \mathbb Z_{(p)}$. We thus infer from the discussion of Section \ref{Section $p$-adic tools} that 
\begin{equation}\label{val Delta}
v_p\big(\mathcal{Q}_{\mathbf{u},\mathbf{v}}(\mathbf{n})\big)
=\sum_{\ell=1}^\infty\xi_{\iota(p^\ell)}\left(\left\{\frac{\mathbf{n}}{p^\ell}\right\}\right)\,.
\end{equation}
For all non-negative integers $\ell$, we have $\iota(p^\ell)\in A$. By Assertion (ii) in Proposition \ref{theo Lucas Hypergeom} and 
since $\xi_{\iota(p^\ell)}(\mathbf{x})= 0$ 
outside $\mathcal{D}_{\mathbf{u},\mathbf{v}}^{\iota(p^\ell)}$, we deduce that 
$\xi_{\iota(p^\ell)}(\mathbf{x})\geq 0$, for all $\mathbf{x}$ in $[0,1)^d$. Hence, we obtain that 
$v_p\big(\mathcal{Q}_{\mathbf{u},\mathbf{v}}(\mathbf{n})\big)\geq 0$ and thus $F_{\mathbf{u},\mathbf{v}}(\mathbf{x})$ belongs to $\mathbb{Z}_{(p)}[[\mathbf{x}]]$, as claimed. 
 
\medskip

Let $k:=\mbox{\rm Card } A$. Let $p$ be a fixed prime in $\mathcal{S}$. We shall show that $F_{\mathbf{u},\mathbf{v}}$ has the $p^k$-Lucas property. 
Let $\mathbf{a}=\mathbf{a}_0+\mathbf{a}_1p+\cdots+\mathbf{a}_{k-1}p^{k-1}$ with $\mathbf{a}_i$ in $\{0,\dots,p-1\}^d$. We need to distinguish two cases.

\medskip

\begin{itemize}
\item[$\bullet$]  Case $1$: there exists an integer $i$, $0\leq i\leq k-1$, such that $\mathbf{a}_i/p$ belongs to $\mathcal{D}_{\mathbf{u},\mathbf{v}}^{\iota(p^{i+1})}$.
\end{itemize} 
\medskip

In that case, we are going to show that $\mathcal{Q}_{\mathbf{u},\mathbf{v}}(\mathbf{a}+\mathbf{m}p^k)\equiv 0\mod p\mathbb Z_{(p)}$, for all ${\bf m}$ in $\mathbb N^d$. 
We first note that since $\iota(p^\ell)$ belongs to $A$, one has $\xi_{\iota(p^\ell)}\geq 0$ on $\mathbb{R}^d$.  
We thus get that 
\begin{align*}
v_p\big(\mathcal{Q}_{\mathbf{u},\mathbf{v}}(\mathbf{a}+\mathbf{m}p^k)\big)& =
\sum_{\ell=1}^\infty\xi_{\iota(p^\ell)}\left(\left\{\frac{\mathbf{a}+\mathbf{n}p^k}{p^\ell}\right\}\right)\\ 
& \geq\xi_{\iota(p^{i+1})}\left(\frac{\mathbf{a}_0+
\mathbf{a}_1p+\cdots+\mathbf{a}_ip^i}{p^{i+1}}\right)\, ,
\end{align*}
for all $\mathbf{m}$ in $\mathbb{N}^d$. Now, since 
$$
\mathbf{x}:=\frac{\mathbf{a}_0+\mathbf{a}_1p+\cdots+\mathbf{a}_ip^i}{p^{i+1}}\geq\frac{\mathbf{a}_i}{p}\in\mathcal{D}_{\mathbf{u},\mathbf{v}}^{\iota(p^{i+1})} \, ,
$$
we obtain that $\mathbf{x}\in\mathcal{D}_{\mathbf{u},\mathbf{v}}^{\iota(p^{i+1})}$, and thus $\xi_{\iota(p^{i+1})}(\mathbf{x})\geq 1$ by Assertion (ii).  
This gives: 
$$
\mathcal{Q}_{\mathbf{u},\mathbf{v}}(\mathbf{a}+\mathbf{m}p^k)\equiv 0\mod p\mathbb Z_{(p)} \, ,
$$
for all ${\bf m}$ in $\mathbb N^d$. Choosing ${\bf m}=0$, we deduce that  
$$
\mathcal{Q}_{\mathbf{u},\mathbf{v}}(\mathbf{a}) \equiv 0 \mod p \mathbb Z_{(p)}
$$
and thus 
$$
\mathcal{Q}_{\mathbf{u},\mathbf{v}}(\mathbf{a}+\mathbf{m}p^k)\equiv 
\mathcal{Q}_{\mathbf{u},\mathbf{v}}(\mathbf{a})\mathcal{Q}_{\mathbf{u},\mathbf{v}}(\mathbf{m}) \mod p \mathbb Z_{(p)}\, ,
$$
for all ${\bf m}$ in $\mathbb N^d$. This ends the proof in that case.

\medskip

\begin{itemize}
\item[$\bullet$] Case $2$: for all integers $i$, $0\leq i\leq k-1$, $\mathbf{a}_i/p$ does not belong to $\mathcal{D}_{\mathbf{u},\mathbf{v}}^{\iota(p^{i+1})}$.
\end{itemize}
\medskip 

We will recursively apply Lemma \ref{Lemma DecompProd} $k$ times. 
We first observe that  $\mathcal{Q}_{\mathfrak{D}_p(\mathbf{u}),\mathfrak{D}_p(\mathbf{v})}({\bf n})$ belongs to $\mathbb Z_{(p)}$ 
for any ${\bf n}$ in $\mathbb N^d$. Indeed, for all $a$ in $A$ and all $\alpha$ in $\boldsymbol{\alpha}$ or $\boldsymbol{\beta}$, we have
$$
\langle a\mathfrak{D}_p(\alpha)\rangle=\big\langle a\langle\iota(p)\alpha\rangle\big\rangle=\langle a\iota(p)\alpha\rangle=\langle b\alpha\rangle \, ,
$$
for some $b$ in $A$ as a consequence of Assertion (i) in Proposition \ref{theo Lucas Hypergeom}. 
Hence the function $\xi_{\mathfrak{D}_p(\mathbf{u}),\mathfrak{D}_p(\mathbf{v})}^a=\xi_{\mathbf{u},\mathbf{v}}^b$ is non-negative on $\mathbb{R}^d$, 
and $\mathcal{D}_{\mathfrak{D}_p(\mathbf{u}),\mathfrak{D}_p(\mathbf{v})}^a=\mathcal{D}_{\mathbf{u},\mathbf{v}}^b$. 
We thus have that both $\mathcal{Q}_{\mathbf{u},\mathbf{v}}({\bf n})$ and $\mathcal{Q}_{\mathfrak{D}_p(\mathbf{u}),\mathfrak{D}_p(\mathbf{v})}({\bf n})$ 
belong to $\mathbb Z_{(p)}$  for all ${\bf n}$ in $\mathbb N^d$. 
Then we infer from Lemma \ref{Lemma DecompProd} that 
$$
\mathcal{Q}_{\mathbf{u},\mathbf{v}}(\mathbf{a}+\mathbf{m}p^k)\equiv
\mathcal{Q}_{\mathbf{u},\mathbf{v}}(\mathbf{a}_0)\mathcal{Q}_{\mathfrak{D}_p(\mathbf{u}),
\mathfrak{D}_p(\mathbf{v})}(\mathbf{a}_1+\cdots+\mathbf{a}_{k-1}p^{k-2}+\mathbf{m}p^{k-1})\mod p\mathbb Z_{(p)} \,,
$$
for all ${\bf m}$ in $\mathbb N^d$. 
Since $\xi_{\mathfrak{D}_p(\mathbf{u}),\mathfrak{D}_p(\mathbf{v})}^a=\xi_{\mathbf{u},\mathbf{v}}^b$  and 
$\mathcal{D}_{\mathfrak{D}_p(\mathbf{u}),\mathfrak{D}_p(\mathbf{v})}^a=\mathcal{D}_{\mathbf{u},\mathbf{v}}^b$, it follows that 
$\mathfrak{D}_p(\mathbf{u})$ and $\mathfrak{D}_p(\mathbf{v})$ satisfy the assumptions of Proposition \ref{theo Lucas Hypergeom} 
with the same set $A$. 
Furthermore, if $a=\iota(p^\ell)$, then we obtain that $b=\iota(p^{\ell+1})$, which leads to 
$\mathcal{D}_{\mathfrak{D}_p(\mathbf{u}),\mathfrak{D}_p(\mathbf{v})}^{\iota(p^\ell)}=\mathcal{D}_{\mathbf{u},\mathbf{v}}^{\iota(p^{\ell+1})}$. 
We can thus apply Lemma \ref{Lemma DecompProd} to $\mathfrak{D}_p(\mathbf{u})$ and $\mathfrak{D}_p(\mathbf{v})$,   
with $\mathbf{a}_1$ instead of $\mathbf{a}$. This yields 
$$
\mathcal{Q}_{\mathbf{u},\mathbf{v}}(\mathbf{a}+\mathbf{m}p^k)\equiv\mathcal{Q}_{\mathbf{u},\mathbf{v}}(\mathbf{a}_0)
\mathcal{Q}_{\mathfrak{D}_p(\mathbf{u}),\mathfrak{D}_p(\mathbf{v})}(\mathbf{a}_1)
\mathcal{Q}_{\mathfrak{D}_p^2(\mathbf{u}),\mathfrak{D}_p^2(\mathbf{v})}(\mathbf{a}_2+\cdots+\mathbf{a}_{k-1}p^{k-3}+\mathbf{m}p^{k-2})\mod p\mathbb Z_{(p)} \,,
$$
for all ${\bf m}$ in $\mathbb N^d$. 
By induction, we obtain that
$$
\mathcal{Q}_{\mathbf{u},\mathbf{v}}(\mathbf{a}+\mathbf{m}p^k)\equiv\left(\prod_{i=0}^{k-1}\mathcal{Q}_{\mathfrak{D}_p^i(\mathbf{u}),\mathfrak{D}_p^i(\mathbf{v})}(\mathbf{a}_i)\right)\mathcal{Q}_{\mathfrak{D}_p^k(\mathbf{u}),\mathfrak{D}_p^k(\mathbf{v})}(\mathbf{m})\mod p\mathbb Z_{(p)} \,,
$$
for all ${\bf m}$ in $\mathbb N^d$.  Choosing $\mathbf{m}=\mathbf{0}$, we deduce that 
$$
\mathcal{Q}_{\mathbf{u},\mathbf{v}}(\mathbf{a}+\mathbf{m}p^k)\equiv
\mathcal{Q}_{\mathbf{u},\mathbf{v}}(\mathbf{a})\mathcal{Q}_{\mathfrak{D}_p^k(\mathbf{u}),\mathfrak{D}_p^k(\mathbf{v})}(\mathbf{m})\mod p\mathbb Z_{(p)} \,,
$$
for all ${\bf m}$ in $\mathbb N^d$.  
Since $k=\mbox{\rm card } A$, one has $p^k\equiv 1\mod d_{\boldsymbol{\alpha},\boldsymbol{\beta}}$. 
Then for all $\alpha$ in $\boldsymbol{\alpha}$ or $\boldsymbol{\beta}$, we have 
$$
\mathfrak{D}_p^k(\alpha)=\langle\iota(p^k)\alpha\rangle=\langle\alpha\rangle=\alpha \,, 
$$ 
since $\alpha\in (0,1]$. 
Hence $\mathfrak{D}_p^k(\mathbf{u})=\mathbf{u}$ and  $\mathfrak{D}_p^k(\mathbf{v})=\mathbf{v}$, and it  follows that 
$$
\mathcal{Q}_{\mathbf{u},\mathbf{v}}(\mathbf{a}+\mathbf{m}p^k)\equiv\mathcal{Q}_{\mathbf{u},\mathbf{v}}(\mathbf{a})\mathcal{Q}_{\mathbf{u},\mathbf{v}}(\mathbf{m})
\mod p\mathbb Z_{(p)} \,,
$$
for all ${\bf m}$ in $\mathbb N^d$. This ends the proof. 
\end{proof}

\medskip

Now, we prove Proposition \ref{prop: spe}. 

\medskip

\begin{proof}[Proof of Proposition \ref{prop: spe}]  Set 
$$
\mathcal N := \big\{ \mathbf{n}\in\mathbb{N}^d \,:\, \forall a\in A, \forall \mathbf{x}\in [0,1)^d \mbox{ with } \mathbf{n}\cdot\mathbf{x}\geq 1,  \mbox{ one has } 
\xi_a(\mathbf{x})\geq 1\big\}\, .
$$
Let $\mathbf{n}=(n_1,\dots,n_d)$ be in $\mathcal N$, $(b_1,\dots,b_d)$ be a vector of non-zero rational numbers. Set
$$
S' := \left\{ p\in S \,:\, (b_1,\dots,b_d)\in\mathbb{Z}_{(p)}^d\right\}
$$
and let $p$ be a prime number in $S'$. We will simply write $F$ for $F_{\mathbf{u},\mathbf{v}}$ and $\mathcal{Q}$ for $\mathcal{Q}_{\mathbf{u},\mathbf{v}}$. 
By Proposition \ref{theo Lucas Hypergeom}, the sequence $\mathcal{Q}({\bf n})$ has the $p^k$-Lucas property, so that 
$$
F(\mathbf{x})\equiv\left(\sum_{\mathbf{0}\leq\mathbf{a}\leq(p^k-1)\mathbf{1}}
\mathcal{Q}(\mathbf{a})\mathbf{x}^{\mathbf{a}}\right)F(\mathbf{x}^{p^k})\mod p\mathbb{Z}_{(p)}[[\mathbf{x}]] \,.
$$
This gives:
\begin{align*}
F(b_1x^{n_1},\dots,b_dx^{n_d})&\equiv\left(\sum_{\mathbf{0}\leq\mathbf{a}\leq(p^k-1)\mathbf{1}}\mathbf{b}^{\mathbf{a}}\mathcal{Q}(\mathbf{a})x^{\mathbf{n}\cdot\mathbf{a}}\right)F\Big(b_1^{p^k}x^{n_1p^k},\dots,b_d^{p^k}x^{n_dp^k}\Big)\mod p\mathbb{Z}_{(p)}[[x]]\\
&\equiv\left(\sum_{\mathbf{0}\leq\mathbf{a}\leq(p^k-1)\mathbf{1}}\mathbf{b}^{\mathbf{a}}\mathcal{Q}(\mathbf{a})x^{\mathbf{n}\cdot\mathbf{a}}\right)
F\Big(b_1x^{n_1p^k},\dots,b_dx^{n_dp^k}\Big)\mod p\mathbb{Z}_{(p)}[[x]] \,,
\end{align*}
since $b_i^{p^k}\equiv b_i\mod p\mathbb{Z}_{(p)}$ for all $i$ in $\{1,\dots,d\}$. 
For all $\mathbf{a}$ in $\{0,\dots,p^k-1\}^d$ satisfying $\mathbf{n}\cdot\mathbf{a}\geq p^k$, we have 
$\mathbf{n}\cdot\mathbf{a}/p^k\geq 1$ and thus $\xi_a(\mathbf{a}/p^k)\geq 1$ for all $a$ in $A$. 
It follows that
$$
v_p\big(\mathcal{Q}(\mathbf{a})\big)=
\sum_{\ell=1}^\infty\xi_{\iota(p^\ell)}\left(\frac{\mathbf{a}}{p^\ell}\right)\geq\xi_{\iota(p^k)}\left(\frac{\mathbf{a}}{p^k}\right)\geq 1\, ,
$$
because $\xi_a$ is non-negative on $\mathbb{R}^d$  for every $a$ in $A$. 
Thus there is a polynomial $A(x)$ with coefficients in $\mathbb{Z}_{(p)}$ and of degree at most $p^k-1$ such that
$$
F(b_1x^{n_1},\dots,b_dx^{n_d})\equiv A(x)F\Big(b_1x^{n_1p^k},\dots,b_dx^{n_dp^k}\Big)\mod p\mathbb{Z}_{(p)}[[x]] \, .
$$
This shows that $F(b_1x^{n_1},\dots,b_dx^{n_d})$ satisfies the $p^k$-Lucas property, as expected. 
\end{proof}

\section{Lucas-type congruences among classical families of $G$-functions}\label{sec: appli}

In Section \ref{sec: mgh}, we gave a general condition involving some step functions $\xi_a({\bf x})$ 
that ensures a function in $\mathcal H_d$ satisfies Lucas-type congruences.  
Propositions \ref{prop: fuv} and \ref{theo Lucas Hypergeom} actually take a much simpler form when working with more specific families of $G$-functions.  We illustrate 
this fact by considering first two classical families: the generating series of factorial ratios and the generalized hypergeometric functions. Then we 
discuss the case of multivariate factorial ratios and show how their specializations lead to Lucas-type congruences for $G$-functions involving 
various sums and products of binomials, such as those associated with Ap\'ery, Franel, Domb, and Delannoy numbers. 

\subsection{Generating series of factorial ratios }\label{sec: fr}

Given two tuples of vectors of natural numbers, $e=(e_1,\dots, e_u)$ and $f=(f_1,\dots,f_v)$, the associated sequence of factorial ratio is defined by 
$$
\mathcal{Q}_{e,f}(n):=\frac{\prod_{i=1}^u(e_in)!}{\prod_{i=1}^v(f_in)!} \,\cdot
$$
The generating series of such a sequence is then denoted by 
$$
F_{e,f}(x):=\sum_{n\in\mathbb{N}}\mathcal{Q}_{e,f}(n)x^{n} \, .
$$
In order to study when $\mathcal{Q}_{e,f}(n)$ is integer valued, Landau introduced the following simple step function $\Delta_{e,f}$ 
defined from $\mathbb R$ to $\mathbb Z$ by:
$$
\Delta_{e,f}(x):=\sum_{i=1}^u\lfloor e_ix\rfloor-\sum_{j=1}^v\lfloor f_jx\rfloor \, .
$$
According to Landau's criterion \cite{Landau}, and Bober's refinement \cite{Bober}, we have the following dichotomy. 

\begin{itemize}

\medskip

\item[$\bullet$] If, for all $x$ in $[0,1]$, one has $\Delta_{e,f}(x)\geq 0$, then  
 $\mathcal{Q}_{e,f}(n)\in\mathbb{N}$,  for all $n\geq 0$.

\medskip

\item[$\bullet$] If there exists $x$ in $[0,1]$ such that $\Delta_{e,f}(x) < 0$, then there are only finitely many prime numbers $p$ 
such that $\mathcal{Q}_{e,f}(n)$ belongs to $\mathbb Z_{(p)}$ for all $n\geq 0$.
\end{itemize}

\medskip

In the sequel, we always assume that the sets $\{e_1,\ldots,e_u\}$ and $\{f_1,\ldots,f_v\}$ are disjoint. We set 
$\vert e\vert :=\sum_{i=1}^u e_i$, $\vert f\vert :=\sum_{i=1}^v f_i$, and 
$$
m_{e,f} := \left( \max \{e_1,\ldots,e_u,f_1,\ldots,f_v\}\right)^{-1} \, . 
$$
Note that the functions of type $F_{e,f}$ correspond to functions of type $F_{{\bf u},{\bf v}}$ for which the parameter $d$, as well as  all the parameters $\alpha_i$ and 
$\beta_i$, are equal to $1$.  Propositions \ref{prop: fuv} and \ref{theo Lucas Hypergeom} can now be restated in a single result that 
 takes the following simple form. 

\begin{prop}\label{prop: fr} 
Let us assume that $\vert e\vert =\vert f\vert$ and that   $\Delta_{e,f}(x)\geq 1$, for all real numbers $x$ such that $m_{e,f} \leq x<1$. Then 
$F_{e,f}(x)\in{\mathfrak{L}}_1(\mathcal{P})$. In other words, $F_{e,f}(x)$ satisfies the  $p$-Lucas property for all prime numbers $p$.
\end{prop}

\begin{rem}
{\rm 
When all the $f_i$'s are equal to $1$, it becomes obvious that $\Delta_{e,f}(x)\geq 1$, for all real numbers $x$ such that $m_{e,f} \leq x<1$. 
This shows that all generating series of the form 
$$
\sum_{n=0}^{\infty} \frac{\prod_{i=1}^r(e_in)!}{(n!)^r} x^n \, ,
$$
where $e_1,\ldots,e_r$ are positive integers, satisfy the $p$-Lucas property for all prime numbers $p$. 
}
\end{rem}

We also prove the following refinement of Proposition \ref{prop: fr}. 

\begin{prop}\label{propo CritFacto}
The following assertions are equivalent. 

\medskip

\begin{enumerate}
\item[{\rm (i)}] There exists an infinite set of primes $\mathcal{S}$ such that $F_{e,f}(x)\in\mathfrak{L}_1(\mathcal{S})$.

\medskip

\item[{\rm (ii)}] The sequence $\mathcal{Q}_{e,f}$ is integer-valued and has the $p$-Lucas property for all primes $p$.

\medskip

\item[{\rm (iii)}]  We have $|e|=|f|$ and $\Delta_{e,f}(x)\geq 1$ for all real numbers $x$ such that $m_{e,f} \leq x<1$.
\end{enumerate}
\end{prop}

\begin{rem}{\rm 
The equivalence of Assertions (ii) and (iii) is contained in \cite[Theorem~3]{Delaygue3}. A consequence of Proposition \ref{propo CritFacto} is that 
$F_{e,f}(x)$ belongs to $\mathfrak{L}_1(\mathcal{S})$ for an infinite set of primes $\mathcal{S}$ if and only if all Taylor coefficients at the origin of the associated 
mirror map $z_{e,f}$ are integers (see Theorems $1$ and $3$ in \cite{Delaygue1}). It would be interesting to investigate in more details this 
intriguing connection. 
}
\end{rem}

\begin{proof}[Proof of Proposition \ref{propo CritFacto}]
Obviously, Assertion (ii) implies Assertion (i), and Assertions~(ii) and (iii) are shown to be equivalent in \cite[Theorem 3]{Delaygue3}. 
Hence it suffices to prove that (i) implies (iii). From now on, we assume that Assertion (i) holds. 
 
\medskip

First, we prove that $|e|=|f|$. Since $\mathcal{S}$ is infinite and $F_{e,f}(x)$ belongs to $\mathbb{Z}_p[[x]]$ for every prime $p$ in $\mathcal{S}$, 
Landau's criterion implies that  $\Delta_{e,f}(x)\geq 0$ for all $x$ in $[0,1]$. In particular, we obtain that 
$|e|-|f|=\Delta_{e,f}(1)\geq 0$. If $|e|>|f|$ then $\Delta_{e,f}(1)\geq 1$. Set $M_{e,f} :=  \max \{e_1,\ldots,e_u,f_1,\ldots,f_v\}$. 
Then, for all prime numbers $p>M_{e,f}$ and all positive integers $k$, we have 
\begin{align*}
v_p\big(\mathcal{Q}_{e,f}(1+p^k)\big) & =\sum_{\ell=1}^\infty\Delta_{e,f}\left(\frac{1+p^k}{p^\ell}\right) \\
& \geq\Delta_{e,f}\left(1+\frac{1}{p^k}\right)\\
&\geq 1 \, ,
\end{align*}
Our choice of $p$ ensures that $v_p(\mathcal{Q}_{e,f}(1))=0$. We thus deduce that, for almost all primes $p$ and all positive integers $k$, 
we have
$$
\mathcal{Q}_{e,f}(1+p^k)\not\equiv\mathcal{Q}_{e,f}(1)\mathcal{Q}_{e,f}(1)\mod p\mathbb{Z}_{(p)} \, ,
$$
which provides a contradiction with Assertion (i). Hence we get that $|e|=|f|$. 

\medskip

Now, we prove the following identity.  
For all prime numbers $p$, all positive integers $k$, all $a$ in $\{0,\dots,p^k-1\}$, and all natural integers $n$, we have
\begin{equation}\label{inter1}
\frac{\mathcal{Q}_{e,f}(a+np^k)}{\mathcal{Q}_{e,f}(a)\mathcal{Q}_{e,f}(n)}\in
\frac{\displaystyle\prod_{i=1}^u\prod_{j=1}^{\lfloor e_ia/p^k\rfloor}\left(1+\frac{e_i}{j}n\right)}
{\displaystyle\prod_{i=1}^v\prod_{j=1}^{\lfloor f_ia/p^k\rfloor}\left(1+\frac{f_i}{j}n\right)}\big(1+p\mathbb{Z}_{(p)}\big) \,.
\end{equation}
Indeed, we have
$$
\frac{\mathcal{Q}_{e,f}(a+np^k)}{\mathcal{Q}_{e,f}(a)\mathcal{Q}_{e,f}(n)}
=\frac{\mathcal{Q}_{e,f}(a+np^k)}{\mathcal{Q}_{e,f}(a)\mathcal{Q}_{e,f}(np^k)}
\prod_{j=0}^{k-1}\frac{\mathcal{Q}_{e,f}(np^{j+1})}{\mathcal{Q}_{e,f}(np^j)} \, \cdot
$$
Since $|e|=|f|$, we can apply \cite[Lemma $7$]{Delaygue2} 
\footnote{The proof of this lemma uses a lemma of Lang which contains an error. Fortunately, Lemma $7$ remains true. Details of this correction are presented in \cite[Section 2.4]{Delaygue4}.}
with $d=1$, $\mathbf{c}=0$, $\mathbf{m}=np^j$ and $s=0$ which leads to
$$
\frac{\mathcal{Q}_{e,f}(np^{j+1})}{\mathcal{Q}_{e,f}(np^j)}\in 1+p\mathbb{Z}_{(p)}\, .
$$
Furthermore, we have
\begin{align*}
\frac{\mathcal{Q}_{e,f}(a+np^k)}{\mathcal{Q}_{e,f}(a)\mathcal{Q}_{e,f}(np^k)}&=\frac{1}{\mathcal{Q}_{e,f}(a)}
\frac{\prod_{i=1}^u\prod_{j=1}^{e_ia}(j+e_inp^k)}{\prod_{i=1}^v\prod_{j=1}^{f_ia}(j+f_inp^k)}\\
&=\frac{\prod_{i=1}^u\prod_{j=1}^{e_ia}\left(1+\frac{e_inp^k}{j}\right)}{\prod_{i=1}^v\prod_{j=1}^{f_ia}\left(1+\frac{f_inp^k}{j}\right)}\\
&\in\frac{\prod_{i=1}^u\prod_{j=1}^{\lfloor e_ia/p^k\rfloor}\left(1+\frac{e_in}{j}\right)}
{\prod_{i=1}^v\prod_{j=1}^{\lfloor f_ia/p^k\rfloor}\left(1+\frac{f_in}{j}\right)}\big(1+p\mathbb{Z}_{(p)}\big) \, ,
\end{align*}
since, if $p^k$ does not divide $j$, then $1+(e_inp^k)/j$ belongs to $1+p\mathbb{Z}_{(p)}$. This ends the proof of Equation \eqref{inter1}.

\medskip

Now we assume that there exists $x$ in $[m_{e,f},1[$ such that $\Delta_{e,f}(x)=0$ and we argue by contradiction. 

By assumption, for all $p$ in $\mathcal{S}$, there exists a positive integer $k_p$, such that, for all $v$ in $\{0,\dots,p^{k_p}-1\}$ and all natural integers $m$, 
we have 
$$
\mathcal{Q}_{e,f}(v+mp^{k_p})\equiv\mathcal{Q}_{e,f}(v)\mathcal{Q}_{e,f}(m)\mod p\mathbb{Z}_{(p)} \,.
$$
 Let $\gamma_1<\cdots<\gamma_t$ 
denote the abscissa of the points of discontinuity of $\Delta_{e,f}$ on $[0,1[$. In particular, we have 
$\gamma_1=m_{e,f}$. There exists $i$ in $\{1,\dots,t-1\}$
such that $\Delta_{e,f}(x)=0$ for all $x$ in $[\gamma_i,\gamma_{i+1}[$. For all large enough prime numbers $p\in\mathcal{S}$, we choose $r_p$ in $\{0,\dots,p-1\}$ 
such that $r_p/p$ belongs to $[\gamma_i,\gamma_{i+1}[$ and we set $a_p=r_pp^{k_p-1}$. Hence $a_p/p^{k_p}$ belongs to $[\gamma_i,\gamma_{i+1}[$. 
Then, by applying \eqref{inter1} in combination with \cite[Lemma $16$]{Delaygue2} (with $\mathbf{E}=e$ and $\mathbf{F}=f$), there are integers $m_1,\dots,m_i$ such that we have  
$$
\frac{\mathcal{Q}_{e,f}(a_p+p^{k_p})}{\mathcal{Q}_{e,f}(a_p)\mathcal{Q}_{e,f}(1)}\in\prod_{k=1}^{i}\left(1+\frac{1}{\gamma_k}\right)^{m_k}\big(1+p\mathbb{Z}_{(p)}\big) \, 
$$
and 
$$
\prod_{k=1}^{i}\left(1+\frac{1}{\gamma_k}\right)^{m_k}>1\, , 
$$
because $\Delta_{e,f}$ is non-negative on $[0,1]$. 
For all large enough primes $p$ in $\mathcal{S}$, we thus deduce that 
$$
\prod_{k=1}^{i}\left(1+\frac{1}{\gamma_k}\right)^{m_k}\notin 1+p\mathbb{Z}_{(p)}\, .
$$
Furthermore, for all large enough $p$ in $\mathcal{S}$, 
we have $1/p<m_{e,f}$, and $(a_p+p^{k_p})/p^\ell<m_{e,f}$, for $\ell\geq k_p+1$. 
It follows that $v_p(\mathcal{Q}_{e,f}(1))=0$, while 
$$
v_p\big(\mathcal{Q}_{e,f}(a_p)\big)=\sum_{\ell=1}^{k_p}\Delta_{e,f}\left(\left\{\frac{a_p}{p^\ell}\right\}\right)=\Delta_{e,f}\left(\frac{r_p}{p}\right)=0\, ,
$$
and
$$
v_p\big(\mathcal{Q}_{e,f}(a_p+p^{k_p})\big)=\sum_{\ell=1}^{k_p}\Delta_{e,f}\left(\left\{\frac{a_p+p^{k_p}}{p^\ell}\right\}\right)=\Delta_{e,f}\left(\frac{r_p}{p}+1\right)=0 \, .
$$
Hence $\mathcal{Q}_{e,f}(a_p+p^{k_p})\not\equiv\mathcal{Q}_{e,f}(a_p)\mathcal{Q}_{e,f}(1)\mod p\mathbb{Z}_{(p)}$ 
which leads to a contradiction,  
and ends the proof of  Proposition \ref{propo CritFacto}.
\end{proof}

Let us remind to the reader  that one easily obtains the graph of $\Delta_{e,f}$ on $[0,1]$ by translating a factorial ratio into hypergeometric form. 
We illustrate this process with the following  example. We consider 
$$
F(x) := \sum_{n=0}^{\infty} \frac{(10n)!}{(5n)!(3n)!n!^2} x^n \, .
$$
We have
\begin{align*}
\frac{(10n)!}{(5n)!(3n)!n!^2}&=\frac{\prod_{k=0}^{n-1}\prod_{j=1}^{10}(10k+j)}{\prod_{k=0}^{n-1}\big(\prod_{j=1}^5(5k+j)\big)(3k+1)(3k+2)(3k+3)(k+1)^2}\\
&=\left(\frac{10^{10}}{5^53^3}\right)^n\frac{\prod_{k=0}^{n-1}\prod_{j=1}^{10}\big(k+\frac{j}{10}\big)}{\prod_{k=0}^{n-1}\big(\prod_{j=1}^5\big(k+\frac{j}{5}\big)\big)\left(k+\frac{1}{3}\right)\left(k+\frac{2}{3}\right)(k+1)^3}\\
&=\left(\frac{10^{10}}{5^53^3}\right)^n\frac{\prod_{j=1}^{10}(j/10)_n}{(1/3)_n(2/3)_n(1)_n^3\prod_{j=1}^5(j/5)_n}\\
&=\left(\frac{10^{10}}{5^53^3}\right)^n\frac{(1/10)_n(3/10)_n(1/2)_n(7/10)_n(9/10)_n}{(1/3)_n(2/3)_n(1)_n^3} \, \cdot
\end{align*}
Then we deduce that $\Delta_{e,f}$ has jumps of amplitude $1$ at $1/10,3/10,1/2,7/10$ and $9/10$, while the abscissa of its jumps of amplitude $-1$ 
are $1/3,2/3$. Furthermore, $\Delta_{e,f}$ has a jump of amplitude $-3$ at $1$. Since
$$
\frac{1}{10}<\frac{3}{10}<\frac{1}{3}<\frac{1}{2}<\frac{2}{3}<\frac{7}{10}<\frac{9}{10}<1\, ,
$$
we get that $\Delta_{e,f}\geq 1$ on $[1/10,1)$ and it follows form Proposition \ref{prop: fr} that the function $F(x)$ 
satisfies the $p$-Lucas property for all prime numbers. 
Along the same lines, one can prove for instance that the $G$-functions 
$$
\sum_{n=0}^\infty\frac{(5n)!(3n)!}{(2n)!^2n!^4}x^n\; \mbox{ and } \; \sum_{n=0}^\infty\frac{(4n)!}{(2n)!n!^2}x^n
$$
also satisfy the $p$-Lucas property for all prime numbers.

\subsection{Generalized hypergeometric series}\label{sec: hyp}

With two tuples $\boldsymbol{\alpha}:=(\alpha_1,\dots,\alpha_r)$ and $\boldsymbol{\beta}:=(\beta_1,\dots,\beta_s)$ 
of elements in $\mathbb{Q}\setminus\mathbb{Z}_{\leq 0}$, we can associate the generalized hypergeometric series
$$
{}_{r}F_{s}\left[\begin{array}{cc}\alpha_1,\dots,\alpha_r\\ \beta_1,\dots,\beta_s\end{array};x\right]
:=\sum_{n=0}^\infty\frac{(\alpha_1)_n\cdots(\alpha_r)_n}{(\beta_1)_n\cdots(\beta_s)_n}\frac{x^n}{n!} \,\cdot
$$
Here, we set 
$$
\mathcal{Q}_{\boldsymbol{\alpha},\boldsymbol{\beta}}(n):=\frac{(\alpha_1)_n\cdots(\alpha_r)_n}{(\beta_1)_n\cdots(\beta_s)_n}\quad\textup{and}\quad 
F_{\boldsymbol{\alpha},\boldsymbol{\beta}}(x):=\sum_{n=0}^\infty\mathcal{Q}_{\boldsymbol{\alpha},\boldsymbol{\beta}}(n)x^n \, ,
$$
so that 
$$
F_{\boldsymbol{\alpha},\boldsymbol{\beta}}(x)={}_{r+1}F_{s}\left[\begin{array}{cc}\alpha_1,\dots,\alpha_r,1\\ \beta_1,\dots,\beta_s\end{array};x\right] \, .
$$
Note that such series simply correspond to series of type $F_{{\bf u},{\bf v}}$ for which the parameter $d$, as well as all parameters $e_i$ and $f_i$,  
are equal to $1$.  
\medskip

To apply Proposition \ref{theo Lucas Hypergeom}, we have to compare the numbers $\langle a\alpha_i\rangle$ and $\langle a\beta_j\rangle$ for $a$ in $\{1,\dots,d_{\boldsymbol{\alpha},\boldsymbol{\beta}}\}$ coprime to $d_{\boldsymbol{\alpha},\boldsymbol{\beta}}$. Indeed, on $[0,1)$, $\xi_a$ have jumps of amplitude $1$ at $\langle a\alpha_i\rangle$ and jumps of amplitude $-1$ at $\langle a\beta_j\rangle$, where $\alpha_i$ and $\beta_j$ are not equal to $1$. For every $a$ in $\{1,\dots,d_{\boldsymbol{\alpha},\boldsymbol{\beta}}\}$ coprime to $d_{\boldsymbol{\alpha},\boldsymbol{\beta}}$, we set
$$
m_{\boldsymbol{\alpha},\boldsymbol{\beta}}(a):=\min\big(\langle a\alpha_1\rangle,\dots,\langle a\alpha_r\rangle,\langle a\beta_1\rangle,\dots,\langle a\beta_s\rangle\big),
$$
so that the corresponding set $\mathcal{D}_{\mathbf{u},\mathbf{v}}^a$ equals $[m_{\boldsymbol{\alpha},\boldsymbol{\beta}}(a),1)$.

\begin{exam} {\rm Let us illustrate Proposition \ref{theo Lucas Hypergeom} with few examples. 

$\bullet$ We first choose $\boldsymbol{\alpha}=(1/2,1/2)$ and $\boldsymbol{\beta}=(2/3,1)$. 
We thus have $d_{\boldsymbol{\alpha},\boldsymbol{\beta}}=6$ 
and $1/2<2/3$, so $\xi_1(x)\geq 1$ for $x$ in $[1/2,1)$. On the other hand, we have 
$$
\frac{1}{3}=\left\langle 5\cdot \frac{2}{3}\right\rangle< \left\langle 5\cdot \frac{1}{2}\right\rangle=\frac{1}{2}\, ,
$$
so $\xi_5(1/3)=-1$. 
The maximal set $A$ for which we can apply Proposition \ref{theo Lucas Hypergeom} is thus $A=\{1\}$. 
Hence we deduce that $F_{\boldsymbol{\alpha},\boldsymbol{\beta}}(x)$ has the $p$-Lucas properties for all primes $p\equiv 1\mod 6$. 
We stress that, according to Theorem A in \cite{Delaygue4} (which is a reformulation of Christol's result \cite[Proposition 1]{Christol}), 
the function 
$$F_{\boldsymbol{\alpha},\boldsymbol{\beta}}(x) = \sum_{n= 0}^{\infty} \frac{(1/2)_n^2}{(2/3)_n(1)_n} x^n 
$$ is not globally bounded, that is there is no $C$ in $\mathbb{Q}$ such that 
$F_{\boldsymbol{\alpha},\boldsymbol{\beta}}(Cx)$ belongs to $\mathbb{Z}[[x]]$. In particular, it cannot be expressed as the 
diagonal of a multivariate algebraic function. 

\medskip

$\bullet$ Let study another example by taking $\boldsymbol{\alpha}=(1/9,4/9,5/9)$ and $\boldsymbol{\beta}=(1/3,1,1)$. 
This choice of parameters was considered by  Christol in \cite{Christol}. 
We have $d_{\boldsymbol{\alpha},\boldsymbol{\beta}}=9$ and
$$
\frac{1}{9}<\frac{1}{3}<\frac{4}{9}<\frac{5}{9}\,.
$$
Hence $m_{\boldsymbol{\alpha},\boldsymbol{\beta}}(1)=1/9$ but $\xi_1(1/3)=0<1$, 
so we cannot just apply Proposition \ref{theo Lucas Hypergeom}. Indeed, any subgroup of $(\mathbb Z/9\mathbb Z)^{\times}$ must of course contain $1$.

\medskip

$\bullet$ For our last example, let us choose $\boldsymbol{\alpha}=(1/3,1/2)$ and $\boldsymbol{\beta}=(3/4,1)$. 
Hence we have $d_{\boldsymbol{\alpha},\boldsymbol{\beta}}=12$ and
$$
\frac{1}{3}<\frac{1}{2}<\frac{3}{4},\quad\left\langle\frac{5}{2}\right\rangle<\left\langle\frac{5}{3}\right\rangle<\left\langle\frac{15}{4}\right\rangle\, ,
$$
while
$$
\left\langle\frac{21}{4}\right\rangle<\left\langle\frac{7}{3}\right\rangle<\left\langle\frac{7}{2}\right\rangle\quad\textup{and}\quad\left\langle\frac{33}{4}\right\rangle<\left\langle\frac{11}{2}\right\rangle<\left\langle \frac{11}{3}\right\rangle\, .
$$
Observe furthermore that $5^2\equiv 1\mod 12$. 
This shows that the maximal set $A$ for which one can apply Proposition \ref{theo Lucas Hypergeom} is $A=\{1,5\}$. Hence 
the generalized hypergeometric series
$$
\sum_{n= 0}^{\infty} \frac{(1/3)_n(1/2)_n}{(3/4)_n(1)_n} x^n 
$$
satisfies the $p$-Lucas property for all primes congruent to $1$ mod $12$ and the $p^2$-Lucas property for all primes congruent to $5$ mod $12$.}
\end{exam}

\medskip

We end this section by observing that our condition is always satisfied in the classical case where  the hypergeometric differential equation associated with 
$F_{\boldsymbol{\alpha},\boldsymbol{\beta}}(x)$ 
has maximal unipotent monodromy at the origin. 

\begin{cor}\label{cor: hyp}
Let $\boldsymbol{\alpha}\in (\mathbb{Q}\cap(0,1))^r$ and $\boldsymbol{\beta}=(1,\dots,1)\in \mathbb Q^r$.  
Then the generalized hypergeometric series $F_{\boldsymbol{\alpha},\boldsymbol{\beta}}(x)$ belongs to ${\mathfrak{L}}_1(\mathcal{S})$, where $\mathcal{S}$ is the set of all 
primes larger than $d_{\boldsymbol{\alpha},\boldsymbol{\beta}}$. 
\end{cor}

\subsection{Multivariate factorial ratios and specializations}\label{sec: spe}

We consider now a class of multivariate power series which provides a higher-dimensional generalization of generating series associated 
with factorial ratios that we discussed in Section \ref{sec: fr}.  
Given two tuples of vectors in $\mathbb{N}^d$, $e=(\mathbf{e}_1,\dots,\mathbf{e}_u)$ and $f=(\mathbf{f}_1,\dots,\mathbf{f}_v)$, we write $|e|=\sum_{i=1}^u\mathbf{e}_i$ and, for all $\mathbf{n}$ in $\mathbb{N}^d$, we set
$$
\mathcal{Q}_{e,f}(\mathbf{n}):=\frac{\prod_{i=1}^u(\mathbf{e}_i\cdot\mathbf{n})!}{\prod_{i=1}^v(\mathbf{f}_i\cdot\mathbf{n})!}\quad\textup{and}\quad F_{e,f}(\mathbf{x}):=\sum_{\mathbf{n}\in\mathbb{N}^d}\mathcal{Q}_{e,f}(\mathbf{n})\mathbf{x}^{\mathbf{n}}.
$$

Such multivariate power series simply correspond to functions of type $F_{{\bf u},{\bf v}}$ for which all parameters  $\alpha_i$ and $\beta_i$ are equal to one.  
In particular, we have $d_{\boldsymbol{\alpha},\boldsymbol{\beta}} =1$.  Propositions \ref{prop: fuv} 
and \ref{theo Lucas Hypergeom} take the following much simpler form. 
As in Section \ref{sec: fr}, we consider the Landau function $\Delta_{e,f}$ 
defined from $\mathbb R^d$ to $\mathbb Z$ by:
$$
\Delta_{e,f}(\mathbf{x}):=\sum_{i=1}^u\lfloor\mathbf{e}_i\cdot\mathbf{x}\rfloor-\sum_{j=1}^v\lfloor\mathbf{f}_j\cdot\mathbf{x}\rfloor \, .
$$
Note that this function precisely corresponds in this setting to the function $\xi_1$ defined in Section \ref{sec: mgh}. 
We also recall that, as in the one-variable case, Landau's criterion \cite{Landau}, and Delaygue's refinement \cite{Delaygue2}, give the following dichotomy. 

\begin{itemize}

\medskip

\item[$\bullet$] If, for all $\mathbf{x}$ in $[0,1]^d$, one has $\Delta_{e,f}(\mathbf{x})\geq 0$, then  
 $\mathcal{Q}_{e,f}(\mathbf{n})$ is an integer for all $\mathbf{n}$ in $\mathbb{N}^d$.

\medskip

\item[$\bullet$] If there exists $\mathbf{x}$ in $[0,1]^d$ such that $\Delta_{e,f}(\mathbf{x}) < 0$, then there are only finitely many prime numbers $p$ 
such that $\mathcal{Q}_{e,f}({\bf n})$ belongs to $\mathbb Z_{(p)}$ for all $\mathbf{n}$ in $\mathbb{N}^d$.
\end{itemize}

\medskip

Set 
$$
\mathcal{D}_{e,f}:= \big\{ \mathbf{x}\in[0,1)^d \,:\, \mbox{ there is } {\bf d} \mbox{ in } \{{\bf e_1},\ldots,{\bf e}_u,{\bf f}_1,\ldots,{\bf f}_v\} \mbox{ such that } \mathbf{d}\cdot\mathbf{x}\geq 1\big\} \, .
$$ 
Note that $\mathcal{D}_{e,f}$ corresponds to the set $\mathcal D_{{\bf u},{\bf v}}^1$ in this setting. Propositions \ref{prop: fuv} and 
\ref{theo Lucas Hypergeom} can now be restated in a single result as follows.   

\begin{prop}\label{prop: gfr} 
Let us assume that $|e|=|f|$ and that $\Delta_{e,f}(\mathbf{x})\geq 1$ for all $\mathbf{x}$ in $\mathcal{D}_{e,f}$. Then 
$F_{e,f}(\mathbf{x})$ belongs to ${\mathfrak{L}}_d(\mathcal{P})$. More precisely, $F_{e,f}(\mathbf{x})$ satisfies the  $p$-Lucas property for all primes $p$.
\end{prop}

This result is also proved by Delaygue in \cite[Theorem 3]{Delaygue3}.  
We give below a simple case of Proposition \ref{prop: gfr} that turns out to be especially useful for applications.   

\begin{cor}\label{cor: gfr}
For every $k$ in $\{1,\dots,d\}$, let us denote by $\mathbf{1}_k$ the vector of $\mathbb{N}^d$  whose $k$-th coordinate is one and all others  are zero.  
Let $e$ and $f=(\mathbf{1}_{k_1},\dots,\mathbf{1}_{k_v})$ be two disjoint tuples of non-zero vectors in $\mathbb{N}^d$ such that $|e|=|f|$ and 
$k_i\in\{1,\dots,d\}$, $1\leq i\leq v$.  
Then $F_{e,f}(\mathbf{x})$ satisfies the  $p$-Lucas property for all primes $p$. 
\end{cor}

\begin{proof}
Let $\mathbf{x}$ be in $\mathcal{D}_{e,f}$. By assumption, there is a coordinate $\mathbf{d}$ of either $e$ or $f$ such that 
$\mathbf{d}\cdot\mathbf{x}\geq 1$. But, since $\mathbf{x}$ belongs to $[0,1)^d$ and $f=(\mathbf{1}_{k_1},\dots,\mathbf{1}_{k_v})$, $\mathbf{d}$ has to be a coordinate of the vector 
$e$ so that 
\begin{align*}
\Delta_{e,f}(\mathbf{x}) &=
\sum_{i=1}^u\lfloor\mathbf{e}_i\cdot\mathbf{x}\rfloor-\sum_{j=1}^v\lfloor\mathbf{1}_{k_j}\cdot\mathbf{x}\rfloor \\
&=\sum_{i=1}^u\lfloor\mathbf{e}_i\cdot\mathbf{x}\rfloor \\
&\geq 1 \,.
\end{align*}
Proposition \ref{prop: gfr} then applies to conclude the proof. 
\end{proof}

\subsubsection{Specializations of factorial ratios}
Our main interest when working with multivariate power series in this setting is to benefit from the following general philosophy:  
interesting classical power series in one variable can be produced as simple specializations of simple multivariate power series. 
In particular, we claim that specializations of functions of type $F_{e,f}$ lead to many classical examples of generating functions arising in combinatorics and number theory.  
As already mentioned in Section \ref{sec: mgh},  the generating function of Ap\'ery's numbers 
$$
f(x)=\sum_{n=0}^\infty\left(\sum_{k=0}^n\binom{n}{k}^2\binom{n+k}{k}\right)x^n
$$
can be for instance obtained as the specialization $f(x)=F_{e,f}(x,x)$ of the two-variate  generating series of factorial ratios
\begin{equation}\label{example1}
F_{e,f}(x_1,x_2)=\sum_{(n_1,n_2)\in\mathbb{N}^2}\frac{(2n_1+n_2)!(n_1+n_2)!}{n_1!^3n_2!^2}x_1^{n_1}x_2^{n_2},
\end{equation}
 corresponding to the choice 
$$
e=\big((2,1),(1,1)\big)\quad\textup{and}\quad f=\big((1,0),(1,0),(1,0),(0,1),(0,1)\big) \, . 
$$

In order to support our claim, we gather in the following table some  classical sequences for which we prove that they satisfy  the $p$-Lucas property for all primes~$p$. 
Indeed, they all arise from specialization in $(x,x)$ of bivariate power series $F_{e,f}(x_1,x_2)$ that belong to $\mathfrak L_2(\mathcal P)$. The fact these bivariate power 
series belong to $\mathfrak L_2(\mathcal P)$ is a direct consequence of  Corollary \ref{cor: gfr}. Proposition \ref{prop: spe} then implies that 
the specialization $F_{e,f}(x,x)$ belong to $\mathfrak L_1(\mathcal P)$. 

\begin{small}
$$
{\renewcommand{\arraystretch}{2.5}
\begin{array}{|c|c|c|}
  \hline
  \textup{Sequence} & \mathcal{Q}_{e,f}(n_1,n_2)  & \textup{Reference from OEIS }\\
  \hline
    \displaystyle{\binom{2n}{n} = \sum_{k=0}^n\binom{n}{k}^2 }& \displaystyle{\frac{(n_1+n_2)!^2}{n_1!^2n_2!^2}} 
    &  \textup{Central binomial coefficients (A000984)} \\\hline
  \displaystyle{\sum_{k=0}^n\binom{n}{k}^2\binom{n+k}{k}^2} & \displaystyle{\frac{(2n_1+n_2)!^2}{n_1!^4n_2!^2}} 
  &  \textup{Ap\'ery numbers (A005259)} \\\hline
  \displaystyle{\sum_{k=0}^n\binom{n}{k}^2\binom{n+k}{k}} & \displaystyle{\frac{(2n_1+n_2)!(n_1+n_2)!}{n_1!^3n_2!^2}} 
  &  \textup{Ap\'ery numbers (A005258)} \\\hline
  
  \displaystyle{\sum_{k=0}^n\binom{n}{k}^3} & \displaystyle{\frac{(n_1+n_2)!^3}{n_1!^3n_2!^3}} 
  &  \textup{Franel numbers (A000172)} \\\hline
  \displaystyle{\sum_{k=0}^n\binom{n}{k}^4} & \displaystyle{\frac{(n_1+n_2)!^4}{n_1!^4n_2!^4}} 
  &  \textup{(A005260)} \\\hline
 
  \displaystyle{\sum_{k=0}^n\binom{n}{k}\binom{2k}{k}\binom{2(n-k)}{n-k}} & \displaystyle{\frac{(n_1+n_2)!(2n_1)!(2n_2)!}{n_1!^3n_2!^3}} 
  &  \textup{ (A081085)} \\\hline
  
  \displaystyle{\sum_{k=0}^n\binom{n}{k}^2\binom{2k}{k}} & \displaystyle{\frac{(n_1+n_2)!^2(2n_1)!}{n_1!^4n_2!^2}} 
  &  \textup{\parbox{5cm}{\centering Number of abelian squares \\ of length $2n$ over an alphabet \\ with $3$ letters (A002893)}} \\\hline
  \displaystyle{\sum_{k=0}^n\binom{n}{k}^2\binom{2k}{k}\binom{2(n-k)}{n-k}} & \displaystyle{\frac{(n_1+n_2)!^2(2n_1)!(2n_2)!}{n_1!^4n_2!^4}} 
  &  \textup{Domb numbers (A002895)} \\\hline
   \displaystyle{\sum_{k=0}^n\binom{n}{k}\binom{n+k}{k}} & \displaystyle{\frac{(2n_1+n_2)!}{n_1!^2n_2!}} 
   &  \textup{Central Delannoy numbers (A001850)} 
   \\\hline
     \displaystyle{\sum_{k=0}^n\binom{2k}{k}^2\binom{2(n-k)}{n-k}^2} & \displaystyle{\frac{(2n_1)!^2(2n_2)!^2}{n_1!^4n_2!^4}} 
     &  \textup{(A036917)} \\\hline

\end{array}
}
$$
\end{small}

\medskip

Let us end this section with an example of a different type, that is for which $f$ is not of the form $(\mathbf{1}_{k_1},\dots,\mathbf{1}_{k_v})$. 
Set 
$$
F_{e,f}(x_1,x_2):=\sum_{(n_1,n_2)\in\mathbb{N}^2}\frac{(3n_1+2n_2)!}{(n_1+n_2)!n_1!^2n_2!}x_1^{n_1}x_2^{n_2} \, .
$$
In that case, we obtain $\mathcal{D}_{e,f}=\{(x,y)\in[0,1)^2\,:\,3x_1+2x_2\geq 1\}$. 
When $(x_1,x_2)$ belongs to $\mathcal{D}_{e,f}$, we get that 
$$
\Delta_{e,f}(x_1,x_2)=\lfloor 3x_1+2x_2\rfloor-\lfloor x_1+x_2\rfloor\geq 1\,.
$$
By Proposition \ref{prop: gfr}, it follows that $F_{e,f}(x_1,x_2)$ has the $p$-Lucas property for all primes. 
Using specializations in $(-x,x)$ and $(2x^3,3x^2)$, we then infer from Proposition \ref{prop: spe} that both sequences   
$$
\sum_{k=0}^n(-1)^k\binom{2n+k}{n}\binom{n+k}{k}\binom{n}{k}
\quad\textup{and}\quad
\underset{k\equiv n\mod 2}{\sum_{k=0}^{\lfloor n/3\rfloor}}2^k3^{\frac{n-3k}{2}}\binom{n}{k}\binom{n-k}{\frac{n-k}{2}}\binom{\frac{n-k}{2}}{k}
$$
also satisfy the $p$-Lucas property for all prime numbers $p$.

\subsection{Examples from differential equations of Calabi-Yau type}\label{section: tables}

Motivated by the search for differential operators associated with particular families of Calabi-Yau varieties, Almkvist \textit{et al.} \cite{AESZ} gave a list of more than $400$ differential operators satisfying some algebraic conditions \cite[Section 1]{AESZ}. In particular, a condition is that the associated differential equation admits a unique power series solution near $z=0$ with constant term $1$ and that this power series has integral Taylor coefficients. In most of the cases, this solution is also given in \cite{AESZ}. It turns out that our method enables us to prove that most of these solutions have the $p$-Lucas property for infinitely many primes $p$. 
\medskip

By studying the integrality of the Taylor coefficients of mirror maps, Kratthentaler and Rivoal in \cite{Tanguy} and Delaygue in \cite[Section 10.2]{Delaygue0} showed that the power series solutions near $z=0$ of $143$ equations in Table \cite{AESZ} are specializations of series $F_{\mathbf{u},\mathbf{v}}(\mathbf{x})$ where $\boldsymbol{\alpha}$ and $\boldsymbol{\beta}$ are tuples of $1$'s. Furthermore, they showed that in $140$ cases the associated functions $\xi_1$ are greater than or equal to $1$ on $\mathcal{D}_{\mathbf{u},\mathbf{v}}^1$. Hence, to prove that these specializations have the $p$-Lucas property for all primes $p$, it suffices to show that the specialization is given by a vector $\mathbf{n}$ such that if $\mathbf{x}\in[0,1)^d$ and $\mathbf{n}\cdot\mathbf{x}\geq 1$, then $\xi_1(\mathbf{x})\geq 1$.
\medskip

Following this method, we checked that $212$ cases have the $p$-Lucas property for infinitely many primes $p$, namely Cases $1$--$25$, $29$, $3^\ast$, $4^\ast$, $4^{\ast\ast}$, $6^\ast$--$10^\ast$, $7^{\ast\ast}$--$10^{\ast\ast}$, $13^\ast$, $13^{\ast\ast}$, $\hat{1}$--$\hat{14}$, $30$, $31$, $34$--$41$, $43$--$83$, $85$--$108$, $110$--$116$, $119$--$122$, $124$--$132$, $145$--$153$, $155$--$172$, $180$, $185$, $188$, $190$--$192$, $197$, $208$, $209$, $212$, $232$, $233$, $237$--$241$, $243$, $278$, $284$, $288$, $292$, $307$, $323$, $330$, $337$, $338$, $340$, $367$, $369$--$372$, $377$, $380$, $398$.  

Among the cases not covered in \cite{Tanguy} nor \cite{Delaygue0}, we explain Cases $4^\ast$ and $31$ to give examples. In Case $4^\ast$, the power series solution near $z=0$ is 
$$
f(z)=\sum_{n=0}^\infty\left(\sum_{k=0}^n27^n\binom{2n}{n}\binom{-1/3}{k}^2\binom{-2/3}{n-k}^2\right)z^n.
$$
Hence $f(z)=F(27z,27z)$ where
$$
F(x,y)=\sum_{n_1,n_2\geq 0}\frac{(2n_1+2n_2)!(1/3)_{n_1}^2(2/3)_{n_2}^2}{(n_1+n_2)!^2n_1!^2n_2!^2}x^{n_1}y^{n_2},
$$ 
which is a series $F_{\mathbf{u},\mathbf{v}}(x,y)$. We have $d_{\boldsymbol{\alpha},\boldsymbol{\beta}}=3$, and, for every $(x,y)$ in $[0,1)^2$, we have
$$
\xi_1(x,y)=\xi_2(x,y)=\lfloor 2x+2y\rfloor+2\lfloor x-1/3\rfloor+2\lfloor y-2/3\rfloor-\lfloor x+y\rfloor+4.
$$
Furthermore, we have 
$$
\mathcal{D}_{\mathbf{u},\mathbf{v}}^1=\mathcal{D}_{\mathbf{u},\mathbf{v}}^2=\big\{(x,y)\in[0,1)^2\,:\,x\geq 1/3\textup{ or }y\geq 2/3\big\}.
$$
If $(x,y)$ belongs to $\mathcal{D}_{\mathbf{u},\mathbf{v}}^1$, then we have $\lfloor x-1/3\rfloor+\lfloor y-2/3\rfloor\in\{-1,0\}$ so that $\xi_1(x,y)\geq 2$. In addition, if $x+y\geq 1$, then we have $(x,y)\in\mathcal{D}_{\mathbf{u},\mathbf{v}}^1$ and $\xi_1(x,y)\geq 1$. By Propositions \ref{theo Lucas Hypergeom} and \ref{prop: spe}, we obtain that $f(z)$ has the $p$-Lucas property for all primes $p>3$.
\medskip

In Case $31$, the power series solution near $z=0$ can be rewritten as 
$$
f(z)=\sum_{n=0}^\infty\left(\sum_{k=0}^n\sum_{i=0}^k4^{3n}4^{2(n-k)}4^{k-i}(-1)^i\binom{2k}{k}\binom{2i}{i}\frac{(1/4)_n(1/4)_{n-k}(3/4)_k}{n!(n-k)!(k-i)!i!}\right)z^n.
$$
Hence $f(z)=F(4^5z,4^4z,-4^3z)$ where $F(x,y,z)$ is
$$
\sum_{n_1,n_2,n_3\geq 0}\binom{2(n_2+n_3)}{n_2+n_3}\binom{2n_3}{n_3}\frac{(1/4)_{n_1+n_2+n_3}(1/4)_{n_1}(3/4)_{n_2+n_3}}{(n_1+n_2+n_3)!n_1!n_2!n_3!}x^{n_1}y^{n_2}z^{n_3},
$$ 
which is a series $F_{\mathbf{u},\mathbf{v}}(x,y,z)$. We have $d_{\boldsymbol{\alpha},\boldsymbol{\beta}}=4$, and, for every $(x,y,z)$ in $[0,1)^3$, we have
$$
\xi_1(x,y,z)=\lfloor 2y+2z\rfloor+\lfloor 2z\rfloor+\lfloor x+y+z-1/4\rfloor+\lfloor x-1/4\rfloor+\lfloor y+z-3/4\rfloor-2\lfloor y+z\rfloor-\lfloor x+y+z\rfloor+3.
$$
Furthermore, we have 
$$
\mathcal{D}_{\mathbf{u},\mathbf{v}}^1=\big\{(x,y,z)\in[0,1)^3\,:\,x+y+z\geq 1/4\big\}.
$$
Let $(x,y,z)$ be in $\mathcal{D}_{\mathbf{u},\mathbf{v}}^1$. We have $\lfloor 2y+2z\rfloor-2\lfloor y+z\rfloor\geq 0$. If $x+y+z<1$ then we easily obtain that $\xi_1(x,y,z)\geq 1$. If $1\leq x+y+z<2$, then we have $x\geq 1/4$ or $y+z\geq 3/4$ which yields $\xi_1(x,y,z)\geq 1$ again. Finally, if $2\leq x+y+z<3$, then we have $\lfloor x+y+z-1/4\rfloor\geq 1$ and $\xi_1(x,y,z)\geq 1$ as expected. Hence, by Propositions \ref{theo Lucas Hypergeom} and \ref{prop: spe}, we obtain that $f(z)$ has the $p$-Lucas property for all primes $p\equiv 1\mod 4$.


\section{Algebraic independence of $G$-functions: a few examples}\label{sec: ex}

In this last section, we gather various examples of statements concerning algebraic independence of  
$G$-functions that follows from simple applications of our method. 

\subsection{Factorial ratios}\label{sec: exfr}

Given two tuples of vectors of natural numbers, $ e=(e_1,\dots, e_u)$ and $f=(f_1,\dots,f_v)$, 
we recall that the associated sequence of factorial ratios is defined by 
$$
\mathcal{Q}_{e,f}(n):=\frac{\prod_{i=1}^u(e_in)!}{\prod_{i=1}^v(f_in)!} 
$$
and that the generating series of such a sequence is denoted by 
$$
F_{e,f}(x):=\sum_{n=0}^{\infty}\mathcal{Q}_{e,f}(n)z^{n} \, .
$$
In the sequel, we always assume that the sets $\{e_1,\ldots,e_u\}$ and $\{f_1,\ldots,f_v\}$ are disjoint. We also recall that 
$\vert e\vert :=\sum_{i=1}^u e_i$, $\vert f\vert :=\sum_{i=1}^v f_i$, and 
$$
m_{e,f} := \left( \max \{e_1,\ldots,e_u,f_1,\ldots,f_v\}\right)^{-1} \, . 
$$
We set
$$
C_{e,f}:=\frac{\prod_{i=1}^ue_i^{e_i}}{\prod_{i=1}^vf_i^{f_i}} \, \cdot 
$$
From Stirling's formula, we deduce the following general asymptotics:  
\begin{equation}\label{Stirling}
\mathcal{Q}_{e,f}(n)\underset{n\rightarrow\infty}{\sim}C_{e,f}^n\big(\sqrt{2\pi n}\big)^{u-v}\sqrt{\frac{\prod_{i=1}^ue_i}{\prod_{i=1}^vf_i}} \, \cdot
\end{equation}

We give the following consequence of Proposition \ref{prop: sing1}.  

\begin{prop}\label{propo 2}
Let $s$ be a positive integer and, for all $i$ in $\{1,\dots,s\}$, let $e_i:=(e_{i,1},\ldots,e_{i,u_i})$ and $f_i:=(f_{i,1},\ldots,f_{i,v_i})$ be disjoint tuples of positive integers. 
Let us assume  that the following hold. 
\medskip

\begin{itemize}

\item[{\rm(i)}]  For all $i$ in $\{1,\dots,s\}$, one has $v_i-u_i\geq 2$.

\medskip

\item[{\rm(ii)}] The rational numbers $C_{e_1,f_1},\dots,C_{e_s,f_s}$ are pairwise distinct.

\medskip

\item[{\rm(iii)}]  For all $i$ in $\{1,\dots,s\}$, $\mathcal{Q}_{e_i,f_i}$ satisfies the $p$--Lucas property for all primes $p$. 
\end{itemize}

\medskip

Then $F_{e_1,f_1}(z),\dots,F_{e_s,f_s}(z)$ 
are algebraically independent over $\mathbb{C}(z)$.
\end{prop}

\begin{proof}
Since $v_i-u_i\geq 2$ for all $i$, we first infer from \eqref{Stirling} that $\mathcal{Q}_{e_i,f_i}(n)=O(C_{e_i,f_i}^n/n)$ and thus, by Remark \ref{rem: W},  we see that 
all $F_{e_i,f_i}(z)$ belong to $\mathfrak W$. 
We also infer from \eqref{Stirling} that  the radius of convergence of $F_{e_i,f_i}(z)$ is $1/C_{e_i,f_i}$.  Hence 
the functions $F_{e_1,f_1}(z),\dots,F_{e_s,f_s}(z)$   have distinct radius of convergence. Since by assumption they all belong  to $\mathfrak{ L}(\mathcal P)$, 
the result follows from Proposition \ref{prop: sing1}. 
\end{proof}

As an illustration, let us give the following result. 

\begin{thm} The functions 
$$
\sum_{n=0}^\infty\frac{(3n)!}{n!^3}z^n,\;\sum_{n=0}^\infty\frac{(5n)!(3n)!}{(2n)!^2n!^4}z^n\quad\textup{and}\quad\sum_{n=0}^\infty\frac{(10n)!}{(5n)!(3n)!n!^2}z^n
$$
are algebraically independent over $\mathbb{C}(z)$. 
\end{thm}

\begin{proof}
It is easy to see that (i) and (ii) are satisfied, while Proposition \ref{prop: fr} can be used to prove that (iii) holds too.  
\end{proof}

Using Proposition \ref{prop: sing2}, we can also obtain the following result.

\begin{prop}\label{prop: fr2}
Let $e_1$ and $f_1$, respectively $e_2$ and $f_2$, be disjoint tuples of positive integers such that the following hold. 

\medskip

\begin{itemize}

\item[{\rm(i)}]  $v_1-u_1= 2$.

\medskip

\item[{\rm(ii)}] $v_2-u_2\geq 3$.

\medskip

\item[{\rm(iii)}]  $\mathcal{Q}_{e_1,f_1}$ and $\mathcal{Q}_{e_2,f_2}$ satisfy the $p$--Lucas property for all primes $p$. 
\end{itemize}

\medskip

Then $F_{e_1,f_1}(z)$ and $F_{e_2,f_2}(z)$ are algebraically independent over $\mathbb{C}(z)$. 
\end{prop}

\begin{proof}
We first observe that if $C_{e_1,f_1}\not= C_{e_2,f_2}$, we can use 
Proposition \ref{propo 2} to conclude. We can thus assume that $C_{e_1,f_1}= C_{e_2,f_2}=:C$ and we will use Proposition  \ref{prop: sing2}. 

We first remark that $F_{e_1,f_1}(z)$ and $F_{e_2,f_2}(z)$ are both transcendental. This follows for instance from applying Proposition \ref{propo 2} twice with a single function. 
Now, by Pringsheim's theorem, $F_{e_1,f_1}(z)$ and $F_{e_2,f_2}(z)$ have a singularity at $1/C$.  Using \eqref{Stirling}, 
we infer from the assumption $v_1-u_1=2$ that $F_{e_1,f_1}$ satisfies Condition (i) of Proposition \ref{prop: sing2}, while we infer from the assumption $v_2-u_2\geq 3$ 
that  $F_{e_2,f_2}$ satisfies Condition (ii) of Proposition \ref{prop: sing2}.  
Since, by assumption, $F_{e_1,f_1}$ and $F_{e_2,f_2}$ both belong to $\mathfrak L(\mathcal P)$, we can apply Proposition \ref{prop: sing2} to conclude the proof. 
\end{proof}

\begin{rem}{\rm 
Using the discussion in \cite{Villegas}, one can actually show that the only case where 
$F_{e,f}(z)$ both belongs to $\mathfrak L(\mathcal P)$ and  is an algebraic function corresponds to 
$e=(2)$ and $f=(1,1)$, that is  to 
$$
F_{e,f}(z) = \frac{1}{\sqrt{1-z}} \, \cdot
$$}
\end{rem}

We give the following illustration of Proposition \ref{prop: fr2}. 

\begin{thm}The functions 
$$
\sum_{n=0}^\infty\frac{(4n)!}{(2n)!n!^2}z^n\quad\textup{and}\quad\sum_{n=0}^\infty\frac{(2n)!^3}{n!^6}z^n
$$
 are algebraically independent over $\mathbb{C}(z)$. 
\end{thm}

\begin{proof}
Here we have $e_1=(4)$ and $f_1=(2,1,1)$, so that $v_1-u_1=2$. Furthermore, for all $x$ in $[1/4,1)$, we have 
$$
\Delta_{e_1,f_1}(x)=\lfloor 4x\rfloor-\lfloor 2x\rfloor\geq 1\, , 
$$
which shows that $\mathcal{Q}_{e_1,f_1}$ satisfies the $p$-Lucas property for all primes $p$. 
On the other hand, we also have $e_2=(2,2,2)$ and $f_2=(1,1,1,1,1,1)$,  so that $v_2-u_2=3$. Furthermore,  for all $x$ in $[1/2,1)$, we have 
$$
\Delta_{e_2,f_2}(x)=3\lfloor 2x\rfloor\geq 1\, ,
$$
which shows that $\mathcal{Q}_{e_2,f_2}$ also satisfies the $p$-Lucas property for all primes $p$. Then the result follows from Proposition \ref{prop: fr2}.  
\end{proof}

\subsection{Generalized hypergeometric functions}\label{sec: exhyp}
Using Stirling formula, it is easy to give a general asymptotic for the coefficients of generalized hypergeometric functions. 
Indeed, it implies that 
$$
\Gamma(x)\underset{x\rightarrow\infty}{\sim}x^{x-\frac{1}{2}}e^{-x}\sqrt{2\pi} \,,
$$
and hence
$$
(\alpha)_n=\frac{\Gamma(\alpha+n)}{\Gamma(\alpha)}\underset{n\rightarrow\infty}{\sim}(\alpha+n)^{\alpha-\frac{1}{2}+n}e^{-\alpha-n}\frac{\sqrt{2\pi}}{\Gamma(\alpha)} \, \cdot
$$
Let us recall that 
$$
 \mathcal{Q}_{\boldsymbol{\alpha},\boldsymbol{\beta}}(n)
:=\frac{(\alpha_1)_n\cdots(\alpha_r)_n}{(\beta_1)_n\cdots(\beta_s)_n} \, 
\quad\mbox{and} \quad
F_{\boldsymbol{\alpha},\boldsymbol{\beta}}(x):=\sum_{n=0}^\infty\mathcal{Q}_{\boldsymbol{\alpha},\boldsymbol{\beta}}(n)x^n \, .
$$
When $r=s$, that is when $F_{\boldsymbol{\alpha},\boldsymbol{\beta}}(x)$ is a $G$-function, we thus obtain that 
$$\displaystyle
\mathcal{Q}_{\boldsymbol{\alpha},\boldsymbol{\beta}}(n) \underset{n\rightarrow\infty}{\sim}n^{\sum_{i=1}^r(\alpha_i-\beta_i)}\left(\frac{\prod_{i=1}^r
(\alpha_i+n)}{\prod_{j=1}^r(\beta_j+n)}\right)^ne^{\sum_{i=1}^r(\beta_i-\alpha_i)}\frac{\prod_{i=1}^r\Gamma(\alpha_i)}{\prod_{j=1}^r\Gamma(\beta_j)}
$$
which leads to the following simple asymptotics: 
\begin{equation}\label{asymhyp}
\displaystyle
\mathcal{Q}_{\boldsymbol{\alpha},\boldsymbol{\beta}}(n) \underset{n\rightarrow\infty}{\sim}\left(\frac{\prod_{i=1}^r
\Gamma(\alpha_i)}{\prod_{j=1}^r\Gamma(\beta_j)}\right) n^{\sum_{i=1}^r(\alpha_i-\beta_i)} \,.
\end{equation}

Note that it is usually easy to detect from such asymptotics when the hypergeometric function $F_{\boldsymbol{\alpha},\boldsymbol{\beta}}(z)$ 
is transcendental by comparison with asymptotics of coefficients of algebraic functions.  
Otherwise, one can always use the beautiful criterion of Beukers and Heckman \cite{BH}.  For instance, one can use our asymptotics to 
show that the hypergeometric function 
$$
\mathfrak f(z) := \sum_{n=0}^\infty\frac{(1/5)_n(4/5)_n}{n!^2}z^n
$$
is transcendental and belong to $\mathfrak W$ (see the proof of Theorem \ref{thm: indhyp}). Furthermore, Corollary \ref{cor: hyp} shows it has the $p$-Lucas property for all primes larger than $5$. 
We can thus 
apply Proposition~\ref{prop: sing1} to deduce the following result.

\begin{thm} The functions 
$\mathfrak f(z), \mathfrak f(2z), \mathfrak f(3z)\ldots$ are algebraically independent over $\mathbb C(z)$.  
\end{thm}

Let us give another kind of example derived from Proposition \ref{prop: sing2}.

\begin{thm}\label{thm: indhyp} The hypergeometric functions 
$$
\mathfrak{f}_1(z)=\sum_{n=0}^\infty\frac{(1/5)_n(4/5)_n}{n!^2}z^n\quad\textup{and}\quad\mathfrak{f}_2(z)=\sum_{n=0}^\infty\frac{(1/3)_n(1/2)_n^2}{(2/3)_nn!^2}z^n
$$
are algebraically independent over $\mathbb{C}(z)$. 
\end{thm}

\begin{proof}
For $\boldsymbol{\alpha}=(1/5,4/5)$ and $\boldsymbol{\beta}=(1,1)$, we infer from Equation \eqref{asymhyp}  that
\begin{equation}\label{eq1}
\mathcal{Q}_{\boldsymbol{\alpha},\boldsymbol{\beta}}(n)\underset{n\rightarrow\infty}{\sim}\frac{\Gamma(1/5)\Gamma(4/5)}{n}
\end{equation}
which is the asymptotic of a transcendental series. For $\boldsymbol{\alpha}=(1/3,1/2,1/2)$ and $\boldsymbol{\beta}=(2/3,1,1)$, we have 
\begin{equation}\label{eq2}
\mathcal{Q}_{\boldsymbol{\alpha},\boldsymbol{\beta}}(n)\underset{n\rightarrow\infty}{\sim}\frac{\Gamma(1/3)\Gamma(1/2)^2}{\Gamma(2/3)}\frac{1}{n^{4/3}} \,,
\end{equation}
so that $\mathfrak{f}_2$ belongs to $\mathfrak{W}$. By Corollary \ref{cor: hyp}, we first get that $\mathfrak{f}_1$ belongs to $\mathfrak L(\mathcal S_0)$, where $\mathcal S_0$ is the 
set of primes larger than $5$,   
while Proposition \ref{theo Lucas Hypergeom}  
implies that $\mathfrak{f}_2$ belongs to $\mathfrak L(\mathcal S_1)$, 
where $\mathcal S_1 = \{p\in \mathcal P \,:\, p \equiv 1 \mod 6\}$. In particular, both belong to  
$\mathfrak L(\mathcal S_1)$. The series $\mathfrak{f}_2$ belongs to $\mathfrak{W}\cap\mathfrak{L}(\mathcal{S}_1)$ so is transcendental over $\mathbb{C}(z)$ by Proposition \ref{prop: sing1}. 
Note that these functions are hypergeometric and thus have 
 the same radius of convergence $1$. Then \eqref{eq1} and \eqref{eq2} show, as earlier, that one can apply Proposition \ref{prop: sing2},  
 which ends the proof. 
\end{proof}

\subsection{Sums and products of binomials}\label{sec: exspe} 
We give below an application of Proposition~\ref{prop: sing1} to generating series associated 
with Ap\'ery numbers, Franel numbers, and some of their generalizations. 

\begin{thm}\label{thm: appli1}
Let $\mathcal F$ be the set formed by the union of the three following sets: 
$$
\left\{\sum_{n=0}^\infty\left(\sum_{k=0}^n\binom{n}{k}^r\right)z^n\,: \,r\geq 3\right\},\;\left\{\sum_{n=0}^\infty
\left(\sum_{k=0}^n\binom{n}{k}^r\binom{n+k}{k}^r\right)z^n\,: \,r\geq 2\right\}
$$
and
$$
\left\{\sum_{n=0}^\infty\left(\sum_{k=0}^n\binom{n}{k}^{2r}\binom{n+k}{k}^r\right)z^n\,: \,r\geq 1\right\} \, .
$$
Then all elements of $\mathcal F$ are algebraically independent over $\mathbb C(z)$.
\end{thm}


\begin{proof}
McIntosh \cite{McIntosh} proves general asymptotics for sequences of the form 
$$
S(n):=\sum_{k=0}^n\prod_{j=0}^m\binom{n+jk}{k}^{r_j},
$$
where $m,r_0,\dots,r_m$ are natural integers. 
Indeed, he shows that 
$$
S(n)\underset{n\rightarrow\infty}{\sim}\frac{\mu^{n+1/2}}{\sqrt{\nu(2\pi\lambda n)^{r-1}}} \, ,
$$
with $r=r_0+\cdots+r_m$,  and where $\lambda$, $0<\lambda<1$, is defined by
$$
\prod_{j=0}^m\left(\frac{(1+j\lambda)^j}{\lambda(1+j\lambda-\lambda)^{j-1}}\right)^{r_j}=1\, , 
$$
and where $\mu$ and $\nu$ are respectively defined by 
$$
\mu=\prod_{j=0}^m\left(\frac{1+j\lambda}{1+j\lambda-\lambda}\right)^{r_j}
$$
and
$$
\nu=\sum_{j=0}^m\frac{r_j}{(1+j\lambda-\lambda)(1+j\lambda)} \, \cdot
$$ 
The particular case we are interested in is considered in \cite{McIntosh}. For a positive integer $r$, we then obtain the following asymptotics: 
$$
\sum_{k=0}^n\binom{n}{k}^r\underset{n\rightarrow\infty}{\sim}\frac{2^{rn}}{\sqrt{r(\pi n/2)^{r-1}}} \, ,
$$
$$
\sum_{k=0}^n\binom{n}{k}^r\binom{n+k}{k}^r\underset{n\rightarrow\infty}{\sim}\frac{(1+\sqrt{2})^{2nr+r}}{\sqrt{4r(\pi n\sqrt{2})^{2r-1}}}\, ,
$$
and 
$$
\sum_{k=0}^n\binom{n}{k}^{2r}\binom{n+k}{k}^r\underset{n\rightarrow\infty}{\sim}\frac{\big((1+\sqrt{5})/2\big)^{5nr+4r}}{\sqrt{(5+2\sqrt{5})r(2\pi n)^{3r-1}}} \, \cdot
$$
These asymptotics show that all  functions in $\mathcal F$ have distinct radius of convergence. 
Furthermore, we infer from Remark \ref{rem: W} and these asymptotics that they all belong to $\mathfrak W$ since $r\geq 3$ for the first family, 
$r\geq 2$ for the second, and $r\geq 1$ for the third.  
On the other hand, we already proved in Section \ref{sec: spe} that all functions in $\mathcal F$  belong to $\mathfrak L(\mathcal P)$.  
The result thus follows directly from Proposition \ref{prop: sing1}. 
\end{proof}

\subsection{A mixed example} One special feature of our approach is that one can easily mix functions of rather different type without having to 
consider all their derivatives and finding a common differential equation for them.   
We illustrate this claim with the following simple example. 

\begin{thm}\label{thm: prop3}
The functions 
$$
f(z) :=\sum_{n=0}^\infty\frac{(4n)!}{(2n)!n!^2}z^n, \; g(z):=\sum_{n=0}^\infty\left(\sum_{k=0}^n\binom{n}{k}^2\binom{n+k}{k}^2\right)z^n, \; 
h(z) := \sum_{n=0}^\infty\frac{(1/6)_n(1/2)_n}{(2/3)_nn!}z^n, 
$$
and
$$
i(z):=\sum_{n=0}^\infty\frac{(1/5)_n^3}{(2/7)_nn!^2}z^n 
$$
are algebraically independent over $\mathbb{C}(z)$.  
\end{thm}

Note that the two last functions are not globally bounded so they cannot be obtained as  the diagonal of some rational functions.  

\begin{proof} 
On the one hand, we already shown in Section \ref{sec: fr} that $f(z)$ belongs to  $\mathfrak L(\mathcal P)$,  
in Section \ref{sec: Allouche} that $g(z)$ belongs to $\mathfrak L(\mathcal P)$, in Section 
\ref{sec: hyp} that $h(z)$ belongs to $\mathfrak L(\mathcal S)$, 
where $\mathcal S = \{p\in \mathcal P \,:\, p \equiv 1 \mod 6\}$, and it follows from Proposition \ref{theo Lucas Hypergeom} that $i(z)$ 
belongs to  $\mathfrak L(\mathcal S')$, where  $\mathcal S' = \{p\in \mathcal P \,:\, p \equiv 1 \mod 35\}$.   
 Hence all belong to $\mathfrak L(\mathcal S_1)$, where $\mathcal S_1 = \{p\in \mathcal P \,:\, p \equiv 1 \mod 210\}$.  
 
On the other hand, we infer from Remark \ref{rem: W} and asymptotics for the coefficients of these functions (see Sections \ref{sec: exfr},  \ref{sec: exhyp}, and \ref{sec: exspe}) 
 that they all belong to $\mathfrak W$ and that $\rho_f=4^{-3}$, $\rho_g:= (1+\sqrt 2)^{-4}$, $\rho_h=1$, and $\rho_i=1$.  

Now, let us assume that $f,g,h,i$ are algebraically dependent over $\mathbb C(z)$. Then by Theorem~\ref{thm: ind}, there should exist integers $a,b,c,d$, not all zero, such that 
$$
f(z)^{a}g(z)^bh(z)^ci(z)^d= r(z) \, ,
$$
where $r(z)$ is a rational fraction. If $a\not=0$, we infer from the equality 
$$
f(z)^a = r(z) g(z)^{-b}h(z)^{-c}i(z)^{-d} 
$$
that $f(z)^a$ is meromorphic at $\rho_f$, which provides a contradiction with the fact that $f(z)$ belongs to $\mathfrak W$. Thus $a$ equals $0$. 
Now, if $b\not =0$, we obtain in the same way that $g(z)^b$ is meromorphic at $\rho_g$, which provides a contradiction with the fact that $g(z)$ belongs to $\mathfrak W$. Thus $b=0$ 
and we have
$$
h(z)^{c}i(z)^d= r(z) \, ,
$$
with $c$ and $d$ not all zero.   
This means that $h$ and $i$ are algebraically dependent. However, using the asymptotics of Section \ref{sec: exhyp}, we see that 
$h$ and $i$ satisfy the assumption of Proposition \ref{prop: sing2} and are thus algebraically independent. 
We thus get a contradiction, concluding the proof. 
\end{proof}


\end{document}